\documentclass[11pt]{article}

\title{Lecture notes on embedded contact homology}
\author{Michael Hutchings\footnote{Partially supported by NSF grant DMS-1105820.}}
\date{}

\usepackage{amssymb}
\usepackage{latexsym}
\usepackage{amsmath}
\usepackage{amsthm}
\usepackage{amscd}

\newcommand{\mc}[1]{{\mathcal #1}}

\numberwithin{equation}{section}

\newtheorem{theorem}{Theorem}[section]
\newtheorem{proposition}[theorem]{Proposition}

\newtheorem{lemma}[theorem]{Lemma}
\newtheorem{lemma-definition}[theorem]{Lemma-Definition}

\newtheorem*{intpos}{Intersection Positivity}
\newtheorem*{adjfor}{Adjunction Formula}
\newtheorem*{gcc}{Gromov Compactness via Currents}
\newtheorem*{raf}{Relative Adjunction Formula}
\newtheorem*{bpechi}{Basic Properties of the ECH Index}
\newtheorem*{indexinequality}{Index Inequality}
\newtheorem*{writhebound}{Writhe Bound}
\newtheorem*{partitionconditions}{Partition Conditions}

\theoremstyle{definition}
\newtheorem{definition}[theorem]{Definition}
\newtheorem{remark}[theorem]{Remark}
\newtheorem{remarks}[theorem]{Remarks}
\newtheorem{example}[theorem]{Example}

\newtheorem{exercise}[theorem]{Exercise}

\newcommand{\floor}[1]{\left\lfloor #1 \right\rfloor}
\newcommand{\ceil}[1]{\left\lceil #1 \right\rceil}

\renewcommand{\frak}{\mathfrak}

\newcommand{\C}{{\mathbb C}}
\newcommand{\Q}{{\mathbb Q}}
\newcommand{\R}{{\mathbb R}}
\newcommand{\N}{{\mathbb N}}
\newcommand{\Z}{{\mathbb Z}}

\newcommand{\op}{\operatorname}
\newcommand{\dbar}{\overline{\partial}}
\newcommand{\zbar}{\overline{z}}

\newcommand{\Spinc}{\op{Spin}^c}

\newcommand{\Spin}{\op{Spin}}

\newcommand{\Ker}{\op{Ker}}

\newcommand{\tensor}{\otimes}

\newcommand{\vu}{\nu}

\newcommand{\rb}[1]{\raisebox{1.5ex}[-1.5ex]{#1}}

\newcommand{\bpm}{\begin{pmatrix}}
\newcommand{\epm}{\end{pmatrix}}

\newcommand{\ind}{\op{ind}}

\renewcommand{\epsilon}{\varepsilon}

\begin{document}

\setcounter{tocdepth}{2}

\maketitle

\begin{abstract}
  These notes give an introduction to embedded contact homology
  (ECH) of contact three-manifolds, gathering together many basic
  notions which are scattered across a number of papers. We also
  discuss the origins of ECH, including
  various remarks and examples which have not been previously
  published.  Finally, we review the recent application to
  four-dimensional symplectic embedding problems.  This article is
  based on lectures given in Budapest and Munich in the summer of
  2012, a series of accompanying blog postings at
  \verb=floerhomology.wordpress.com=, and related lectures at UC
  Berkeley in Fall 2012.  There is already a brief introduction to ECH
  in the article \cite{icm}, but the present notes give much more
  background and detail.
\end{abstract}

\tableofcontents

\section{Introduction}

We begin by describing an application of ECH to four-dimensional symplectic embedding problems. We will then give an overview of the basic structure of ECH and how it leads to the application.

\subsection{Symplectic embeddings in four dimensions}
\label{sec:embedding}

Let $(X_0,\omega_0)$ and $(X_1,\omega_1)$ be symplectic four-manifolds, possibly with boundary or corners. A {\em symplectic embedding\/} of $(X_0,\omega_0)$ into $(X_1,\omega_1)$ is a smooth embedding $\varphi:X_0\to X_1$ such that $\varphi^*\omega_1=\omega_0$. It is interesting to ask when such a symplectic embedding exists.

This is a nontrivial question already for domains in $\R^4$.
For example, given $a,b>0$, define the {\em ellipsoid\/}
\begin{equation}
\label{eqn:ellipsoid}
E(a,b) = \left\{(z_1,z_2)\in\C^2\;\bigg|\; \frac{\pi|z_1|^2}{a}+\frac{\pi|z_2|^2}{b}\le 1\right\}.
\end{equation}
Here we identify $\C^2=\R^4$ with coordinates $z_k=x_k+y_k$ for $k=1,2$, with the standard symplectic form $\omega=\sum_{k=1}^2dx_kdy_k$. In particular, define the {\em ball\/} $B(a)=E(a,a)$. Also, define the {\em polydisk}
\begin{equation}
\label{eqn:polydisk}
P(a,b) = \left\{(z_1,z_2)\in\C^2\;\big|\; \pi|z_1|^2\le a,\; \pi|z_2|^2\le b\right\}.
\end{equation}
We can now ask, when does one ellipsoid or polydisk symplectically embed into another?

A landmark in the theory of symplectic embeddings is Gromov's nonsqueezing theorem from 1985. The four-dimensional case of this theorem asserts that $B(r)$ symplectically embeds into $P(R,\infty)$ if and only if $r\le R$.

The question of when one four-dimensional ellipsoid symplectically embeds into another was answered only in 2010, by McDuff. To state the embedding criterion, let $N(a,b)$ denote the sequence of all nonnegative integer linear combinations of $a$ and $b$, arranged in nondecreasing order, and indexed starting at $0$. For example,
\begin{equation}
\label{eqn:N11}
N(1,1) = (0,1,1,2,2,2,\ldots)
\end{equation}
and
\begin{equation}
\label{eqn:N12}
N(1,2) = (0,1,2,2,3,3,4,4,4,5,5,5,\ldots).
\end{equation}

\begin{theorem}[McDuff \cite{m}]
\label{thm:m}
There is a symplectic embedding $\op{int}(E(a,b))\to E(c,d)$ if and only if $N(a,b)\le N(c,d)$, i.e.\ $N(a,b)_k\le N(c,d)_k$ for each $k \ge 0$.
\end{theorem}

For example, it is not hard to deduce from Theorem~\ref{thm:m}, together with \eqref{eqn:N11} and \eqref{eqn:N12}, that 
 $\op{int}(E(1,2))$ symplectically embeds into $B(c)$ if and only if $c\ge 2$.

Given more general $a,b,c,d$, it can be nontrivial to decide whether $N(a,b)\le N(c,d)$. For example, consider the problem of an embedding an ellipsoid into a ball, i.e.\ the case $c=d$. By scaling, we can encode this problem into a single function $f:[1,\infty)\to[1,\infty)$, where $f(a)$ is defined to be the infimum over $c$ such that $E(1,a)$ symplectically embeds into $B(c)=E(c,c)$. 

In general, if there is a symplectic embedding of $(X_0,\omega_0)$ into $(X_1,\omega_1)$, then necessarily
\begin{equation}
\label{eqn:volumeconstraint}
\op{vol}(X_0,\omega_0)\le \op{vol}(X_1,\omega_1),
\end{equation}
where in four dimensions
\[
\op{vol}(X,\omega) = \frac{1}{2}\int_X\omega\wedge\omega.
\]
In particular, the ellipsoid has volume $\op{vol}(E(a,b))=ab/2$, cf.\ equation \eqref{eqn:toricvolume},  so it follows from the volume constraint \eqref{eqn:volumeconstraint} that $f(a)\ge\sqrt{a}$.

McDuff-Schlenk computed the function $f$ explicitly and found that the volume constraint is the only constraint if $a$ is sufficiently large, while for smaller $a$ the situation is more interesting. In particular, their calculation implies the following\footnote{An analogue of Theorem~\ref{thm:ms} for symplectically embedding $\op{int}(E(1,a))$ into $P(c,c)$  was recently worked out in \cite{fm}. This is equivalent to symplectically embedding $\op{int}(E(1,a))$ into $E(c,2c)$, by Remark~\ref{rem:sharp}(b) and equation \eqref{eqn:p11e12} below.}:

\begin{theorem}[McDuff-Schlenk \cite{ms}]
\label{thm:ms}
\begin{itemize}
\item
On the interval $\left[1,\left(1+\sqrt{5}/2\right)^4\right)$, the function $f$ is piecewise linear, given by a ``Fibonacci staircase''.
\item
The interval $\left[\left(1+\sqrt{5}/2\right)^4,(17/6)^2\right]$ is divided into finitely many intervals, on each of which either $f$ is linear or $f(a)=\sqrt{a}$.
\item
On the interval $\left[(17/6)^2,\infty\right)$, we have $f(a)=\sqrt{a}$.
\end{itemize}
\end{theorem}

Note that Theorems~\ref{thm:m} and \ref{thm:ms} were proved by different methods. It is a subtle number-theoretic problem to deduce Theorem~\ref{thm:ms} directly from Theorem~\ref{thm:m}.

\subsection{Properties of ECH capacities}
\label{sec:capacities}

Embedded contact homology can be used to prove the obstruction half of Theorem~\ref{thm:m}, namely the fact that if $\op{int}(E(a,b))$ symplectically embeds into $E(c,d)$ then $N(a,b)\le N(c,d)$. This follows from the more general theory of ``ECH capacities''. Here are some of the key properties of ECH capacities; the definition of ECH capacities will be given in \S\ref{sec:definecapacities}.

\begin{theorem}\cite{qech}
\label{thm:ECHcapacities}
For each symplectic four-manifold $(X,\omega)$ (not necessarily connected, possibly with boundary or corners), there is a sequence of real numbers
\[
0 = c_0(X,\omega) \le c_1(X,\omega) \le \cdots \le \infty,
\]
called {\em ECH capacities\/}, with the following properties:
\begin{description}
\item{(Monotonicity)}
If $(X_0,\omega_0)$ symplectically embeds into $(X_1,\omega_1)$, then
\begin{equation}
\label{eqn:monotonicity}
c_k(X_0,\omega_0)\le c_k(X_1,\omega_1)
\end{equation}
for all $k\ge 0$.
\item{(Conformality)}
If $r$ is a nonzero real number, then
\[
c_k(X,r\omega)=|r|c_k(X,\omega).
\]
\item{(Ellipsoid)}
\begin{equation}
\label{eqn:ellipsoidcapacities}
c_k(E(a,b))=N(a,b)_k.
\end{equation}
\item{(Polydisk)}
\begin{equation}
\label{eqn:polydiskcapacities}
c_k(P(a,b))=\min\left\{am+bn\;\big|\; m,n\in\N,\; (m+1)(n+1)\ge k+1\right\}.
\end{equation}
\item{(Disjoint union)}
\[
c_k\left(\coprod_{i=1}^n(X_i,\omega_i)\right) = \max_{k_1+\cdots+k_n=k}\sum_{i=1}^nc_{k_i}(X_i,\omega_i).
\]
\item{(Volume)} \cite{vc} If $(X,\omega)$ is a Liouville domain (see Definition~\ref{def:LD}) with all ECH capacities finite (for example a star-shaped domain in $\R^4$), then
\begin{equation}
\label{eqn:volume}
\lim_{k\to\infty}\frac{c_k(X,\omega)^2}{k} = 4\op{vol}(X,\omega).
\end{equation}
\end{description}
\end{theorem}

In particular, the Monotonicity and Ellipsoid properties immediately imply the obstruction half of Theorem~\ref{thm:m}. Theorem~\ref{thm:ECHcapacities} does not say anything about the other half of Theorem~\ref{thm:m}, namely the existence of symplectic embeddings.

The Volume property says that for Liouville domains with all ECH capacities finite, the asymptotic behavior of the Monotonicity property \eqref{eqn:monotonicity} as $k\to\infty$ recovers the volume constraint \eqref{eqn:volumeconstraint}.

\begin{exercise}
\label{ex:ellvol}
Check the volume property \eqref{eqn:volume} when $(X,\omega)$ is an ellipsoid $E(a,b)$. (See answer in \S\ref{sec:answers}.)
\end{exercise}

\begin{remark}
\label{rem:sharp}
Here is what we know about the sharpness of the ECH obstruction for some other symplectic embedding problems.
\begin{description}
\item{(a)} ECH capacities give a sharp obstruction to symplectically embedding a disjoint union of balls of possibly different sizes into a ball. This follows by comparison with work of McDuff \cite{msdi} and Biran \cite{biran} from the 1990's which solved this embedding problem. See \cite{pnas} for details.
\item{(b)} It follows from work of M\"uller that ECH capacities give a sharp obstruction to embedding an ellipsoid into a polydisk, see \cite{pnas} and \cite{fm}.
\item{(c)}
ECH capacities do not give a sharp obstruction to symplectically embedding a polydisk into an ellipsoid. For example, one can check that
\begin{equation}
\label{eqn:p11e12}
c_k(P(1,1))=c_k(E(1,2))
\end{equation}
for all $k$, so ECH capacities give no obstruction to symplectically embedding $P(1,1)$ into $E(a,2a)$ when $a>1$. However the Ekeland-Hofer capacities imply that $P(1,1)$ does not symplectically embed into $E(a,2a)$ when $a<3/2$; these capacities are $(1,2,3,\ldots)$ and $(a,2a,2a,3a,4a,4a,\ldots)$ respectively \cite{eh,chls}. The Ekeland-Hofer obstruction is sharp, because it follows from \eqref{eqn:ellipsoid} and \eqref{eqn:polydisk} that $P(1,1)$, as defined, is a subset of $E(3/2,3)$. 
\item{(d)}
We know very little about when one polydisk can be symplectically embedded into another or how good the ECH obstruction to this is.
\end{description}
\end{remark}

In \S\ref{sec:toric} we will compute the ECH capacities of a larger family of examples, namely ``toric domains'' in $\C^2$.

\subsection{Overview of ECH}
\label{sec:overview}

We now outline the definition of embedded contact homology; details will be given in \S\ref{sec:defech}.

Let $Y$ be a closed oriented three-manifold. Recall that a {\em contact form\/} on $Y$ is a $1$-form $\lambda$ on $Y$ such that $\lambda\wedge d\lambda>0$ everywhere. The contact form $\lambda$ determines the {\em contact structure\/} $\xi=\Ker\lambda$, which is an oriented two-plane field, and the {\em Reeb vector field\/} $R$ characterized by $d\lambda(R,\cdot)=0$ and $\lambda(R)=1$.

A {\em Reeb orbit\/} is a closed orbit of $R$, i.e.\ a map $\gamma:\R/T\Z\to Y$ for some $T>0$, modulo reparametrization, such that $\gamma'(t)=R(\gamma(t))$. A Reeb orbit is either embedded in $Y$, or an $m$-fold cover of an embedded Reeb orbit for some integer $m>1$.

We often want to assume that the Reeb orbits are ``cut out transversely'' in the following sense. Given a Reeb orbit $\gamma$ as above, the {\em linearized return map\/} is a symplectic automorphism $P_\gamma$ of the symplectic vector space $(\xi_{\gamma(0)},d\lambda)$, which is defined as the derivative of the time $T$ flow of $R$. The Reeb orbit $\gamma$ is called {\em nondegenerate\/} if $1$ is not an eigenvalue of $P_\gamma$. The contact form $\lambda$ is called nondegenerate if all Reeb orbits are nondegenerate. This holds for generic contact forms. 

A nondegenerate Reeb orbit $\gamma$ is called {\em elliptic\/} if the eigenvalues of $P_\gamma$ are on the unit circle, so that $P_\gamma$ is conjugate to a rotation. Otherwise $\gamma$ is {\em hyperbolic\/}, meaning that the eigenvalues of $P_\gamma$ are real. There are two kinds of hyperbolic orbits: {\em positive hyperbolic\/} orbits for which the eigenvalues of $P_\gamma$ are positive, and {\em negative hyperbolic\/} orbits for which the eigenvalues of $P_\gamma$ are negative.

Assume now that $\lambda$ is nondegenerate, and fix a homology class $\Gamma\in H_1(Y)$. One can then define the {\em embedded contact homology\/} $ECH_*(Y,\xi,\Gamma)$ as follows. This is the homology of a chain complex $ECC(Y,\lambda,\Gamma,J)$. The chain complex is freely generated over $\Z/2$ by finite sets of pairs $\alpha=\{(\alpha_i,m_i)\}$ where:
\begin{itemize}
\item The $\alpha_i$ are distinct embedded Reeb orbits.
\item The $m_i$ are positive integers.
\item The total homology class $\sum_im_i[\alpha_i]=\Gamma$.
\item $m_i=1$ whenever $\alpha_i$ is hyperbolic.
\end{itemize}
It is a frequently asked question why the last condition is necessary; we will give one answer in \S\ref{sec:mappingtorusexample}--\ref{sec:gf} and another answer in \S\ref{sec:200pages}. Note also that ECH can be defined with integer coefficients, see \cite[\S9]{obg2}; however the details of the signs are beyond the scope of these notes, and $\Z/2$ coefficients are sufficient for all the applications we will consider here.

The chain complex differential is defined roughly as follows.
We call an almost complex structure $J$ on the ``symplectization'' $\R\times Y$ {\em symplectization-admissible\/} if $J$ is $\R$-invariant, 
 $J(\partial_s)=R$ where $s$ denotes the $\R$ coordinate on $\R\times Y$, and $J$ sends the contact structure $\xi$ to itself, rotating positively with respect to $d\lambda$. These are the standard conditions on $J$ for defining various flavors of contact homology. In the notation for the chain complex, $J$ is a generic symplectization-admissible almost complex structure on $\R\times Y$.

If $\alpha=\{(\alpha_i,m_i)\}$ and $\beta=\{(\beta_j,n_j)\}$ are chain complex generators, then the differential coefficient $\langle\partial\alpha,\beta\rangle\in\Z/2$ is a mod 2 count of $J$-holomorphic curves $C$ in $\R\times Y$, modulo $\R$ translation and equivalence of currents, satisfying two conditions. The first condition is that, roughly speaking, $C$ converges as a current to $\sum_im_i\alpha_i$ as $s\to+\infty$, and to $\sum_jn_j\beta_j$ as $s\to-\infty$. The second condition is that $C$ has ``ECH index'' equal to $1$. The definition of the ECH index is the key nontrivial part of the definition of ECH; the original references are \cite{pfh2,ir}, and we will spend considerable time explaining this in \S\ref{sec:defech}.  We will see in Proposition~\ref{prop:I03} that our assumption that $J$ is generic implies every ECH index $1$ curve is embedded, except possibly for multiple covers of ``trivial cylinders'' $\R\times\gamma$ where $\gamma$ is a Reeb orbit; hence the name ``embedded contact homology''.
We will explain in \S\ref{sec:differentialdefined} why $\partial$ is well-defined.
It is shown in \cite[\S7]{obg1} that $\partial^2=0$; we will introduce some of what is involved in the proof in \S\ref{sec:200pages}.

Let $ECH_*(Y,\lambda,\Gamma,J)$ denote the homology of the chain complex $ECC_*(Y,\lambda,\Gamma,J)$. It turns out that this homology does not depend on the almost complex structure $J$ or on the contact form $\lambda$ for $\xi$, and so defines a well-defined $\Z/2$-module $ECH_*(Y,\xi,\Gamma)$.
 In principle one should be able to prove this by counting holomorphic curves with ECH index zero, but there are unsolved technical problems with this approach which we will describe in \S\ref{sec:cobmap}. Currently the only way to prove the above invariance is using:

\begin{theorem}[Taubes \cite{e1}]
\label{thm:echswf}
If $Y$ is connected, then there is a canonical isomorphism of relatively graded modules (with $\Z/2$ or $\Z$ coefficients)
\begin{equation}
\label{eqn:echswf}
ECH_*(Y,\lambda,\Gamma,J) = \widehat{HM}^{-*}(Y,\frak{s}_{\xi}+\op{PD}(\Gamma)).
\end{equation}
\end{theorem}

\noindent
Here $\widehat{HM}^*$ denotes the ``from'' version of Seiberg-Witten Floer cohomology defined by Kronheimer-Mrowka \cite{km}, and $\frak{s}_{\xi}$ denotes a spin-c structure determined by the oriented 2-plane field $\xi$, see \S\ref{sec:SW3}. The relative grading is explained in \S\ref{sec:differential}.
Kutluhan-Lee-Taubes \cite{klt1} and Colin-Ghiggini-Honda \cite{cgh} also showed that both sides of \eqref{eqn:echswf} are isomorphic
to the Heegaard Floer homology $HF^+(-Y,\frak{s}_\xi+\op{PD}(\Gamma))$ defined in \cite{os}.
The upshot is that ECH is a topological invariant of $Y$, except that one needs to shift $\Gamma$ if one changes 
the contact structure.

\begin{remark}
\label{rem:ag}
In fact, both Seiberg-Witten Floer cohomology and ECH have absolute gradings by homotopy classes of oriented two-plane fields \cite{km,ir}, and Taubes's isomorphism \eqref{eqn:echswf} respects these absolute gradings \cite{cg}. Thus one can write the isomorphism \eqref{eqn:echswf} as $ECH_{\mathfrak{p}}(Y,\lambda,J)=\widehat{HM}^{\mathfrak{p}}(Y)$ where $\mathfrak{p}$ denotes a homotopy class of oriented two-plane fields on $Y$.
\end{remark}

Although ECH does not depend on the contact form, because it is defined using the contact form it has applications to contact geometry. For example, Theorem~\ref{thm:echswf}, together with known properties of Seiberg-Witten Floer cohomology, implies the three-dimensional {\em Weinstein conjecture\/}: every contact form on a closed connected three-manifold has at least one Reeb orbit.
Indeed, Taubes's proof of the Weinstein conjecture in \cite{tw1} can be regarded as a first step towards proving Theorem~\ref{thm:echswf}.

The reason that Theorem~\ref{thm:echswf} implies the Weinstein conjecture is that if there is no closed orbit, then $\lambda$ is nondegenerate and
\[
ECH_*(Y,\xi,\Gamma) = \left\{\begin{array}{cl} \Z/2, & \Gamma=0,\\
0, & \Gamma\neq 0.
\end{array}
\right.
\]
Here the $\Z/2$ comes from the empty set of Reeb orbits, which is a legitimate chain complex generator when $\Gamma=0$.  However results of Kronheimer-Mrowka \cite{km} imply that if $c_1(\xi)+2\op{PD}(\Gamma)\in H^2(Y;\Z)$ is torsion (and by a little algebraic topology one can always find a class $\Gamma\in H_1(Y)$ with this property), then $\widehat{HM}^{*}(Y,\frak{s}_\xi+\Gamma)$ is infinitely generated, which is a contradiction.

Note that although $ECH(Y,\xi,\Gamma)$ is infinitely generated for $\Gamma$ as above, there might not exist infinitely many embedded Reeb orbits. To give a counterexample, first recall that in any symplectic manifold $(M,\omega)$, a {\em Liouville vector field\/} is a vector field $\rho$ such that $\mc{L}_\rho\omega=\omega$. A hypersurface $Y\subset M$ is of {\em contact type\/} if there exists a Liouville vector field $\rho$ transverse to $Y$ defined in a neighborhood of $Y$. In this case the ``Liouville form'' $\imath_\rho\omega$ restricts to a contact form on $Y$, whose Reeb vector field is parallel to the Hamilonian vector field $X_H$ where $H:M\to\R$ is any smooth function having $Y$ as a regular level set.

For example, the radial vector field
\[
\rho = \frac{1}{2}\sum_{k=1}^2\left(x_k\frac{\partial}{\partial x_j}
+y_k\frac{\partial}{\partial y_k}\right)
\]
is a Liouville vector field defined on all of $\R^4$. It follows that if $Y$ is a hypersurface in $\R^4$ which is ``star-shaped'', meaning transverse to the radial vector field $\rho$, then the Liouville form
\begin{equation}
\label{eqn:lambdastd}
\lambda=\frac{1}{2}\sum_{k=1}^2\left(x_kdy_k - y_kdx_k\right)
\end{equation}
restricts to a contact form on $Y$, with Reeb vector field determined as above.

\begin{example}
\label{ex:ell1}
If $Y=\partial E(a,b)$ is the boundary of an ellipsoid, then it follows from the above discussion that the Liouville form $\lambda$ in \eqref{eqn:lambdastd} restricts to a contact form on $Y$, whose Reeb vector field is given in polar coordinates by
\[
R = \frac{2\pi}{a}\frac{\partial}{\partial \theta_1} + \frac{2\pi}{b}\frac{\partial}{\partial \theta_2}.
\]
If $a/b$ is irrational, then there are just two embedded Reeb orbits, which we denote by $\gamma_1=(z_2=0)$ and $\gamma_2=(z_1=0)$. The linearized return map $P_{\gamma_1}$ is rotation by $2\pi a/b$, and the linearized return map $P_{\gamma_2}$ is rotation by $2\pi b/a$, so both of these Reeb orbits are elliptic. A generator of the ECH chain complex then has the form $\gamma_1^{m_1}\gamma_2^{m_2}$, where this notation indicates the set consisting of the pair $(\gamma_1,m_1)$ (if $m_1\neq 0$) and the pair $(\gamma_2,m_2)$ (if $m_2\neq 0$). For grading reasons to be explained in \S\ref{sec:ellipsoid}, the differential $\partial$ is identically zero. Thus $ECH(\partial E(a,b),\lambda,0)$ has one generator for each pair of nonnegative integers.
\end{example}

By making stronger use of the isomorphism \eqref{eqn:echswf}, one can prove some slight refinements of the Weinstein conjecture. For example, there are always at least two embedded Reeb orbits \cite{gh}; and if $\lambda$ is nondegenerate and $Y$ is not a sphere or a lens space then there at least three embedded Reeb orbits \cite{wh}. To put this in perspective, Colin-Honda \cite{ch} used linearized contact homology to show that for many contact structures, every contact form has infinitely many embedded Reeb orbits. The only examples of closed contact three-manifolds we know of with only finitely many embedded Reeb orbits are the ellipsoid examples in Example~\ref{ex:ell1}, and quotients of these on lens spaces, with exactly two embedded Reeb orbits.

{\em Historical note.\/}
The original motivation for the definition of ECH was to find a symplectic model for Seiberg-Witten Floer homology, so that an isomorphism of the form \eqref{eqn:echswf} would hold\footnote{More precisely, we first defined an analogous theory for mapping tori of symplectomorphisms of surfaces, called {\em periodic Floer homology\/}, and conjectured that this was isomorphic to Seiberg-Witten Floer homology, see \cite[\S1.1]{pfh2}. This conjecture was later proved by Lee and Taubes \cite{lee-taubes}. 
Initially it was not clear if ECH would also be isomorphic to Seiberg-Witten Floer homology because the geometry of contact manifolds is slightly different than that of mapping tori.
However the calculation of the ECH of $T^3$ then provided nontrivial evidence that this is the case, see \cite[\S1.1]{t3}.
}, analogously to Taubes's Seiberg-Witten $=$ Gromov theorem for closed symplectic four-manifolds. We will explain this motivation in detail in \S\ref{sec:origins}--\ref{sec:defech}.

\subsection{Additional structure on ECH}
\label{sec:addstr}

The definition of ECH capacities uses four additional structures on ECH, which we now briefly describe.

\paragraph{1. The $U$ map.}
Assuming that $Y$ is connected, there is a degree $-2$ map
\begin{equation}
\label{eqn:U}
U:ECH_*(Y,\xi,\Gamma) \longrightarrow ECH_{*-2}(Y,\xi,\Gamma).
\end{equation}
This is induced by a chain map
which is defined similarly to the differential, except that instead of counting ECH index 1 curves modulo $\R$ translation, it counts ECH index 2 curves that pass through a base point $(0,z)\in \R\times Y$.
Since $Y$ is connected, the induced map on homology
\begin{equation}
\label{eqn:UJ}
U:ECH_*(Y,\lambda,\Gamma,J) \longrightarrow ECH_{*-2}(Y,\lambda,\Gamma,J)
\end{equation}
does not depend on the choice of base point $z$, see \S\ref{sec:Udetails} for details. Taubes \cite{e5} showed that \eqref{eqn:UJ} agrees with an analogous $U$ map on Seiberg-Witten Floer cohomology, and in particular it gives a well-defined map \eqref{eqn:U}. Thus the $U$ map, like ECH, is in fact a topological invariant of $Y$.

If $Y$ is disconnected, then there is a different $U$ map for each component of $Y$. More precisely, suppose that $(Y,\lambda)=\coprod_{i=1}^n(Y_i,\lambda_i)$ with $Y_i$ connected, and let $\Gamma=(\Gamma_1,\ldots,\Gamma_n)\in H_1(Y)$. It follows from the definitions, and the fact that we are using coefficients in a field, that there is a canonical isomorphism
\[
ECH(Y,\xi,\Gamma) = \bigotimes_{i=1}^nECH(Y_i,\xi_i,\Gamma_i).
\]
The $U$ map on the left hand side determined by the component $Y_i$ is the tensor product on the right hand side of the $U$ map on $ECH(Y_i,\xi_i,\Gamma_i)$ with the identity on the other factors.

\paragraph{2. The ECH contact invariant.} 

ECH contains a canonical class defined as follows.
Observe that for any nondegenerate contact three-manifold $(Y,\lambda)$, the empty set of Reeb orbits is a generator of the chain complex $ECC(Y,\lambda,0,J)$. It follows from \eqref{eqn:daf} below that this chain complex generator is actually a cycle, i.e.
\[
\partial\emptyset = 0.
\]
(In this equation, the empty set is not the same as zero!)
ECH cobordism maps, described below, can be used to show that the homology class of this cycle does not depend on $J$ or $\lambda$, and thus represents a well-defined class
\[
c(\xi)\in ECH_*(Y,\xi,0),
\]
which we call the {\em ECH contact invariant\/}.
Taubes \cite{e5} showed that under the isomorphism \eqref{eqn:echswf}, this agrees with a related contact invariant in Seiberg-Witten Floer cohomology.

Although ECH and the U map on it are topological invariants of the three-manifold $Y$, the contact invariant can distinguish some contact structures. For example, if $\xi$ is overtwisted then $c(\xi)=0$. This holds because, as shown in the appendix to \cite{yml}, if $\xi$ is overtwisted then one can find a contact form such that the shortest Reeb orbit $\gamma$ bounds a unique holomorphic curve (which is a holomorphic plane) in $\R\times Y$; the latter turns out to have ECH index 1, so $\partial\gamma=\emptyset$.
On the other hand, it follows using the ECH cobordism maps defined in \cite{field} that $c(\xi)\neq 0$ whenever $(Y,\xi)$ is strongly symplectically fillable; a special case of this is proved in Example~\ref{ex:filling} below.

\paragraph{3. Filtered ECH.}

There is a refinement of ECH which sees not just the contact {\em structure\/} but also the contact {\em form\/}. To describe this, recall that if $\gamma$ is a Reeb orbit, its {\em symplectic action\/} is defined by
\[
\mc{A}(\gamma) = \int_\gamma\lambda.
\]
If $\alpha=\{(\alpha_i,m_i)\}$ is an ECH generator, define its symplectic action by
\[
\mc{A}(\alpha) = \sum_im_i\mc{A}(\alpha_i).
\]
It follows from the conditions on the almost complex structure $J$ that the restriction of $d\lambda$ to any $J$-holomorphic curve in $\R\times Y$ is pointwise nonnegative. Consequently, by Stokes's theorem, the differential decreases\footnote{In fact the inequality on the right side of \eqref{eqn:daf} is strict,
but we do not need this.} the symplectic action, i.e.
\begin{equation}
\label{eqn:daf}
\langle\partial\alpha,\beta\rangle\neq 0 \Longrightarrow \mc{A}(\alpha) \ge \mc{A}(\beta).
\end{equation}

Given $L\in\R$, define $ECC^L(Y,\lambda,\Gamma,J)$ to be the span of those generators $\alpha$ with $\mc{A}(\alpha)<L$. It follows from \eqref{eqn:daf} that this is a subcomplex of $ECC(Y,\lambda,\Gamma,J)$. The homology of this subcomplex is called {\em filtered ECH\/} and denoted by $ECH^L(Y,\lambda,\Gamma)$. It is shown in \cite[Thm.\ 1.3]{cc2} that filtered ECH does not depend on $J$. There is also a $U$ map (or $U$ maps when $Y$ is disconnected) defined on filtered ECH, which we continue to denote by $U$.

Unlike the usual ECH, filtered ECH depends heavily on the contact form $\lambda$. For example, if $Y=\partial E(a,b)$ with the standard contact form as in Example~\ref{ex:ell1}, then the symplectic action of a chain complex generator is given by
\begin{equation}
\label{eqn:ellipsoidaction}
\mc{A}(\gamma_1^{m_1}\gamma_2^{m_2}) = am_1+bm_2.
\end{equation}
Thus the rank of $ECH^L(\partial E(a,b))$ is the number of nonnegative integer linear combinations of $a$ and $b$ that are less than $L$. Obviously this depends on $a$ and $b$; but the ellipsoids for different $a$ and $b$ with their contact forms all determine the unique tight contact structure on $S^3$.
There is also a general scaling property: if $r>0$ is a positive constant, then there is a canonical isomorphism
\begin{equation}
\label{eqn:scaling}
ECH^L(Y,\lambda,\Gamma) = ECH^{rL}(Y,r\lambda,\Gamma).
\end{equation}

\paragraph{4. Cobordism maps.}

We now consider maps on ECH induced by cobordisms. For this purpose there are various kinds of cobordisms that one can consider. To describe these, let $(Y_+,\lambda_+)$ and $(Y_-,\lambda_-)$ be closed contact three-manifolds.

A {\em strong symplectic cobordism\/} from\footnote{Our use of the words ``from'' and ``to'' in this connection is controversial. In the usual TQFT language, one would say that $X$ is a cobordism from $Y_-$ to $Y_+$. However cobordism maps on ECH and other kinds of contact homology naturally go from the invariant of $Y_+$ to the invariant of $Y_-$. We apologize for any confusion.} $(Y_+,\lambda_+)$ to $(Y_-,\lambda_-)$ is a compact symplectic four-manifold $(X,\omega)$ with boundary \begin{equation}
\label{eqn:boundaryX}
\partial X = Y_+ - Y_-,
\end{equation}
such that $\omega|_{Y_\pm}=d\lambda_\pm$. Note that the signs in \eqref{eqn:boundaryX} are important; here $X$ has an orientation determined by the symplectic structure, while $Y_+$ and $Y_-$ have orientations determined by the contact structures. In particular, there is a distinction between the positive (or ``convex'') boundary $Y_+$ and the negative (or ``concave'') boundary $Y_-$.

An {\em exact symplectic cobordism\/} is a strong symplectic cobordism as above  such that there is a $1$-form $\lambda$ on $X$ with $d\lambda = \omega$ and $\lambda|_{Y_\pm}=\lambda_\pm$.
 
A strong (resp.\ exact) {\em symplectic filling\/} of $(Y,\lambda)$ is a strong (resp.\ exact) symplectic cobordism from $(Y,\lambda)$ to the empty set.

For example, if $X$ is a compact star-shaped domain in $\R^4$ with boundary $Y$, if $\omega$ is the standard symplectic form on $\R^4$, and if $\lambda$ is the Liouville form \eqref{eqn:lambdastd}, then $(X,\omega)$ is an exact symplectic filling of $(Y,\lambda|_Y)$.

Maps on ECH induced by exact symplectic cobordisms were constructed in \cite{cc2}, where they were used to prove the Arnold chord conjecture in three dimensions. More generally, maps on ECH induced by arbitrary strong symplectic cobordisms are constructed in \cite{field}.

To set up the theory of ECH capacities, we need a notion in between exact and strong symplectic cobordisms. Define a {\em weakly exact symplectic cobordism\/} to be a strong symplectic cobordism as above such that $\omega$ is exact (but $\omega$ need not have a primitive on $X$ which restricts to the contact forms on the boundary).

\begin{theorem}[{\cite[Thm.\ 2.3]{qech}}]
\label{thm:wesc}
Let $(X,\omega)$ be a weakly exact symplectic cobordism from $(Y_+,\lambda_+)$ to $(Y_-,\lambda_-)$, and assume that the contact forms $\lambda_\pm$ are nondegenerate. Then for each $L>0$ there are maps
\[
\Phi^L(X,\omega): ECH^L(Y_+,\lambda_+,0) \longrightarrow ECH^L(Y_-,\lambda_-,0)
\]
with the following properties:
\begin{description}
\item{(a)} $\phi^L(X,\omega)[\emptyset]=[\emptyset]$.
\item{(b)} If $U_+$ and $U_-$ are $U$ maps on $ECH^L(Y_\pm,\lambda_\pm,0)$ corresponding to components of $Y_\pm$ that are contained in the same component of $X$, then
\[
\phi^L(X,\omega)\circ U_+ = U_-\circ\phi^L(X,\omega).
\]
\end{description}
\end{theorem}

\begin{example}
\label{ex:filling}
If $Y_-=\emptyset$, i.e.\ if $(X,\omega)$ is a weakly exact symplectic filling of $(Y_+,\lambda_+)$, then the content of the theorem is that there are maps
\[
\Phi^L(X,\omega): ECH^L(Y_+,\lambda_+,0) \longrightarrow \Z/2
\]
with $\Phi^L(X,\omega)[\emptyset]=1$. In particular, it follows that $c(\xi_+)\neq 0 \in ECH(Y_+,\xi_+,0)$.
\end{example}

Theorem~\ref{thm:wesc} is proved using Seiberg-Witten theory, as we describe in \S\ref{sec:cobmap}. For now let us see how the above structure can be used to define ECH capacities.

\subsection{Definition of ECH capacities}
\label{sec:definecapacities}

Before defining ECH capacities of symplectic four-manifolds, we first need another three-dimensional definition. 

\paragraph{ECH spectrum.}

Let $(Y,\lambda)$ be a closed contact three-manifold, write $\xi=\Ker(\lambda)$ as usual, and assume that $c(\xi)\neq 0\in ECH(Y,\xi,0)$. We define a sequence of real numbers
\[
0 = c_0(Y,\lambda) < c_1(Y,\lambda) \le c_2(Y,\lambda) \le \cdots \le \infty,
\]
called the {\em ECH spectrum\/} of $(Y,\lambda)$, as follows.

Suppose first that $\lambda$ is nondegenerate and $Y$ is connected. Then $c_k(Y,\lambda)$ is the infimum of $L$ such that there is a class $\eta\in ECH^L(Y,\lambda,0)$ with $U^k\eta=[\emptyset]$. If no such class exists then $c_k(Y,\lambda)=\infty$. In particular, $c_k(Y,\lambda)<\infty$ if and only if $c(\xi)$ is in the image of $U^k$ on $ECH(Y,\xi,0)$.

\begin{example}
\label{ex:zeta}
Suppose $Y=\partial E(a,b)$ with $a/b$ irrational.
Denote the chain complex generators in order of increasing symplectic action by $\zeta_0,\zeta_1,\ldots$. We will see in \S\ref{sec:Uell} that $U\zeta_k=\zeta_{k-1}$ for $k>0$. It follows from this and \eqref{eqn:ellipsoidaction} that
\begin{equation}
\label{eqn:ckell}
c_k(\partial E(a,b))=N(a,b)_k.
\end{equation}
\end{example}

Continuing the definition of the ECH spectrum, if $(Y,\lambda) = \coprod_{i=1}^n(Y_i,\lambda_i)$ with $Y_i$ connected and $\lambda_i$ nondegenerate, let $U_i$ denote the $U$ map corresponding to the $i^{th}$ component. Then $c_k(Y,\lambda)$ is the infimum of $L$ such that there exists a class $\eta \in ECH^L(Y,\lambda,0)$ with
\begin{equation}
\label{eqn:multieta}
U_1^{k_1}\circ\cdots\circ U_n^{k_n}\eta=[\emptyset]
\end{equation}
whenever $k_1+\cdots+k_n=k$. It follows from some algebra
in \cite[\S5]{qech} that
\begin{equation}
\label{eqn:sdu}
c_k\left(\coprod_{i=1}^n(Y_i,\lambda_i)\right) = \max_{k_1+\cdots+k_n=k}\sum_{i=1}^nc_{k_i}(Y_i,\lambda_i).
\end{equation}

Finally, if $(Y,\lambda)$ is a closed contact three-manifold with $\lambda$ possibly degenerate, define $c_k(Y,\lambda)=\lim_{n\to\infty}c_k(Y,f_n\lambda)$, where $f_n:Y\to\R^{>0}$ are functions on $Y$ with $f_n\lambda$ nondegenerate and $\lim_{n\to\infty}f_n=1$ in the $C^0$ topology. It can be shown using Theorem~\ref{thm:wesc} that this is well-defined
and still satisfies \eqref{eqn:sdu}.
For example, equation \eqref{eqn:ckell} also holds when $a/b$ is rational.

\paragraph{ECH capacities.} We are now ready to define ECH capacities.

\begin{definition}
\label{def:LD}
A (four-dimensional) {\em Liouville domain\/} is a weakly\footnote{Our definition of ``Liouville domain'' is more general than the usual definition, and perhaps should be called a ``weak Liouville domain''. Ordinarily a ``Liouville domain'' is an exact symplectic filling.} exact symplectic filling $(X,\omega)$ of a contact three-manifold $(Y,\lambda)$.
\end{definition}

\begin{definition}
If $(X,\omega)$ is a four-dimensional Liouville domain with boundary $(Y,\lambda)$, define the {\em ECH capacities\/} of $(X,\omega)$ by
\[
c_k(X,\omega) = c_k(Y,\lambda)\in[0,\infty].
\]
\end{definition}

To see why this definition makes sense, first note that $c(\xi)\neq 0 \in ECH(Y,\xi,0)$ by Example~\ref{ex:filling}, so $c_k(Y,\lambda)$ is defined. We also need to explain why $c_k(X,\omega)$ does not depend on the choice of contact form $\lambda$ on $Y$ with $d\lambda=\omega|_Y$. Let $\lambda'$ be another such contact form. Assume that $\lambda$ and $\lambda'$ are nondegenerate (one can handle the degenerate case by taking a limit of nondegenerate forms). Since $d\lambda=d\lambda'$, the Reeb vector fields $R$ and $R'$ for $\lambda$ and $\lambda'$ are related by $R'=fR$ where $f:Y\to\R^{>0}$. Let $J$ be an almost complex structure on $\R\times Y$ as needed to define the ECH of $\lambda$. Let $J'$ be the almost complex structure on $\R\times Y$ which agrees with $J$ on the contact planes $\xi$ but sends $\partial_s\mapsto R'$. There is then a canonical isomorphism of chain complexes
\begin{equation}
\label{eqn:dll}
ECC^L(Y,\lambda,0,J) = ECC^L(Y,\lambda',0,J')
\end{equation}
which preserves the $U$ maps and the empty set. The reason is that the chain complexes $ECC(Y,\lambda,\Gamma,J)$ and $ECC(Y,\lambda',\Gamma,J')$ have the same generators, and when $\Gamma=0$ the symplectic actions as defined using $\lambda$ or $\lambda'$ agree by Stokes's theorem because $d\lambda=d\lambda'$. Furthermore the $J$-holomorphic curves in $\R\times Y$ agree with the $J'$-holomorphic curves after rescaling the $\R$ coordinate on $\R\times Y$ using the function $f$. And it follows immediately from \eqref{eqn:dll} that $c_k(Y,\lambda) = c_k(Y,\lambda')$.

For example, the Ellipsoid property of ECH capacities now follows from \eqref{eqn:ckell}.

\paragraph{Monotonicity for Liouville domains.}
We now explain why the Monotonicity property holds when $(X_0,\omega_0)$ and $(X_1,\omega_1)$ are Liouville domains. By a limiting argument, one can assume that $(X_0,\omega_0)$ symplectically embeds into the interior of $(X_1,\omega_1)$. For $i=0,1$, let $Y_i=\partial X_i$, and let $\lambda_i$ be a contact form on $Y_i$ with $\partial\lambda_i=\omega|_{Y_i}$. Then $(X\setminus\varphi(\op{int}(X_0)),\omega_1)$ is a weakly exact symplectic cobordism from $(Y_1,\lambda_1)$ to $(Y_0,\lambda_0)$. The Monotonicity property in this case now follows from:

\begin{lemma}
Let $(X,\omega)$ be a weakly exact symplectic cobordism from $(Y_+,\lambda_+)$ to $(Y_-,\lambda_-)$. Then
\[
c_k(Y_-,\lambda_-) \le c_k(Y_+,\lambda_+)
\]
for each $k\ge 0$.
\end{lemma}

This lemma follows almost immediately from the fact that $c_k$ is defined solely in terms of the filtration, the $U$ maps, and the contact invariant, and these structures are preserved by the cobordism map. Here are the details.

\begin{proof}
By a limiting argument we may assume that the contact forms $\lambda_\pm$ are nondegenerate. Let $U_1^+,\ldots,U_m^+$ denote the $U$ maps on $ECH(Y_+,\lambda_+,0)$ associated to the components of $Y_+$, and let $U_1^-,\ldots,U_n^-$ denote the $U$ maps on $ECH(Y_-,\lambda_-,0)$ associated to the components of $Y_-$. Let $L>0$ and suppose that $c_k(Y_+,\lambda)<L$; it is enough to show that $c_k(Y_-,\lambda_-)\le L$. Since $c_k(Y_+,\lambda)<L$, there exists a class $\eta_+\in ECH^L(Y_+,\lambda_+,0)$ such that
\begin{equation}
\label{eqn:Ueta+}
(U_1^+)^{k_1}\cdots(U_m^+)^{k_m}\eta_+=[\emptyset]
\end{equation}
whenever $k_1+\cdots+k_m=k$.

Let
\[
\eta_- = \Phi^L(X,\omega)\eta_+ \in ECH^L(Y_-,\lambda_-,0).
\]
We claim that
\begin{equation}
\label{eqn:Ueta-}
(U_1^-)^{k_1}\cdots(U_n^-)^{k_n}\eta_-=[\emptyset]
\end{equation}
whenever $k_1+\cdots+k_n=k$, so that $c_k(Y_-,\lambda_-)\le L$.
To prove this, first note that by Exercise~\ref{ex:wesc} below, each component of $Y_-$ is contained in the same component of $X$ as some component of $Y_+$. Equation \eqref{eqn:Ueta-} then follows from equation \eqref{eqn:Ueta+} together with Theorem~\ref{thm:wesc}.
\end{proof}

\begin{exercise}
\label{ex:wesc}
Show that if $(X,\omega)$ is a weakly exact symplectic cobordism from $(Y_+,\lambda_+)$ to $(Y_-,\lambda_-)$ with $Y_-\neq\emptyset$, then $Y_+\neq\emptyset$.
(See answer in \S\ref{sec:answers}.)
\end{exercise}

\paragraph{Non-Liouville domains.}

We extend the definition of ECH capacities to symplectic four-manifolds which are not Liouville domains by a simple trick:
If $(X,\omega)$ is any symplectic four-manifold, define
\[
c_k(X,\omega) = \sup\{c_k(X',\omega')\},
\]
where the supremum is over Liouville domains $(X',\omega')$ that can be symplectically embedded into $X$. It is a tautology that this new definition of $c_k$ is monotone with respect to symplectic embeddings. And this new definition agrees with the old one when $(X,\omega)$ is already a Liouville domain, by the Monotonicity property for the old definition of $c_k$ with respect to symplectic embeddings of Liouville domains.

\paragraph{Properties of ECH capacities.}

The remaining properties of ECH capacities in Theorem~\ref{thm:ECHcapacities} are proved as follows.
The Disjoint Union property follows from \eqref{eqn:sdu}. The Conformality property follows from the definitions and the scaling property \eqref{eqn:scaling} when $r>0$, and a similar argument\footnote{
In particular, there is a canonical isomorphism of chain complexes (with $\Z/2$ coefficients)
\[
ECC^L(Y,\lambda,\Gamma,J) = ECC^L(Y,-\lambda,-\Gamma,-J).
\]
Note that the resulting isomorphism $ECH(Y,\xi,\Gamma)=ECH(Y,-\xi,-\Gamma)$ corresponds, under Taubes's isomorphism \eqref{eqn:echswf}, to ``charge conjugation invariance'' of Seiberg-Witten Floer cohomology (with $\Z/2$ coefficients).
} when $r<0$. We will prove the Polydisk property at the end of \S\ref{sec:toric}. The proof of the Volume property is beyond the scope of these notes; it is given in \cite{vc}, using ingredients from Taubes's proof of the Weinstein conjecture \cite{tw1}.

\section{Origins of ECH}
\label{sec:origins}

One of the main goals of these notes is to explain something about where ECH comes from. The starting point for the definition of ECH is Taubes's ``SW=Gr'' theorem \cite{swgr} asserting that the Seiberg-Witten invariants of a symplectic four-manifold agree with a ``Gromov invariant'' counting holomorphic curves. The basic idea of ECH is that it is a three-dimensional analogue of Taubes's Gromov invariant. So we will now review Taubes's Gromov invariant in such a way as to make the definition of ECH appear as natural as possible. The impatient reader may wish to skip ahead to the definition of ECH in \S\ref{sec:defech}, and refer back to this section when more motivation is needed.

\subsection{Taubes's ``SW=Gr'' theorem}

We first briefly recall the statement of Taubes's ``SW=Gr'' theorem.
Let $X$ be a closed connected oriented four-manifold. (All manifolds in these notes are smooth.) Let $b_2^+(X)$ denote the dimension of a maximal positive definite subspace $H_2^+(X;\R)$ of $H_2(X;\R)$ with respect to the intersection pairing. Let $\Spinc(X)$ denote the set of spin-c structures\footnote{A spin-c struture on an oriented $n$-manifold is a lift of the frame bundle from a principal $SO(n)$ bundle to a principal $\op{Spin}^c(n)=\op{Spin}(n)\times_{\Z/2}U(1)$ bundle. However we will not need this here.} on $X$; this is an affine space over $H^2(X;\Z)$.
If $b_2^+(X)>1$, one can define the {\em Seiberg-Witten invariant\/}
\begin{equation}
\label{eqn:sw}
SW(X): \Spin^c(X)\to \Z
\end{equation}
by counting solutions to the Seiberg-Witten equations, see e.g.\ \cite{morgan}. More precisely, the Seiberg-Witten invariant depends on a choice of ``homology orientation'' of $X$, namely an orientation of $H_0(X;\R)\oplus H_1(X;\R) \oplus H_2^+(X;\R)$. Switching the homology orientation will multiply the Seiberg-Witten invariant by $-1$. If $b_2^+(X)=1$, the Seiberg-Witten invariant \eqref{eqn:sw} is still defined, but depends on an additional choice of one of two possible ``chambers''; one can identify a chamber with an orientation of the line $H_2^+(X;\R)$.

While the Seiberg-Witten invariants are very powerful for distinguishing smooth four-manifolds, it is also nearly impossible to compute them directly except in very special cases (although there are axiomatic properties which one can use to compute the invariants for more interesting examples). However, Taubes showed that if $X$ has a symplectic form $\omega$, then the Seiberg-Witten invariants of $X$ are equal to a certain count of holomorphic curves, which are much easier to understand than solutions to the Seiberg-Witten equations. Namely, for each $A\in H_2(X)$, Taubes defines a ``Gromov invariant''
\[
Gr(X,\omega,A)\in\Z,
\]
which is a certain count of holomorphic curves in the homology class $A$, which we will review in \S\ref{sec:defgr} below. Further, the symplectic structure $\omega$ determines a distinguished spin-c structure $\frak{s}_\omega$, so that we can identify
\begin{equation}
\label{eqn:spinc}
\begin{split}
H_2(X) &= \Spinc(X),\\
A &\leftrightarrow \frak{s}_\omega + \op{PD}(A).
\end{split}
\end{equation}
We can now state:

\begin{theorem}[Taubes]
Let $(X,\omega)$ be a closed connected symplectic four-manifold with $b_2^+(X)>1$. Then $X$ has a homology orientation such that under the identification \eqref{eqn:spinc},
\[
SW(X) = Gr(X,\omega,\cdot).
\]
\end{theorem}

\begin{remark}
A version of this theorem also holds when $b_2^+(X)=1$. Here one needs to compute the Seiberg-Witten invariant using the chamber determined by the cohomology class of $\omega$. Also, in this case the definition of the Gromov invariant needs to be modified in the presence of symplectic embedded spheres of square $-1$, see \cite{liliu}.
\end{remark}

\subsection{Holomorphic curves in symplectic manifolds}

We now briefly review what we will need to know about holomorphic curves in order to define Taubes's Gromov invariant. Proofs of the facts recalled here may be found for example in \cite{msj}.

Let $(X^{2n},\omega)$ be a closed symplectic manifold. An {\em $\omega$-compatible almost complex structure\/} is a bundle map $J:TX\to TX$ such that $J^2=-1$ and $g(v,w) = \langle Jv,w\rangle$ defines a Riemannian metric on $X$. Given $\omega$, the space of compatible almost complex structures $J$ is contractible. Fix an $\omega$-compatible\footnote{Taubes's theorem presumably still works if one generalizes from compatible to tame almost complex structures.} almost complex structure $J$.

A {\em $J$-holomorphic curve\/} in $(X,\omega)$ is a holomorphic map $u:(\Sigma,j)\to(X,J)$ where $(\Sigma,j)$ is a compact Riemann surface (i.e.\ $\Sigma$ is a compact surface and $j$ is an almost complex structure on $\Sigma$), $u:\Sigma\to X$ is a smooth map, and $J\circ du=du\circ j$.  The curve $u$ is considered equivalent to $u':(\Sigma',j')\to (X,J)$ if there exists a holomorphic bijection $\phi:(\Sigma,j)\to(\Sigma',j')$  such that $u'\circ \phi = u$. Thus a $J$-holomorphic curve is formally an equivalence class of triples $(\Sigma,j,u)$ satisfying the above conditions.

We call a $J$-holomorphic curve {\em irreducible\/} if its domain is connected.

If $u:(\Sigma,j)\to(X,J)$ is an embedding, then the equivalence class of the $J$-holomorphic curve $u$ is determined by its image $C=u(\Sigma)$ in $X$. Indeed, an embedded $J$-holomorphic curve is equivalent to a closed two-dimensional submanifold $C\subset X$ such that $J(TC)=TC$.

More generally, a holomorphic curve $u:\Sigma\to X$ is called {\em somewhere injective\/} if there exists $z\in\Sigma$ such that $u^{-1}(u(z))=\{z\}$ and $du_z:T_z\Sigma\to T_{u(z)}X$ is injective. One can show that in this case $u$ is an embedding on the complement of a countable subset of $\Sigma$ (which is finite in the case of interest where $n=2$), and the equivalence class of $u$ is still determined by its image in $X$. On the other hand, $u$ is called {\em multiply covered\/} if there exists a branched cover $\phi:(\Sigma,j)\to(\Sigma',j')$ of degree $d>1$ and a holomorphic map $u':(\Sigma',j')\to(X,J)$ such that $u=u'\circ \phi$.

It is a basic fact that every irreducible holomorphic curve is either somewhere injective or multiply covered. In particular, every irreducible holomorphic curve is the composition of a somewhere injective holomorphic curve with a branched cover of degree $d\ge 1$. When $d>1$, the holomorphic curve is not determined just by its image in $X$; it depends also on the degree $d$, the images of the branch points in $X$, and the monodromy around the branch points.

Define the {\em Fredholm index\/} of a holomorphic curve $u:(\Sigma,j)\to (X,J)$ by
\begin{equation}
\label{eqn:ind}
\ind(u) = (n-3)\chi(\Sigma) + 2\langle c_1(TX),u_*[\Sigma]\rangle.
\end{equation}
Here $c_1(TX)$ denotes the first Chern class of $TX$, regarded as a complex vector bundle using the almost complex structure $J$. The isomorphism class of this complex vector bundle depends only on the symplectic structure and not on the compatible almost complex structure.

A transversality argument shows that if $J$ is generic, then for each somewhere injective holomorphic curve $u$, the moduli space of holomorphic curves near $u$ is a smooth manifold of dimension $\ind(u)$, cut out transversely in a sense to be described below.
 Unfortunately, this usually does not hold for multiply covered curves. Even if all somewhere injective holomorphic curves are cut out transversely, there can still be multiply covered holomorphic curves $u$ such that $\ind(u)$ is less than the dimension of the moduli space near $u$, or even negative. This is a major technical problem in defining holomorphic curve counting invariants in general, and it also causes some complications for ECH, as we will see in the proof that $\partial^2=0$ in \S\ref{sec:200pages} and especially in the construction of cobordism maps in \S\ref{sec:cobmap}.

\subsection{Deformations of holomorphic curves}

We now clarify what it means for a holomorphic curve to be ``cut out transversely''. To simplify the discussion we restrict attention to immersed curves, which are all we need to consider to define Taubes's Gromov invariant. 

Let $u:C\to X$ be an immersed $J$-holomorphic curve, which by abuse of notation we will usually denote by $C$. Then $C$ has a well defined normal bundle $N_C$, which is a complex vector bundle of rank $n-1$ over $C$. The derivative of the equation for $C$ to be $J$-holomorphic defines a first-order elliptic differential operator
\[
D_C:\Gamma(N_C) \longrightarrow \Gamma(T^{0,1}C\tensor N_C),
\]
which we call the {\em deformation operator\/} of $C$. Here $\Gamma$ denotes the space of smooth sections.

To explain this operator in more detail, we first recall some general formalism. Suppose $E\to B$ is a smooth vector bundle and $\psi:B\to E$ is a smooth section. Let $x\in B$ be a zero of $\psi$. Then the derivative of the section $\psi$ at $x$ defines a canonical map
\begin{equation}
\label{eqn:nablacan}
\nabla\psi:T_xB\to E_x.
\end{equation}
Namely, the derivative of $\psi$, regarded as a smooth map $B\to E$, has a differential $d\psi_x:T_xB\to T_{(x,0)}E$, and the map \eqref{eqn:nablacan} is obtained by composing this with the projection $T_{(x,0)}E=T_xB\oplus E_x\to E_x$.

To put holomorphic curves into the above framework, let $\mc{B}$ be the infinite dimensional (Frechet) manifold of immersed compact surfaces in $X$.  Given an immersed surface $u:C\to X$, let $N_C=u^*TX/TC$ denote the normal bundle to $C$, which is a rank $2n-2$ real vector bundle over $C$, and let $\pi_{N_C}:u^*TX\to N_C$ denote the quotient map. We can identify $T_C\mc{B}=\Gamma(N_C)$. There is an infinite dimensional vector bundle $\mc{E}\to \mc{B}$ whose fiber over $C$ is the space of smooth bundle maps $TC\to N_C$. We define a smooth section $\dbar:\mc{B}\to \mc{E}$ by defining $\dbar(C):TC\to N_C$ to be the map sending $v\mapsto \pi_{N_C}(Jv)$. Then $C$ is $J$-holomorphic if and only if $\dbar(C)=0$. In this case the derivative of $\dbar$ defines a map $\Gamma(N_C)\to\Gamma(T^*C\tensor N_C)$. Furthermore, since $C$ is $J$-holomorphic, the values of this map anticommute with $J$, so it is actually an operator $\Gamma(N_C)\to\Gamma(T^{0,1}C\tensor N_C)$. This is the deformation operator $D_C$.

One can write the operator $D_C$ in local coordinates as follows. Let $z=s+it$ be a local coordinate on $C$, use $id\zbar$ to locally trivialize $T^{0,1}C$, and choose a local trivialization of $N_C$ over this coordinate neighborhood. With respect to these coordinates and trivializations, the operator $D_C$ locally has the form
\[
D_C = \partial_s + J\partial_t + M(s,t).
\]
Here $M(s,t)$ is a real matrix of size $2n-2$ determined by the derivatives of $J$ in the normal directions to $C$.

We say that $C$ is {\em regular\/}, or ``cut out transversely'', if the operator $D_C$ is surjective. In this case the moduli space of holomorphic curves is a manifold near $C$, and its tangent space at $C$ is the kernel of $D_C$.

In the analysis one often needs to extend the operator $D_C$ to suitable Banach space completions of the spaces of smooth sections, for example to extend it to an operator
\begin{equation}
\label{eqn:DCext}
D_C:L^2_1(C,N_C) \longrightarrow L^2(C,T^{0,1}C\tensor N_C).
\end{equation}
Since $D_C$ is elliptic, the extended operator is Fredholm, and its kernel consists of smooth sections. It follows from the Riemann-Roch theorem that the index of this Fredholm operator is the Fredholm index $\ind(C)$ defined in \eqref{eqn:ind}. This is why the moduli space of holomorphic curves near a regular curve $C$, under our simplifying assumption that $C$ is immersed, has dimension $\ind(C)$.

\subsection{Special properties in four dimensions}
\label{sec:4dspecial}

In four dimensions, holomorphic curves enjoy three additional special properties which are important for our story.
To state the first special property, if $p$ is an isolated intersection point of surfaces $S_1$ and $S_2$ in $X$, let $m_p(S_1\cap S_2)\in\Z$ denote the intersection multiplicity at $p$.

\begin{intpos}
 Let $C_1$ and $C_2$ be distinct irreducible somewhere injective $J$-holomorphic curves in a symplectic four-manifold. Then the intersection points of $C_1$ and $C_2$ are isolated; and for each $p\in C_1\cap C_2$, the intersection multiplicity $m_p(C_1\cap C_2)>0$. Moreover, $m_p(C_1\cap C_2)=1$ if and only if $C_1$ and $C_2$ are embedded near $p$ and intersect transversely at $p$.
\end{intpos}

It is easy to see that if $C_1$ and $C_2$ are embedded near $p$ and intersect transversely at $p$, so that $m_p(C_1\cap C_2)=\pm1$, then in fact $m_p(C_1\cap C_2)=+1$, esentially because a complex vector space has a canonical orientation. The hard part of the theorem is to deal with the cases where $C_1$ and $C_2$ are not embedded near $p$ or do not intersect transversely at $p$.

In particular, intersection positivity implies that the homological intersection number
\[
[C_1]\cdot[C_2]\ = \sum_{p\in C_1\cap C_2}m_p(C_1\cap C_2)\ge 0,
\]
with equality if and only if $C_1$ and $C_2$ are disjoint. Note that the assumption that $C_1$ and $C_2$ are distinct is crucial. A single holomorphic curve $C$ can have $[C]\cdot [C]<0$; for example, the exceptional divisor in a blowup is a holomorphic sphere $C$ of square $-1$. What intersection positivity implies in this case is that the exceptional divisor is the unique holomorphic curve in its homology class.

The second special property of holomorphic curves in four dimensions is the adjunction formula. To state it, define a {\em singularity\/} of a somewhere injective $J$-holomorphic curve $C$ in a symplectic four-manifold $X$ to be a point in $X$ where $C$ is not locally an embedding. A {\em node\/} is a singularity given by a transverse self-intersection whose inverse image in the domain of $C$ consists of two points (where $C$ is an immersion). Let $\chi(C)$ denote the Euler characteristic of the domain of $C$ (which may be larger than the Euler characteristic of the image of $C$ in $X$ if there are singularities).

\begin{adjfor} Let $C$ be a somewhere injective $J$-holomorphic curve in a symplectic four-manifold $(X,\omega)$. Then the singularities of $C$ are isolated, and
\begin{equation}
\label{eqn:adj4}
\langle c_1(TX),[C]\rangle = \chi(C) + [C]\cdot[C]-2\delta(C)
\end{equation}
where $\delta(C)$ is a count of the singularities of $C$ with positive integer weights. Moreover, a singularity has weight $1$ if and only if it is a node.
\end{adjfor}

In particular, we have
\begin{equation}
\label{eqn:adj4i}
\chi(C) + [C]\cdot[C] - \langle c_1(TX),[C]\rangle \ge 0,
\end{equation}
with equality if and only if $C$ is embedded.

\begin{exercise}
\label{ex:adj}
Prove the adjunction formula in the special case when $C$ is immersed and the only singularities of $C$ are nodes.
\end{exercise}

The third special property of holomorphic curves in four dimensions is a version of Gromov compactness using currents, which does not require any genus bound. The usual version of Gromov compactness asserts that a sequence of holomorphic curves of fixed genus with an upper bound on the symplectic area has a subsequence which converges in an appropriate sense to a holomorphic curve. In the connection with Seiberg-Witten theory, multiply covered holomorphic curves naturally arise, but the information about the branch points, and hence about the genus of their domains, is not relevant. To keep track of the relevant information, define a {\em holomorphic current\/} in $X$ to be a finite set of pairs $\mc{C} = \{(C_i,d_i)\}$ where the $C_i$ are distinct irreducible somewhere injective $J$-holomorphic curves, and the $d_i$ are positive integers. 

\begin{gcc} (Taubes, \cite[Prop.\ 3.3]{tgc})
Let $(X,\omega)$ be a compact symplectic four-manifold, possibly with boundary, and let $J$ be an $\omega$-compatible almost complex structure. Let $\{\mc{C}_n\}_{n\ge 1}$ be a sequence of $J$-holomorphic currents (possibly with boundary in $\partial X$) such that $\int_{\mc{C}_n}\omega$ has an $n$-independent upper bound. Then there is a subsequence which converges as a current and as a point set to a $J$-holomorphic current $\mc{C}\subset X$ (possibly with boundary in $\partial X$).
\end{gcc}

Here ``convergence as a current'' means that if $\sigma$ is any 2-form then $\lim_{n\to\infty}\int_{\mc{C}_n}\sigma=\int_\mc{C}\sigma$. ``Convergence as a point set'' means that the corresponding subsets of $X$ converge with respect to the metric on compact sets defined by
\[
d(K_1,K_2) = \sup_{x_1\in K_1}\inf_{x_2\in K_2}d(x_1,x_2) + \sup_{x_2\in K_2}\inf_{x_1\in K_1}d(x_2,x_1).
\]

\subsection{Taubes's Gromov invariant}
\label{sec:defgr}

We now have enough background in place to define Taubes's Gromov invariant. While the definition is a bit complicated, we will be able to compute examples in \S\ref{sec:mappingtorusexample}, and this is a useful warmup for the definition of ECH.

\paragraph{What to count.}

Let $(X^4,\omega)$ be a closed connected symplectic four-manifold, and let $A\in H_2(X)$. We define the Gromov invariant $Gr(X,\omega,A)\in{\mathbb Z}$ as follows. Fix a generic $\omega$-compatible almost complex structure $J$. The rough idea is to count $J$-holomorphic currents representing the homology class $A$ in ``maximum dimensional moduli spaces''.

To explain the latter notion, define an integer
\begin{equation}
\label{eqn:I4}
I(A) = \langle c_1(TX),A\rangle + A\cdot A.
\end{equation}
In fact one can show that $I(A)$ is always even. The integer $I(A)$  is the closed four-manifold version of the ECH index, a crucial notion which we will introduce in \S\ref{sec:ECHindex}. For now, the significance of the integer $I(A)$ is the following. Let $C$ be a somewhere injective $J$-holomorphic curve. By \eqref{eqn:ind}, the Fredholm index of $C$ is given by
\begin{equation}
\label{eqn:ind4}
\ind(C) = -\chi(C) + 2\langle c_1(TX),[C]\rangle.
\end{equation}
It follows from this equation and the adjunction formula \eqref{eqn:adj4} that
\begin{equation}
\label{eqn:ie4}
\ind(C) = I([C]) - 2\delta(C).
\end{equation}
That is, the maximum possible value of $\ind(C)$ for a somewhere injective holomorphic curve $C$ with homology class $[C]=A$ is $I(A)$, which is attained exactly when $C$ is embedded.

The Gromov invariant $Gr(X,\omega,A)\in\Z$ is now a count of ``admissible'' holomorphic currents in the homology clas $A$.
Here the homology class of a holomorphic current $\mc{C}=\{(C_i,d_i)\}$ is defined by
\[
[\mc{C}] = \sum_id_i[C_i] \in H_2(X).
\]
Furthermore, the current $\mc{C}$ is called ``admissible'' if $d_i=1$ whenever $C_i$ is a sphere with $[C_i]\cdot[C_i]<0$.

If $I(A)<0$, then there are no admissible holomorphic currents in the homology class $A$ as we will show in a moment, and we define $Gr(X,\omega,A)=0$.

The most important case for our story is when $I(A)=0$. The admissible holomorphic currents in this case are described by the following lemma.

\begin{lemma}
\label{lem:I04}
Let $\mc{C}=\{(C_i,d_i)\}$ be an admissible holomorphic current with homology class $[\mc{C}]=A$. Then $I(A)\ge 0$. Moreover, if $I(A)=0$, then the following hold:
\begin{description}
\item{(a)} The holomorphic curves $C_i$ are embedded and disjoint.
\item{(b)} $d_i=1$ unless $C_i$ is a torus with $[C_i]\cdot[C_i]=0$.
\item{(c)} $\ind(C_i)=I([C_i])=0$ for each $i$.
\end{description}
\end{lemma}

\begin{proof}
It follows directly from the definition of $I$ that if $B_1,B_2\in H_2(X)$ then
\begin{equation}
\label{eqn:Isum}
I(B_1+B_2)=I(B_1)+I(B_2)+2B_1\cdot B_2.
\end{equation}
Applying this to $A=\sum_id_i[C_i]$ gives
\begin{equation}
\label{eqn:I(A)}
I(A) = \sum_id_iI([C_i]) + \sum_i(d_i^2-d_i)[C_i]\cdot[C_i] + \sum_{i\neq j}[C_i]\cdot [C_j].
\end{equation}
Now the terms on the right hand side are all nonnegative. To see this, first note that $\ind(C_i)\ge 0$, since we are assuming that $J$ is generic so that $C_i$ is regular. So by \eqref{eqn:ie4} we have $I([C_i])\ge 0$, with equality only if $C_i$ is embedded. In addition, if we combine the inequality $\ind(C_i)\ge 0$ with the adjunction formula \eqref{eqn:adj4i} for $C_i$, we find that
\begin{equation}
\label{eqn:adjchi}
\chi(C_i) + 2[C_i]\cdot[C_i]\ge 0
\end{equation}
with equality only if $C_i$ is embedded. In particular, the only way that $[C_i]\cdot[C_i]$ can be negative is if $C_i$ is an embedded sphere with square $-1$; and in this case admissibility forces $d_i=1$, so that the corresponding term in \eqref{eqn:I(A)} is zero. Finally, we know by intersection positivity that $[C_i]\cdot[C_j]\ge 0 $ with equality if and only if $C_i$ and $C_j$ are disjoint. We conclude that $I(A)\ge 0$, and if $I(A)=0$ then the curves $C_i$ are embedded and disjoint, $\ind(C_i)=I([C_i])=0$, and $d_i>1$ only if $C_i$ is a torus with square zero. (The inequality \eqref{eqn:adjchi} also allows $[C_i]\cdot[C_i]=0$ when $C_i$ is a sphere, but this would require $I([C_i])= 2$ and so cannot happen here.)
\end{proof}

One consequence of this lemma is that when $I(A)=0$, we have a finite set of holomorphic currents to count:

\begin{lemma}
\label{lem:finite4}
If $I(A)=0$, then the set of admissible holomorphic currrents $\mc{C}$ with homology class $[\mc{C}]=A$ is finite.
\end{lemma}

\begin{proof}
Suppose $\{\mc{C}_k\}_{k=1,2,\ldots}$ is an infinite sequence of distinct such currents. By Gromov compactness with currents, the sequence converges as a current and a point set to a holomorphic current $\mc{C}_\infty$. Convergence as a current implies that $[\mc{C}_\infty]=A$. An argument using the Fredholm index which we omit shows that $\mc{C}_\infty$ is also admissible. Then by Lemma~\ref{lem:I04}, $\mc{C}_\infty=\{(C_i,d_i)\}$ where $\ind(C_i)=0$ for each $i$ and $d_i=1$ unless $C_i$ is a torus of square zero. We are assuming that $J$ is generic, so each $C_i$ is isolated in the moduli space of holomorphic curves. If every $d_i=1$, one can use convergence as a current and a point set to show that possibly after passing to a subsequence, each $\mc{C}_k$ has an embedded component such that the sequence of these embedded components converges in the smooth topology to $C_i$, which is a contradiction. If any $d_i>1$, one needs an additional lemma from \cite{taubes:counting} asserting that if $J$ is generic, then the unbranched multiple covers of the tori of square zero are also regular.
\end{proof}

\paragraph{How to count.}

When $I(A)=0$, we define $Gr(X,\omega,A)\in\Z$ to be the sum, over all admissible holomorphic currents $\mc{C}=\{(C_i,d_i)\}$ with homology class $[\mc{C}]=A$, of a weight $w(\mc{C})\in\Z$ which we now define. The weight is given by a product of weights associated to the irreducible components,
\[
w(\mc{C})=\prod_iw(C_i,d_i).
\]
To complete the definition, we need to define the integer $w(C,d)$ when $C$ is an irreducible embedded holomorphic curve with $\ind=0$, and $d$ is a positive integer (which is $1$ unless $C$ is a torus with square $0$).

If $d=1$, then $W(C,1)=\epsilon(C)\in\{\pm1\}$ is defined as follows. 
Roughly speaking, $\epsilon(C)$ is the sign of the determinant of the operator $D_C$, which is the sign of the spectral flow from $D_{C}$ (extended as in \eqref{eqn:DCext}) to a complex linear operator. What this means is the following: one can show that there exists a differentiable 1-parameter family of operators $\{D_t\}_{t\in[0,1]}$ between the same spaces such that $D_0=D_{C}$; the operator $D_1$ is complex linear; there are only finitely many $t$ such that $D_t$ is not invertible; and for each such $t$, the operator $D_t$ has one-dimensional kernel, and the derivative of $D_t$ defines an isomorphism from the kernel of $D_t$ to the cokernel of $D_t$. Then $\epsilon(C)$ is simply $-1$ to the number of such $t$. One can show that this is well-defined, and we will compute some examples in \S\ref{sec:mappingtorusexample}.

It remains to define the weights $w(C,d)$ when $d>1$ and $C$ is a torus of square zero. The torus $C$ has three connected unbranched double covers, classified by nonzero elements of $H^1(C;{\mathbb Z}/2)$. By \cite{taubes:counting}, if $J$ is generic then the corresponding doubly covered holomorphic curves are regular. Each of these double covers then has a sign $\epsilon$ defined above. The weight $w(C,d)$ depends only on $d$, the sign of $C$, and the number of double covers with each sign. We denote this number by $f_{\pm,k}(d)$, where $\pm$ indicates the sign $\epsilon(C)$, and $k\in\{0,1,2,3\}$ is the number of double covers whose sign disagrees with that of $C$. To define the numbers $f_{\pm,k}(d)$, combine them into a generating function
\[
f_{\pm,k}=1+\sum_{d\ge 1}f_{\pm,k}(d)t^d.
\]
Then
\begin{equation}
\label{eqn:gf}
\begin{split}
f_{+,0} &= \frac{1}{1-t},\\
f_{+,1}&=1+t,\\
f_{+,2} &= \frac{1+t}{1+t^2},\\
f_{+,3} &= \frac{(1+t)(1-t^2)}{1+t^2},\\
f_{-,k}&=\frac{1}{f_{+,k}}.
\end{split}
\end{equation}

Where do these generating functions come from? It is shown in \cite{taubes:counting} that $Gr(X,\omega,A)$ is independent of the choice of $J$ and invariant under deformation of the symplectic form $\omega$; another proof is given in \cite{ip}. This invariance requires the generating functions $f_{\pm,k}$ to satisfy certain relations, because of bifurcations of holomorphic curves that can occur as one deforms $J$ or $\omega$. For example, it is possible for a pair of cancelling tori with opposite signs to be created or destroyed, and this forces the relation $f_{+,k}f_{-,k}=1$. We will see another relation in the example in \S\ref{sec:mappingtorusexample}. One still has some leeway in choosing the generating functions to obtain an invariant of symplectic four-manifolds; however the choice above is the one that agrees with Seiberg-Witten theory, for reasons we will explain in \S\ref{sec:gf}.

\paragraph{The case $I(A)>0$.}
To define the Gromov invariant $Gr(X,\omega,A)$ when  $I(A)\ge 0$, choose $I(A)/2$ generic points $x_1,\ldots,x_{I(A)/2}\in X$. Then $Gr(X,A)$ is a count of admissible holomorphic currents $\mc{C}$ in the homology class $A$ that pass through all of the points $x_1,\ldots,x_{I(A)/2}$. We omit the details as this case is less important for motivating the definition of ECH, although it is related to the $U$ map introduced in \S\ref{sec:addstr}. The Gromov invariants for classes $A$ with $I(A)>0$ are interesting when $b_2^+(X)=1$. However the
``simple type conjecture'' for Seiberg-Witten invariants implies that if $b_2^+(X)>1$ and $b_1(X)=0$, then $Gr(X,\omega,A)=0$ for all classes $A$ with $I(A)>0$.

\subsection{The mapping torus example}
\label{sec:mappingtorusexample}

We now compute Taubes's Gromov invariant for an interesting family of examples, namely mapping tori cross $S^1$, for $S^1$-invariant homology classes. This example will indicate what the generators of the ECH chain complex should be.

\paragraph{Mapping tori.}

Let $(\Sigma,\omega)$ be a closed connected symplectic two-manifold and let $\phi$ be a symplectomorphism from $(\Sigma,\omega)$ to itself. The {\em mapping torus\/} of $\phi$ is the three-manifold
\[
\begin{split}
Y_\phi &= [0,1]\times \Sigma/\sim,\\
(1,x) &\sim (0,\phi(x)).
\end{split}
\]
The three-manifold $Y_\phi$ fibers over $S^1={\mathbb R}/{\mathbb Z}$ with fiber $\Sigma$, and $\omega$ defines a symplectic form on each fiber. We denote the $[0,1]$ coordinate on $[0,1]\times\Sigma$ by $t$. The vector field $\partial_t$ on $[0,1]\times \Sigma$
descends to a vector field on $Y_\phi$, which we also denote by $\partial_t$. A fixed point of the map $\phi^p$ determines a periodic orbit of the vector field $\partial_t$ of period $p$, and conversely a simple periodic orbit of $\partial_t$ of period $p$ determines $p$ fixed points of $\phi^p$.

The fiberwise symplectic form $\omega$ extends to a closed 2-form on $Y_\phi$ which annihilates $\partial_t$, and which we still denote by $\omega$. We then define a symplectic form $\Omega$ on $S^1\times Y_\phi$ by
\begin{equation}
\label{eqn:Omegamt}
\Omega = ds\wedge dt + \omega
\end{equation}
where $s$ denotes the $S^1$ coordinate.

We will now calculate the Gromov invariant $Gr(S^1\times Y_\phi,\Omega,A)$, where
\[
A = [S^1]\times \Gamma\in H_2(S^1\times Y_\phi)
\]
for some $\Gamma\in H_1(Y_\phi)$. Observe to start that $I(A)=0$, so we just need to count holomorphic currents of the type described in Lemma~\ref{lem:I04}.

\paragraph{Almost complex structure.}

Choose a fiberwise $\omega$-compatible almost complex structure $J$ on the fibers of $Y_\phi\to S^1$.  That is, for each $t\in S^1={\mathbb R}/{\mathbb Z}$, choose an almost complex structure $J_t$ on the fiber over $t$, such that $J_t$ varies smoothly with $t$. Note that compatibility here just means that $J_t$ rotates positively with respect to the orientation on $\Sigma$.

The fiberwise almost complex structure extends to a unique almost complex structure $J$ on $S^1\times Y_\phi$ such that 
\begin{equation}
\label{eqn:Jst}
J\partial_s=\partial_t.
\end{equation}
It is an exercise to check that $J$ is $\Omega$-compatible.

\paragraph{Holomorphic curves.}

If $\gamma\subset Y_\phi$ is an embedded periodic orbit of $\partial_t$, then it follows from \eqref{eqn:Jst} that $S^1\times\gamma\subset S^1\times Y$ is an embedded $J$-holomorphic torus. These are all the holomorphic curves we need to consider, because of the following lemma.

\begin{lemma}
\label{lem:S1inv}
If ${\mc C}=\{(C_i,d_i)\}$ is a $J$-holomorphic current in $S^1\times Y_\phi$ with homology class $A=[S^1]\times\Gamma$, then each $C_i$ is a torus $S^1\times\gamma$ with $\gamma$ a periodic orbit of $\partial_t$.
\end{lemma}

\begin{proof}
We have $\langle A,[\omega]\rangle=0$, because the class $A$ is $S^1$-invariant while $\omega$ is pulled back via the projection to $Y_\phi$. On the other hand, by the construction of $J$, the restriction of $\omega$ to any $J$-holomorphic curve $C$ is pointwise nonnegative, with equality only where $C$ is tangent to the span of $\partial_s$ and $\partial_t$ (or singular). Thus $\int_{C_i}\omega=0$ for each $i$, and then each $C_i$ is everywhere tangent to $\partial_s$ and $\partial_t$.
\end{proof}

\paragraph{Transversality and nondegeneracy.}

We now determine when the holomorphic tori $S^1\times\gamma$ are regular.

Let $\gamma$ be a periodic orbit of period $p$, and let $x\in\Sigma$ be one of the corresponding fixed points of $\phi^p$. The fixed point $x$ of $\phi^p$ is called {\em nondegenerate\/} if the differential $d\phi^p_x:T_x\Sigma\to T_x\Sigma$ does not have $1$ as an eigenvalue. In this case, the {\em Lefschetz sign\/} is the sign of $\det(1-d\phi^p_x)$. Also, since the linear map $d\phi^p_x$ is symplectic, we can classify the fixed point $x$ as elliptic, positive hyperbolic, or negative hyperbolic according to the eigenvalues of $d\phi^p_x$, just as we did for Reeb orbits in \S\ref{sec:overview}. In particular, the Lefschetz sign is $+1$ if the fixed point is elliptic or negative hyperbolic, and $-1$ if the fixed point is positive hyperbolic. We say that the periodic orbit $\gamma$ is nondegenerate if the fixed point $x$ is nondegenerate. All of the above conditions depend only on $\gamma$ and not on the choice of corresponding fixed point $x$.

The following lemma tells us that if all periodic orbits $\gamma$ are nondegenerate (which will be the case if $\phi$ is generic), then for any $S^1$-invariant $J$, all the $J$-holomorphic tori that we need to count are regular\footnote{This is very lucky; in other $S^1$-invariant situations, obtaining transversality for $S^1$-invariant $J$ may not be possible. See e.g.\ \cite{fabert,farris} for examples of this difficulty and ways to deal with it.}.

\begin{lemma}
\label{lem:KerDC}
The $J$-holomorphic torus $C=S^1\times\gamma$ is regular
 if and only if the periodic orbit $\gamma$ is nondegenerate. In this case, the sign $\epsilon(C)$ agrees with the Lefschetz sign.
\end{lemma}

\begin{proof}
Since the deformation operator
\[
D_C:\Gamma(N_C)\longrightarrow \Gamma(T^{0,1}C\tensor N_C)
\]
has index zero, $C$ is regular if and only if $\Ker(D_C)=\{0\}$.

To determine $\Ker(D_C)$, we need to understand the deformation operator $D_C$ more explicitly.
To start, identify $N_C$ with the pullback of the normal bundle to $\gamma$ in $Y_\phi$. The latter can be identified with $T^{vert}Y_\phi|_\gamma$, where $T^{vert}Y_\phi$ denotes the vertical tangent bundle of the fiber bundle $Y_\phi\to S^1$. The linearization of the flow $\partial_t$ along $\gamma$ defines a connection $\nabla$ on the bundle $T^{vert}Y_\phi|_\gamma$. 

\begin{exercise}
\label{ex:DC}
With the above identifications, if  we use $i(ds-idt)$ to trivialize $T^{0,1}C$, then
\[
D_C = \partial_s + J\nabla_t.
\]
\end{exercise}

\begin{exercise}
\label{ex:KerDC}
(See answer in \S\ref{sec:answers}.)
Every element of $\Ker(D_C)$ is $S^1$-invariant, so $\Ker(D_C)$ is identified with the kernel of the operator
\[
\nabla_t:\Gamma(T^{vert}Y_\phi|_\gamma) \longrightarrow \Gamma(T^{vert}Y_\phi|_\gamma).
\]
\end{exercise}

\begin{exercise}
\label{ex:KerJnablat}
Let $p$ denote the period of $\gamma$ and let $x$ be a fixed point of $\phi^p$ corresponding to $\gamma$. Then there is a canonical identification
\[
\Ker(\nabla_t) = \Ker(1-d\phi^p_x).
\]
\end{exercise}

The above three exercises imply that $C$ is regular if and only if $\gamma$ is nondegenerate.

To prove that $\epsilon(C)$ agrees with the Lefschetz sign when $\gamma$ is nondegenerate, suppose first that $\gamma$ is elliptic. Then one can choose a basis for $T_x\Sigma$ in which $d\phi_x^p$ is a rotation. It follows that one can choose a trivialization of $T^{vert}Y_\phi|_\gamma$ in which the parallel transport of the connection $\nabla$ between any two points is a rotation. One can now choose $J$ to be the standard almost complex structure in this trivialization. With these choices, the operator $D_C$ is complex linear, so $\epsilon(C)=1$. The same will be true for any other choice of $J$, because one can find a path between any two almost complex structures $J$, and by the exercises above the operator $D_C$ will never have a nontrivial kernel. On the other hand, the Lefschetz sign is $+1$ in this case because the eigenvalues of $d\phi_x^p$ are complex conjugates of each other.

To prove that $\epsilon(C)$ agrees with the Lefschetz sign when $\gamma$ is not elliptic, one deforms the operator $D_C$ in an $S^1$-invariant fashion to look like the elliptic case and uses the above exercises to show that the spectral flow changes by $\pm1$ whenever one switches between the elliptic case and the positive hyperbolic case, cf.\ \cite[Lem.\ 2.6]{salamon97}.
\end{proof}

\paragraph{How to count multiple covers.}

Assume now that $\phi$ is generic so that all periodic orbits $\gamma$ are nondegenerate.
Then by the above lemmas, the Gromov invariant $Gr(S^1\times Y_\phi,\Omega,[S^1]\times\Gamma)$ counts unions of (possibly multiply covered) periodic orbits of $\partial_t$ in $Y_\phi$ with total homology class $\Gamma$. We now determine the weight with which each union of periodic orbits is counted.

For each embedded torus $C=S^1\times\gamma$, there is a generating function $f_\gamma(t)$ from \eqref{eqn:gf} encoding how its multiple covers are counted; the coefficient of $t^d$ is the number of times we count the current given by the $d$-fold cover of $C$.

\begin{lemma}
\label{lem:gf}
\[
f_\gamma(t) = \left\{\begin{array}{ll} (1-t)^{-1}=1+t+t^2+\cdots, & \mbox{$\gamma$  elliptic},\\ 1-t, & \mbox{$\gamma$ positive hyperbolic},\\ 1+t, &  \mbox{$\gamma$ negative hyperbolic}.
\end{array}
\right.
\]
\end{lemma}

\begin{proof}
To compute the generating function $f_\gamma(t)$, we need to compute the sign of $C$ (which we have already done in Lemma~\ref{lem:KerDC}) as well as the signs of the three connected double covers of $C$. Let $C_s$ denote the double cover obtained by doubling in the $s$ direction, let $C_t$ denote the double cover obtained by doubling in the $t$ direction, and let $C_{s,t}$ denote the third connected double cover. We have $\epsilon(C_s)=\epsilon(C)$, because one can compute the kernels of the operators $D_{C_s}$ and $D_C$ in the same way. After a change of coordinates, one can similarly show that $\epsilon(C_{s,t})=\epsilon(C)$. Finally $\epsilon(C_t)$ is the sign corresponding to the double cover of $\gamma$, which is positive if $\gamma$ is elliptic, and negative if $\gamma$ is positive or negative hyperbolic. So the signs are as shown in the following table:
\[
\begin{array}{c|c|c|c|c}
\gamma & \epsilon(C) & \epsilon(C_s) & \epsilon(C_{s,t}) & \epsilon(C_t)\\
\cline{1-5}
\mbox{elliptic} & +1 & +1 & +1 & +1 \\
\cline{1-5}
\mbox{positive hyperbolic} & -1 & -1 & -1 & -1 \\
\cline{1-5} 
\mbox{negative hyperbolic} & +1 & +1 & +1 & -1
\end{array}
\]
The lemma now follows from these sign calculations and \eqref{eqn:gf}.
\end{proof}

\paragraph{Conclusion.}

The above calculation shows the following:

\begin{proposition}
\label{prop:S1Y}
Let $\phi$ be a symplectomorphism of a closed connected surface $(\Sigma,\omega)$ such that all periodic orbits of $\phi$ are nondegenerate. Then $Gr(S^1\times Y_\phi,\Omega,[S^1]\times\Gamma)$ is a signed count of finite sets of pairs $\{(\gamma_i,d_i)\}$ where:
\begin{description}
\item{(i)} the $\gamma_i$ are distinct embedded periodic orbits of $\phi_t$,
\item{(ii)} the $d_i$ are positive integers,
\item{(iii)} $\sum_id_i[\gamma_i]=\Gamma\in H_1(Y)$, and
\item{(iv)} $d_i=1$ whenever $\gamma_i$ is hyperbolic.
\end{description}
The sign associated to a set $\{(\gamma_i,d_i)\}$ is $-1$ to the number of $i$ such that $\gamma_i$ is positive hyperbolic.
\end{proposition}

\begin{proof}
It follows from Lemma~\ref{lem:S1inv} that $Gr(S^1\times Y_\phi,\Omega,[S^1]\times\Gamma)$ is a count, with appropriate weights, of finite sets $\{(\gamma_i,d_i)\}$ satisfying conditions (i)--(iii). The weight associated to a set $\{(\gamma_i,d_i)\}$ is the product over $i$ of the coefficient of $t^{d_i}$ in the generating function $f_{\gamma_i}(t)$. By Lemma~\ref{lem:gf}, this weight is zero unless condition (iv) holds, in which case it is $\pm1$ and given as claimed.
\end{proof}

\subsection{Two remarks on the generating functions}
\label{sec:gf}

We now attempt to motivate the generating functions \eqref{eqn:gf} a
bit more, by explaining why they are what they are in the mapping torus example.

\paragraph{1.} One could try to define an invariant of the isotopy class of $\phi$ by counting multiple covers of tori $S^1\times\gamma$ using other generating funtions. For example, suppose
we choose generating functions $e(t)$, $h_+(t)$, and $h_-(t)$, and replace the generating functions in Lemma~\ref{lem:gf} by
\[
f_\gamma(t) = \left\{\begin{array}{ll} e(t), & \mbox{$\gamma$  elliptic},\\ h_+(t), & \mbox{$\gamma$ positive hyperbolic},\\ h_-(t), &  \mbox{$\gamma$ negative hyperbolic}.
\end{array}
\right.
\]

These generating functions must satisfy certain relations in order to give an isotopy invariant of $\phi$. First, as one isotopes $\phi$, it is possible for a bifurcation to occur in which an elliptic orbit cancels a positive hyperbolic orbit of the same period. To obtain invariance under this bifurcation, we must have
\begin{equation}
\label{eqn:bd}
e(t)h_+(t)=1.
\end{equation}
Second, a ``period-doubling'' bifurcation can occur in which an elliptic orbit turns into a negative hyperbolic orbit of the same period and an elliptic orbit of twice the period. For invariance under this bifurcation we need
\begin{equation}
\label{eqn:pd}
e(t)=h_-(t)e(t^2).
\end{equation}
In fact, any triple of generating functions $e(t)$, $h_+(t)$, and $h_-(t)$ satisfying the relations \eqref{eqn:bd} and \eqref{eqn:pd} will give rise to an invariant of the isotopy class of $\phi$.

The generating functions in Lemma~\ref{lem:gf} are $e(t)=(1-t)^{-1}$ and $h_\pm(t)=1\mp t$, which of course satisfy the relations \eqref{eqn:bd} and \eqref{eqn:pd}.
If we allowed multiply covered hyperbolic orbits also and counted them with their Lefschetz signs, then the generating functions would be $e(t)=(1-t)^{-1}$, $h_+(t)=1-t-t^2-\cdots$, and $h_-(t)=1+t-t^2+\cdots$, which do not satisfy the above relations. Throwing out all multiple covers and defining $e(t)=h_-(t)=1+t$ and $h_+(t)=1-t$ would not work either\footnote{There are of course other triples of generating functions which satisfy the above relations. For example, the Euler characteristic of the mapping torus analogue of symplectic field theory \cite{egh} (just using the $q$ variables) is computed by the generating functions
\[
\begin{split}
e(t)&=(1-t)^{-1}(1-t^2)^{-1}\cdots,\\
h_+(t) &= (1-t)(1-t^2)\cdots,\\
h_-(t) &= (1-t)^{-1}(1-t)^{-3}\cdots
\end{split}
\]
Here the omission of even powers of $(1-t)^{-1}$ in $h_-(t)$ corresponds to the omission of ``bad'' orbits, without which we would not have invariance under period doubling.}.

\paragraph{2.} Given that there are different triples of generating functions that satisfy the relations \eqref{eqn:bd} and \eqref{eqn:pd}, why is the triple in Lemma~\ref{lem:gf} the right one for determining the Seiberg-Witten invariant of $S^1\times Y_\phi$? Here is one answer: 
Let $[\Sigma]\in H_2(S^1\times Y_\phi)$ denote the homology class of a fiber of $Y_\phi\to S^1$. One can use Proposition~\ref{prop:S1Y} and the Lefschetz fixed point theorem to show that for each nonnegative integer $d$, we have
\[
\sum_{\Gamma\cdot [\Sigma] = d}Gr(S^1\times Y_\phi,\Omega,[S^1]\times\Gamma) = L(\op{Sym}^d\phi),
\]
where $\op{Sym}^d\phi$ denotes the homeomorphism from the $d^{th}$ symmetric product of $\Sigma$ to itself determined by $\phi$, and $L$ denotes the Lefschetz number.
This is what we are supposed to get, because Salamon \cite{salamon99} showed that the corresponding Seiberg-Witten invariant is a signed count of fixed points of a smooth perturbation of $\op{Sym}^d\phi$. (Similar considerations locally in a neighborhood of a holomorphic torus arise in Taubes's work in \cite{swgr} which originally led to the generating functions.)

\subsection{Three dimensional Seiberg-Witten theory}
\label{sec:SW3}

We now briefly review two basic ways to use the Seiberg-Witten equations on four-manifolds to define invariants of three-manifolds.

Let $Y$ be a closed oriented connected three-manifold. A {\em spin-c structure\/} on $Y$ can be regarded as an equivalence class of 
oriented two-plane fields (two-dimensional subbundles of $TY$), where two oriented two-plane fields are considered equivalent if they are homotopic on the complement of a ball in $Y$. The set of spin-c structures on $Y$ is an affine space over $H^2(Y;\Z)$. A spin-c structure $\frak{s}$ has a first Chern class $c_1(\frak{s})\in H^2(Y;\Z)$, and $\frak{s}$ is called ``torsion'' when $c_1(\frak{s})$ is torsion. A spin-c structure on $Y$ is equivalent to an $S^1$-invariant spin-c structure on $S^1\times Y$, or an $\R$-invariant spin-c structure on $\R\times Y$.

The first way to define invariants of $Y$ is to consider the Seiberg-Witten invariants of the four-manifold $S^1\times Y$ for $S^1$-invariant spin-c structures. These invariants are the ``Seiberg-Witten invariants'' of $Y$, which we denote by $SW(Y,\frak{s})\in\Z$, and it turns out that they count $S^1$-invariant solutions to the Seiberg-Witten equations.  Since $b_2^+(S^1\times Y)=b_1(Y)$, these invariants are well-defined\footnote{$S^1\times Y$ has a canonical homology orientation, so there is no sign ambiguity in the definition.} when $b_1(Y)>0$, up to a choice of chamber when $b_1(Y)=1$. There is also a distinguished ``zero'' chamber to use when $b_1(Y)=1$ and ${\mathfrak s}$ is not torsion.  Proposition~\ref{prop:S1Y} computed this invariant when $Y$ is a mapping torus\footnote{When $b_1(Y)=1$, we used the ``symplectic'' chamber, which disagrees with the ``zero'' chamber for spin-c structures corresponding to $\Gamma\in H_1(Y_\phi)$ with $\Gamma\cdot[\Sigma]>g(\Sigma)-1$. If $\Gamma\in H_1(Y)$ corresponds to a torsion spin-s structure then $\Gamma\cdot[\Sigma]=g(\Sigma)-1$.}. Indeed, we saw that the invariant counts $S^1$-invariant holomorphic curves.

In general, however, the Seiberg-Witten invariants of three-manifolds are not very interesting, because it was shown by Meng-Taubes \cite{meng-taubes} and Turaev \cite{turaev} that they agree with a kind of Reidemeister torsion of $Y$.

The second, more interesting way to define invariants of $Y$, constructed by Kronheimer-Mrowka \cite{km}, is to ``categorify'' the previous invariant by defining a chain complex (over ${\mathbb Z}$) whose generators are ${\mathbb R}$-invariant solutions to the Seiberg-Witten equations on ${\mathbb R}\times Y$, and whose differential counts non-${\mathbb R}$-invariant solutions to the Seiberg-Witten equations on ${\mathbb R}\times Y$ which converge to two different ${\mathbb R}$-invariant solutions as the ${\mathbb R}$-coordinate converges to $\pm\infty$. If the spin-c structure ${\mathfrak s}$ is non-torsion, then the homology of this chan complex is a well-defined invariant $HM_*(Y,{\mathfrak s})$, called ``Seiberg-Witten Floer homology'' or ``monopole Floer homology''. This is a relatively ${\mathbb Z}/d$-graded ${\mathbb Z}$-module, where $d$ denotes the divisibility of $c_1({\mathfrak s})$ in $H^2(Y;{\mathbb Z})$ mod torsion (which turns out to always be an even integer). This means that it splits into $d$ summands, and there is a well-defined grading difference in ${\mathbb Z}/d$ between any two of them, which is additive for the pairwise differences between any three summands. Each summand is finitely generated. There is also a canonical ${\mathbb Z}/2$-grading, with respect to which the Euler characteristic of the Seiberg-Witten Floer homology $HM_*(Y,{\mathfrak s})$ is the Seiberg-Witten invariant $SW(Y,{\mathfrak s})$.

If ${\mathfrak s}$ is torsion, then there is a difficulty in defining Seiberg-Witten Floer homology caused by ``reducible'' solutions to the Seiberg-Witten equations. There are two ways to resolve this difficulty, which lead to two versions of Seiberg-Witten Floer homology, which are denoted by $\widehat{HM}_*(Y,{\mathfrak s})$ and $\check{HM}_*(Y,{\mathfrak s})$. These are relatively ${\mathbb Z}$-graded; the former is zero in sufficiently negative grading, and the latter is zero in sufficiently positive grading. They fit into an exact triangle
\[
\overline{HM}_*(Y,{\mathfrak s})\to \check{HM}_*(Y,{\mathfrak s})\to\widehat{HM}_*(Y,{\mathfrak s}) \to \overline{HM}_{*-1}(Y,{\mathfrak s}))\to\cdots
\]
where $\overline{HM}_*(Y,\frak{s})$ is a third invariant which is computable in terms of the triple cup product on $Y$. In particular, $\overline{HM}_*(Y,\frak{s})$ is two-periodic, i.e.\ $\overline{HM}_*(Y,\frak{s})=\overline{HM}_{*+2}(Y,\frak{s})$, and nonzero in at least half of the gradings. In conjunction with the above exact triangle, this implies that $\widehat{HM}_*$ (resp.\ $\check{HM}_*$) is likewise 2-periodic and nontrivial when the grading is sufficiently positive (resp.\ negative). This fact is the key input from Seiberg-Witten theory to the proof of the Weinstein conjecture, see \S\ref{sec:overview}.

If ${\mathfrak s}$ is not torsion, then both $\widehat{HM}_*(Y,{\mathfrak s})$ and $\check{HM}_*(Y,{\mathfrak s})$ are equal to the invariant $HM_*(Y,{\mathfrak s})$ discussed previously. 

\subsection{Towards ECH}

The original motivation for defining ECH was to find an analogue of Taubes's $SW=Gr$ theorem for a three-manifold. That is, we would like to identify Seiberg-Witten Floer homology with an appropriate analogue of Taubes's Gromov invariant for a three-manifold $Y$. The latter should be the homology of a chain complex which is generated by ${\mathbb R}$-invariant holomorphic curves in ${\mathbb R}\times Y$, and whose differential counts non-${\mathbb R}$-invariant holomorphic curves in ${\mathbb R}\times Y$.

For holomorphic curve counts to make sense, ${\mathbb R}\times Y$ should have a symplectic structure. This is the case for example when $Y$ is the mapping torus of a symplectomorphism $\phi$; the symplectic form \eqref{eqn:Omegamt} on $S^1\times Y_\phi$ also makes sense on ${\mathbb R}\times Y_\phi$. The analogue of Taubes's Gromov invariant in this case is the ``periodic Floer homology'' of $\phi$; it is the homology of a chain complex which is generated by the unions of periodic orbits counted in Proposition~\ref{prop:S1Y}, and its differential counts certain holomorphic curves in ${\mathbb R}\times Y$. The definition of PFH is given in \cite{pfh2,pfh3}, and it shown in \cite{lee-taubes} that PFH agrees with Seiberg-Witten Floer homology.

Which holomorphic curves to count in the PFH differential is a subtle matter which we will explain below.
However, since not every three-manifold is a mapping torus, we will instead carry out the analogous construction of ECH for contact three-manifolds\footnote{To spell out the analogy here, both mapping tori and contact structures
are examples of the more general notion of ``stable Hamiltonian structure''. A {\em stable Hamiltonian structure\/} on an oriented 3-manifold consists of a $1$-form $\lambda$ and a closed $2$-form $\omega$ such that $\lambda\wedge\omega>0$ and $d\lambda=f\omega$ with $f:Y\to\R$. These data determine an oriented 2-plane field $\xi=\Ker(\lambda)$ and a ``Reeb vector field'' $R$ characterized by $\omega(R,\cdot)=0$ and $\lambda(R)=1$. For a mapping torus, $\lambda=dt$, $\omega\equiv 0$, $f\equiv 0$, and $R=\partial_t$. For a contact structure, $\omega=d\lambda$, $f\equiv 1$, and $R$ is the usual Reeb vector field. A version of ECH for somewhat more general stable Hamiltonian structures with $f\ge 0$ appears in the work of Kutluhan-Lee-Taubes \cite{klt2}.
}, which is more general since every oriented three-manifold admits a contact structure. 
Finding the appropriate definition of the ECH chain complex is not obvious, but Taubes's $SW=Gr$ theorem and the computation of $Gr$ for mapping tori give us a lot of hints. 

\section{The definition of ECH}
\label{sec:defech}

Guided by the discussion in \S\ref{sec:origins}, we now define the embedded contact homology of a contact three-manifold $(Y,\lambda)$, using $\Z/2$ coefficients for simplicity.

Assume that $\lambda$ is nondegenerate and fix $\Gamma\in H_1(Y)$. 
We wish to define the chain complex $ECC_*(Y,\lambda,\Gamma,J)$, where $J$ is a generic symplectization-admissible almost complex structure on $\R\times Y$, see \S\ref{sec:overview}.

Define an {\em orbit set\/} in the homology class $\Gamma$ to be a finite set of pairs $\{(\alpha_i,m_i)\}$ where the $\alpha_i$ are distinct embedded Reeb orbits, the $m_i$ are positive integers, and $\sum_im_i[\alpha_i]=\Gamma\in H_1(Y)$.
Motivated by Proposition~\ref{prop:S1Y}, we define the the chain complex to be generated by orbit sets as above such that $m_i=1$ whenever $\alpha_i$ is hyperbolic.  (We also need to study orbit sets not satisfying this last condition in order to develop the theory.) Proposition~\ref{prop:S1Y} also suggests that there should be a canonical $\Z/2$-grading by the parity of the number of $i$ such that $\alpha_i$ is positive hyperbolic, and we will see in \S\ref{sec:differential} that this is the case.

The differential should count $J$-holomorphic currents in $\R\times Y$ by analogy with the Gromov invariant. The three key formulas that entered into the definition of the Gromov invariant were the Fredholm index formula \eqref{eqn:ind4}, the adjunction formula \eqref{eqn:adj4}, and the definition of $I$ in \eqref{eqn:I4}. To define the ECH differential we need analogues of these three formulas for holomorphic curves in $\R\times Y$, plus one additional ingredient, the ``writhe bound''. We now explain these.

\subsection{Holomorphic curves and holomorphic currents}

We consider $J$-holomorphic curves of the form $u:(\Sigma,j)\to({\mathbb R}\times Y,J)$ where the domain $(\Sigma,j)$ is a punctured compact Riemann surface. If $\gamma$ is a (possibly multiply covered) Reeb orbit, a {\em positive end\/} of $u$ at $\gamma$ is a puncture near which $u$ is asymptotic to ${\mathbb R}\times\gamma$ as $s\to\infty$. This means that a neighborhood of the puncture can be given coordinates $(\sigma,\tau)\in ({\mathbb R}/T{\mathbb Z})\times[0,\infty)$ with $j(\partial_\sigma)=\partial_\tau$ such that $\lim_{\sigma\to\infty}\pi_{\mathbb R}(u(\sigma,\tau))=\infty$ and $\lim_{\sigma\to\infty}\pi_Y(u(s,\cdot))=\gamma$. A {\em negative end\/} is defined analogously with $\sigma\in(-\infty,0]$ and $s\to -\infty$. We assume that all punctures are positive ends or negative ends as above. We mod out by the usual equivalence relation on holomorphic curves, namely composition with biholomorphic maps between domains.

Let $\alpha=\{(\alpha_i,m_i)\}$ and $\beta=\{(\beta_j,n_j)\}$ be orbit sets in the class $\Gamma$.  
Define a {\em $J$-holomorphic current\/} from $\alpha$ to $\beta$ to be a finite set of pairs $\mc{C}=\{(C_k,d_k)\}$ where the $C_k$ are distinct irreducible somewhere injective $J$-holomorphic curves in $\R\times Y$, the $d_k$ are positive integers, $\mc{C}$ is asymptotic to $\alpha$ as a current as the $\R$ coordinate goes to $+\infty$, and $\mc{C}$ is asymptotic to $\beta$ as a current as the $\R$ coordinate goes to $-\infty$. This last condition means that the positive ends of the curves $C_k$ are at covers of the Reeb orbits $\alpha_i$, the sum over $k$ of $d_k$ times the total covering multiplicity of all ends of $C_k$ at covers of $\alpha_i$ is $m_i$, and analogously for the negative ends.
Let $\mc{M}(\alpha,\beta)$ denote the set of $J$-holomorphic currents from $\alpha$ to $\beta$. 
A holomorphic current $\mc{C}=\{(C_k,d_k)\}$ is ``somewhere injective'' if $d_k=1$ for each $k$, in which case it is ``embedded'' if furthermore each $C_k$ is embedded and the $C_k$ are pairwise disjoint.

Let $H_2(Y,\alpha,\beta)$ denote the set of $2$-chains $\Sigma$ in $Y$ with
\[
\partial\Sigma=\sum_im_i\alpha_i-\sum_jn_j\beta_j,
\]
modulo boundaries of $3$-chains. Then $H_2(Y,\alpha,\beta)$ is an affine space over $H_2(Y)$, and every $J$-holomorphic current $\mc{C}\in{\mathcal M}(\alpha,\beta)$ defines a class $[\mc{C}]\in H_2(Y,\alpha,\beta)$.

\subsection{The Fredholm index in symplectizations}
\label{sec:ind}

We now state a symplectization analogue of the index formula \eqref{eqn:ind}.

\begin{proposition}
\label{prop:ind}
If $J$ is generic, then every somewhere injective $J$-holomorphic curve $C$ in $\R\times Y$ is regular (i.e.\ an appropriate deformation operator is surjective), so the moduli space of $J$-holomorphic curves as above near $C$ is a manifold. Its dimension is the Fredholm index given by equation \eqref{eqn:ind3} below.
\end{proposition}

If $C$ has $k$ positive ends at Reeb orbits $\gamma_1^+,\ldots,\gamma_k^+$ and $l$ negative ends at Reeb orbits $\gamma_1^-,\ldots,\gamma_l^-$, the {\em Fredholm index\/} of $C$ is defined by
\begin{equation}
\label{eqn:ind3}
\ind(C) = -\chi(C) + 2c_\tau(C) + \sum_{i=1}^k CZ_\tau(\gamma_i^+)-\sum_{i=1}^l CZ_\tau(\gamma_i^-),
\end{equation}
where the terms on the right hand side are defined as follows. First, $\tau$ is a trivialization of $\xi$ over the Reeb orbits $\gamma_i^\pm$, which is symplectic with respect to $d\lambda$. Second, $\chi(C)$ denotes the Euler characteristic of the domain of $C$ as usual.
Third,
\[
c_\tau(C) = c_1(\xi|_C,\tau)\in\Z
\]
is the {\em relative first Chern class\/} of the complex line bundle $\xi|_C$ with respect to the trivialization $\tau$. To define this, note that the trivialization $\tau$ determines a trivialization of $\xi|_C$ over the ends of $C$, up to homotopy. One chooses a generic section $\psi$ of $\xi|_C$ which on each end is nonvanishing and constant with respect to the trivialization on the ends. One then defines $c_1(\xi|_C,\tau)$ to be the algebraic count of zeroes of $\psi$.

To say more about what the relative first Chern class depends on, note that $C\in \mc{M}(\alpha,\beta)$ for some orbit sets $\alpha=\{(\alpha_i,m_i)\}$ and $\beta=\{(\beta_j,n_j)\}$ in the same homology class. Write $Z=[C]\in H_2(Y,\alpha,\beta)$. Then in fact $c_1(\xi|_C,\tau)$ depends only on $\alpha$, $\beta$, $\tau$, and $Z$. To see this, let $S$ be a compact oriented surface with boundary, and let $f:S\to [-1,1]\times Y$ be a smooth map, such that $f|_{\partial S}$ consists of positively oriented covers of $\{1\}\times \alpha_i$ with total multiplicity $m_i$ and negatively oriented covers of $\{-1\}\times\beta_j$ with total multiplicity $n_j$, and the projection of $f$ to $Y$ represents the relative homology class $Z$. Then $c_1(f^*\xi,\tau)\in\Z$ is defined as before.

\begin{exercise}
\begin{description}
\item{(a)}
The relative first Chern class
$c_1(f^*\xi,\tau)$ above depends only on $\alpha$, $\beta$, $\tau$, and $Z$, and so can be denoted by $c_\tau(Z)$.
\item{(b)}
If $Z'\in H_2(Y,\alpha,\beta)$ is another relative homology class, then
\[
c_\tau(Z) - c_\tau(Z') = \langle c_1(\xi),Z-Z'\rangle,
\]
where on the right hand side, $c_1(\xi)\in H^2(Y;\Z)$ denotes the usual first Chern class of the complex line bundle $\xi\to Y$.
\end{description}
\end{exercise}

Continuing with the explanation of the index formula \eqref{eqn:ind3},
$CZ_\tau(\gamma)\in{\mathbb Z}$ denotes the {\em Conley-Zehnder index\/} of $\gamma$ with respect to the trivialization $\tau$. To define this, pick a parametrization $\gamma:{\mathbb R}/T{\mathbb Z}\to Y$. Let $\{\psi_t\}_{t\in{\mathbb R}}$ denote the one-parameter group of diffeomorphisms of $Y$ given by the flow of $R$. Then $d\psi_t:T_{\gamma(0)}Y\to T_{\gamma(t)}Y$ induces a symplectic linear map $\phi_t:\xi_{\gamma(0)}\to\xi_{\gamma(t)}$, which using our trivialization $\tau$ we can regard as a $2\times 2$ symplectic matrix. In particular, $\phi_0=1$, and $\phi_T$ is the linearized return map (in our trivialization), which does not have 1 as an eigenvalue. We now define $CZ_\tau(\gamma)\in{\mathbb Z}$ to be the Conley-Zehnder index of the family of symplectic matrices $\{\phi_t\}_{t\in[0,T]}$, which is given explicitly as follows. (See e.g.\ \cite[\S2.4]{salamon97} for the general definition of the Conley-Zehnder index for paths of symplectic matrices in any dimension.)

 If $\gamma$ is hyperbolic, let $v\in{\mathbb R}^2$ be an eigenvector of $\phi_T$; then the family of vectors $\{\phi_t(v)\}_{t\in[0,T]}$ rotates by angle $\pi k$ for some integer $k$ (which is even in the positive hyperbolic case and odd in the negative hyperbolic case), and
\[
CZ_\tau(\gamma)=k.
\]
If $\gamma$ is elliptic, then we can change the trivialization so that each $\phi_t$ is rotation by angle $2\pi \theta_t\in{\mathbb R}$ where $\theta_t$ is a continuous function of $t\in[0,T]$ and $\theta_0=0$. The number $\theta=\theta_T\in{\mathbb R}\setminus{\mathbb Z}$ is called the ``rotation angle'' of $\gamma$ with respect to $\tau$, and
\begin{equation}
\label{eqn:CZell}
CZ_\tau(\gamma)=2\lfloor\theta\rfloor+1.
\end{equation}

\begin{exercise}
\label{ex:ind3}
The right hand side of the index formula \eqref{eqn:ind3} does not depend on $\tau$, even though the individual terms in it do. (See hint in \S\ref{sec:answers}.)
\end{exercise}

The proof of Proposition~\ref{prop:ind} consists of a tranversality argument in \cite{dragnev} and an index calculation in \cite{schwarz}. As usual, the somewhere injective assumption is necessary; there is no $J$ for which transversality holds for all multiply covered curves. For example, transversality fails for some branched covers of trivial cylinders, see Exercise~\ref{ex:partialorder} below.

\subsection{The relative adjunction formula}

Our next goal is to obtain an analogue of the adjunction formula \eqref{eqn:adj4} for a somewhere injective holomorphic curve in $\R\times Y$. To do so we need to re-interpret each term in the formula \eqref{eqn:adj4} in the symplectization context; and there is also a new term arising from the asymptotic behavior of the holomorphic curve.

\begin{raf}
\cite[Rmk.\ 3.2]{pfh2}
Let $C\in\mc{M}(\alpha,\beta)$ be somewhere injective. Then $C$ has only finitely many singularities, and
\begin{equation}
\label{eqn:adj3}
c_\tau(C) = \chi(C) + Q_\tau(C) + w_\tau(C) - 2\delta(C).
\end{equation}
\end{raf}

Here $\tau$ is a trivialization of $\xi$ over the Reeb orbits $\alpha_i$ and $\beta_j$;  the left hand side is the relative first Chern class defined in \S\ref{sec:ind}; $\chi(C)$ is the Euler characteristic of the domain as usual; and $\delta(C)\ge 0$ is an algebraic count of singularities with positive integer weights as in \S\ref{sec:4dspecial}. The term $Q_\tau(C)$ is the ``relative intersection pairing'', which is a symplectization analogue of the intersection number $[C]\cdot[C]$ in the closed case. The new term $w_\tau(C)$ is the ``asymptotic writhe''. Let us now explain both of these. 

\paragraph{The relative intersection pairing.}

Given a class $Z\in H_2(Y,\alpha,\beta)$, we want to define the relative intersection pairing $Q_\tau(Z)\in{\mathbb Z}$.

To warm up to this, recall that given a closed oriented 4-manifold $X$, and given a class $A\in H_2(X)$, to compute $A\cdot A$ one can choose two embedded oriented surfaces $S,S'\subset X$ representing the class $A$ that intersect transversely, and count the intersections of $S$ and $S'$ with signs.

In the symplectization case, we could try to choose two embedded (except at the boundary) oriented surfaces $S,S'\subset[-1,1]\times Y$ representing the class $Z$ such that
\[
\partial S = \partial S' = \sum_im_i\{1\}\times \alpha_i-\sum_jn_j\{-1\}\times\beta_j,
\]
and $S$ and $S'$ intersect transversely (except at the boundary), and algebraically count intersections of the interior of $S$ with the interior of $S'$. However this count of intersections is not a well-defined function of $Z$, because if one deforms $S$ or $S'$, then intersection points can appear or disappear at the boundary.

To get a well-defined count of intersections, we need to specify something about the boundary behavior. The choice of trivialization $\tau$ allows us to do this. We require that the projections to $Y$ of the intersections of $S$ and $S'$ with $(1-\epsilon,1]\times Y$ are embeddings, and their images in a transverse slice to $\alpha_i$ are unions of rays which do not intersect and which do not rotate with respect to the trivialization $\tau$ as one goes around $\alpha_i$.  Likewise, 
the projections to $Y$ of the intersections of $S$ and $S'$ with $[-1,-1+\epsilon)\times Y$ are embeddings, and their images in a transverse slice to $\beta_j$ are unions of rays which do not intersect and which do not rotate with respect to the trivialization $\tau$ as one goes around $\beta_j$.
If we count the interior intersections of two such surfaces $S$ and $S'$, then we get an integer which depends only on $\alpha,\beta,Z$, and $\tau$, and we denote this integer by $Q_\tau(Z)$. For more details see \cite[\S2.4]{pfh2} and \cite[\S2.7]{ir}.

If $\mc{C}\in{\mathcal M}(\alpha,\beta)$ is a $J$-holomorphic current, write $Q_\tau(\mc{C})=Q_\tau([\mc{C}])$.

\paragraph{The asymptotic writhe.}

Given a somewhere injective $J$-holomorphic curve $C\in{\mathcal M}(\alpha,\beta)$, consider the slice $C\cap(\{s\}\times Y)$. If $s>>0$, then the slice $C\cap(\{s\}\times Y)$ is an embedded curve which is the union, over $i$, of a braid $\zeta_i^+$ around the Reeb orbit $\alpha_i$ with $m_i$ strands. This fact, due to Siefring \cite{siefring1}, is shown along the way to proving the writhe bound \eqref{eqn:writhebound} below, see Lemma~\ref{lem:siefring}. This, together with an analogous statement for the negative ends and the fact that the singularities of $C$ are isolated, implies that $C$ has only finitely many singularities. Since the braid $\zeta_i^+$ is embedded for all $s>>0$, its isotopy class does not depend on $s>>0$.

We can use the trivialization $\tau$ to identify the braid $\zeta_i^+$ with a link in $S^1\times D^2$. The writhe of this link, which we denote by $w_\tau(\zeta_i^+)\in\Z$, is defined by identifying $S^1\times D^2$ with an annulus cross an interval, projecting $\zeta_i^+$ to the annulus, and counting crossings with signs. We use the sign convention in which counterclockwise rotations in the $D^2$ direction as one goes counterclockwise around $S^1$ contribute positively to the writhe; this is opposite the usual convention in knot theory, but makes sense in the present context.

Likewise, the slice $C\cap(\{s\}\times Y)$ for $s<<0$ is the union over $j$ of a braid $\zeta_j^-$ around the Reeb orbit $\beta_j$ with $n_j$ strands, and this braid has a writhe $w_\tau(\zeta_j^-)\in\Z$.

We now define the {\em asymptotic writhe\/} of $C$ by
\[
w_\tau(C) = \sum_iw_\tau(\zeta_i^+) - \sum_jw_\tau(\zeta_j^-).
\]
This completes the definition of all of the terms in the relative adjunction formula \eqref{eqn:adj3}.

\begin{exercise}
\label{ex:adj3}
Show that the two sides of the relative adjunction formula \eqref{eqn:adj3} change the same way if one changes the trivialization $\tau$. (See hint in \S\ref{sec:answers}.)
\end{exercise}

Here is an outline of the proof of the relative adjunction formula \eqref{eqn:adj3} in the special case where $C$ is immersed and the only singularities of $C$ are nodes. Let $N_C$ denote the normal bundle of $C$, which can be identified with $\xi|_C$ near the ends of $C$. We compute $c_1(N_C,\tau)$ in two ways.
First, the decomposition $({\mathbb C}\oplus\xi)|_C = T(\R\times Y)|_C = TC\oplus N_C$ implies that
\[
c_\tau(C) = \chi(C) + c_1(N_C,\tau),
\]
see \cite[Prop.\ 3.1(a)]{pfh2}. Second, one can count the intersections of $C$ with a nearby surface and compare with the definition of $Q_\tau$ to show that
\[
c_1(N_C,\tau) = Q_\tau(C) + w_\tau(C)-2\delta(C),
\]
cf.\ \cite[Prop.\ 3.1(b)]{pfh2}.

\subsection{The ECH index}
\label{sec:ECHindex}

We come now to the key nontrivial part of the definition in ECH, which is to define an analogue of the quantity $I$ in \eqref{eqn:I4} for relative homology classes in symplectizations.

Let $C\in\mc{M}(\alpha,\beta)$ be somewhere injective. By \eqref{eqn:ind3}, we can write the Fredholm index of $C$ as
\[
\ind(C) = -\chi(C) + 2c_\tau(C) + CZ_\tau^{ind}(C),
\]
where $CZ_\tau^{ind}(C)$ is shorthand for the Conley-Zehnder term that appears in $\ind$, namely the sum over all positive ends of $C$ at a Reeb orbit $\gamma$ of $CZ_\tau(\gamma)$ (these Reeb orbits are covers of the Reeb orbits $\alpha_i$), minus the corresponding sum for the negative ends of $C$. We know that if $J$ is generic then ${\mathcal M}(\alpha,\beta)$ is a manifold near $C$ of dimension $\ind(C)$. We would like to bound this dimension in terms of the relative homology class $[C]$.

If $\gamma$ is an embedded Reeb orbit and $k$ is a positive integer, let $\gamma^k$ denote the $k$-fold iterate of $\gamma$.

\begin{definition}
If $Z\in H_2(Y,\alpha,\beta)$, define the {\em ECH index\/}
\begin{equation}
\label{eqn:I3}
I(\alpha,\beta,Z) = c_\tau(Z) + Q_\tau(Z) + CZ_\tau^I(\alpha,\beta),
\end{equation}
where $CZ_\tau^I$ is the Conley-Zehnder term that appears in $I$, namely
\begin{equation}
\label{eqn:CZI}
CZ_\tau^I(\alpha,\beta) = \sum_i\sum_{k=1}^{m_i}CZ_\tau(\alpha_i^k) - \sum_j\sum_{k=1}^{n_j}CZ_\tau(\beta_j^k).
\end{equation}
If $C\in{\mathcal M}(\alpha,\beta)$, define $I(C)=I(\alpha,\beta,[C])$.
\end{definition}

Note that the Conley-Zehnder terms $CZ_\tau^{ind}(C)$ and $CZ_\tau^I(\alpha,\beta)$ are quite different.  The former just involves the Conley-Zehnder indices of orbits corresponding to ends of $C$; while the latter sums up the Conley-Zehnder indices of all iterates of $\alpha_i$ up to multiplicity $m_i$, minus the Conley-Zehnder indices of all iterates of $\beta_j$ up to multiplicity $n_j$. For example, if $C$ has positive ends at $\alpha_i^3$ and $\alpha_i^5$ (and no other positive ends at covers of $\alpha_i$), then the corresponding contribution to $CZ_\tau^{ind}(C)$ is $CZ_\tau(\alpha_i^3)+CZ_\tau(\alpha_i^5)$, while the contribution to $CZ_\tau^I(\alpha,\beta)$ is $\sum_{k=1}^8CZ_\tau(\alpha_i^k)$.

\begin{bpechi}
\begin{description}
\item{(Well Defined)} The ECH index $I(Z)$ does not depend on the choice of trivialization $\tau$.
\item{(Index Ambiguity Formula)} If $Z'\in H_2(\alpha,\beta)$ is another relative homology class, then
\begin{equation}
\label{eqn:iaf}
I(Z) - I(Z') = \langle Z-Z',c_1(\xi) + 2\op{PD}(\Gamma)\rangle.
\end{equation}
\item{(Additivity)} If $\delta$ is another orbit set in the homology class $\Gamma$, and if $W\in H_2(Y,\beta,\delta)$, then $Z+W\in H_2(Y,\alpha,\delta)$ is defined and
\[
I(Z+W)=I(Z)+I(W).
\]
\item{(Index Parity)}
If $\alpha$ and and $\beta$ are generators of the ECH chain complex (i.e.\ all hyperbolic orbits have multiplicity $1$), then
\begin{equation}
\label{eqn:indexparity}
(-1)^{I(Z)} = \epsilon(\alpha)\epsilon(\beta),
\end{equation}
where $\epsilon(\alpha)$ denotes $-1$ to the number of positive hyperbolic orbits in $\alpha$.
\end{description}
\end{bpechi}

\begin{exercise}
Prove the above basic properties. (See \cite[\S3.3]{pfh2}.)
\end{exercise}

We now have the following analogue of \eqref{eqn:ie4}, which is the key result that gets ECH off the ground.

\begin{indexinequality}
If $C\in{\mathcal M}(\alpha,\beta)$ is somewhere injective, then
\begin{equation}
\label{eqn:ii}
\ind(C)\le I(C)-2\delta(C).
\end{equation}
\end{indexinequality}

\noindent
In particular, $\ind(C)\le I(C)$, with equality only if $C$ is embedded.

The index inequality follows immediately by combining the definition of the ECH index in \eqref{eqn:I3}, the formula for the Fredholm index in \eqref{eqn:ind3}, the relative adjunction formula \eqref{eqn:adj3}, and the following inequality:

\begin{writhebound}
If $C\in\mc{M}(\alpha,\beta)$ is somewhere injective, then
\begin{equation}
\label{eqn:writhebound}
w_\tau(C) \le CZ_\tau^I(\alpha,\beta) - CZ_\tau^{ind}(C).
\end{equation}
\end{writhebound}

\noindent
The proof of the writhe bound will be outlined in \S\ref{sec:writhebound}.

\paragraph{Holomorphic curves with low ECH index}

The index inequality \eqref{eqn:ii} is most of what is needed to prove the following analogue of Lemma~\ref{lem:I04}. Below, a {\em trivial cylinder\/} means a cylinder ${\mathbb R}\times \gamma\subset{\mathbb R}\times Y$ where $\gamma$ is an embedded Reeb orbit.

\begin{proposition}
\label{prop:I03}
Suppose $J$ is generic. Let $\alpha$ and $\beta$ be orbit sets and let $\mc{C}\in{\mathcal M}(\alpha,\beta)$ be any $J$-holomorphic current in $\R\times Y$, not necessarily somewhere injective. Then:
\begin{description}
\item{0.} $I(\mc{C}) \ge 0$, with equality if and only if $\mc{C}$ is a union of trivial cylinders with multiplicities.
\item{1.} If $I(\mc{C})=1$, then $\mc{C}=\mc{C}_0\sqcup C_1$, where $I(\mc{C}_0)=0$, and $C_1$ is embedded and has $\ind(C_1)=I(C_1)=1$.
\item{2.} If $I(\mc{C})=2$, and if $\alpha$ and $\beta$ are generators of the chain complex $ECC_*(Y,\lambda,\Gamma,J)$, then $\mc{C}=\mc{C}_0\sqcup C_2$, where $I(\mc{C}_0)=0$, and $C_2$ is embedded and has $\ind(C_2)=I(C_2)=2$.
\end{description}
\end{proposition}

\begin{proof}
Let $\mc{C}=\{(C_k,d_k)\}$ be a holomorphic current in ${\mathcal M}(\alpha,\beta)$. We first consider the special case in which $d_k=1$ whenever $C_k$ is a trivial cylinder.

Since $J$ is $\R$-invariant, any $J$-holomorphic curve can be translated in the ${\mathbb R}$-direction to make a new $J$-holomorphic curve. Let $C'$ be the union over $k$ of the union of $d_k$ different translates of $C_k$. Then $C'$ is somewhere injective, thanks to our simplifying assumption that $d_k=1$ whenever $C_k$ is a trivial cylinder. So the index inequality applies to $C'$ to give
\[
\ind(C') \le I(C') - 2\delta(C').
\]
Now because the Fredholm index $\ind$ is additive under taking unions of holomorphic curves, and because the ECH index $I$ depends only on the relative homology class, this gives
\begin{equation}
\label{eqn:readoff}
\sum_kd_k\ind(C_k) \le I(\mathcal{C}) - 2\delta(C').
\end{equation}
Since $J$ is generic, we must have $\ind(C_k)\ge 0$, with equality if and only if $C_k$ is a trivial cylinder. Parts (0) and (1) of the Proposition can now be immediately read off from the inequality \eqref{eqn:readoff}.

To prove part (2), we just need to rule out the case where there is one nontrivial $C_k$ with $d_k=2$. In this case, since $\alpha$ and $\beta$ are ECH generators, all ends of $C_k$ must be at elliptic Reeb orbits. It then follows from the Fredholm index formula \eqref{eqn:ind3} that $\ind(C_k)$ is even. Thus $\ind(C_k)\ge 2$, contradicting the inequality \eqref{eqn:readoff}.

To remove the simplifying assumption, one can show that if $\mc{C}$ contains no trivial cylinders and if $\mc{T}$ is a union of trivial cylinders, then
\[
I(\mc{C}\cup \mc{T})\ge I(\mc{C}) + 2\#(\mc{C}\cap\mc{T}),
\]
compare \eqref{eqn:Isum}. This is proved in \cite[Prop.\ 7.1]{pfh2}, and a more general statement bounding the ECH index of any union of holomorphic currents is proved in \cite[Thm.\ 5.1]{ir}. Now by intersection positivity, $\#(\mc{C}\cap\mc{T})\ge 0$, with equality if and only if $\mc{C}$ and $\mc{T}$ are disjoint. The proposition for $\mc{C}\cup\mc{T}$ then follows from the proposition for $\mc{C}$.
\end{proof}

\subsection{The ECH differential}
\label{sec:differential}

We can now define the differential $\partial$ on the chain complex $ECC_*(Y,\lambda,\Gamma,J)$.  If $\alpha$ and $\beta$ are orbit sets and $k$ is an integer, define
\[
\mc{M}_k(\alpha,\beta) = \{\mc{C}\in\mc{M}(\alpha,\beta)\mid I(\mc{C})=k\}.
\]
If $\alpha$ is a chain complex generator, we define
\[
\partial\alpha =\sum_\beta\#(\mc{M}_1(\alpha,\beta)/\R)\beta,
\]
where the sum is over chain complex generators $\beta$, and `$\#$' denotes the mod 2 count. Here $\R$ acts on $\mc{M}_1(\alpha,\beta)$ by translation of the $\R$ coordinate on $\R\times Y$; and by Proposition~\ref{prop:I03} the quotient is a discrete set.
We will show in \S\ref{sec:differentialdefined}, analogously to Lemma~\ref{lem:finite4}, that $\mc{M}_1(\alpha,\beta)/\R$ is finite so that the count $\#(\mc{M}_1(\alpha,\beta)/\R)$ is well defined. Next, it follows from the inequality \eqref{eqn:daf} and Exercise~\ref{ex:fro} below that for any $\alpha$, there are only finitely many $\beta$ with $\mc{M}(\alpha,\beta)$ nonempty, so $\partial\alpha$ is well defined.

\begin{exercise}
\label{ex:fro}
If $\lambda$ is a nondegenerate contact form on $Y$ and if $L$ is a real number, then $\lambda$ has only finitely many Reeb orbits with symplectic action less than $L$.
\end{exercise}

The proof that $\partial^2=0$ is much more difficult, and we will give an introduction to this in \S\ref{sec:200pages}. Modulo this and the other facts we have not proved, we have now defined $ECH_*(Y,\lambda,\Gamma,J)$, and as reviewed in the introduction this is an invariant $ECH_*(Y,\xi,\Gamma)$.

\subsection{The grading}
\label{sec:grading}

The chain complex $ECC_*(Y,\lambda,\Gamma,J)$, and hence its homology, is relatively ${\mathbb Z}/d$ graded, where $d$ denotes the divisibility of $c_1(\xi)+2\op{PD}(\Gamma)$ in $H^2(Y;{\mathbb Z})$ mod torsion. That is, if $\alpha$ and $\beta$ are two chain complex generators, we can define their ``index difference'' $I(\alpha,\beta)$ by choosing an arbitrary $Z\in H_2(Y,\alpha,\beta)$ and setting
\[
I(\alpha,\beta) = \left[I(\alpha,\beta,Z)\right] \in {\mathbb Z}/d.
\]
This is well defined by the index ambiguity formula \eqref{eqn:iaf}. When the chain complex is nonzero, we can further define an absolute $\Z/d$ grading by picking some generator $\beta$ and declaring its grading to be zero, so that the grading of any other generator is $\alpha$ is
\[
|\alpha| = I(\alpha,\beta).
\]
By the Additivity property of the ECH index, the differential decreases this absolute grading by $1$.

\begin{remarks}
(1) In particular, if $\Gamma=0$, then the empty set of Reeb orbits is a generator of the chain complex, which represents a homology class depending only on $Y$ and $\xi$, see \S\ref{sec:addstr}. Thus  $ECH_*(Y,\xi,0)$ has a canonical absolute ${\mathbb Z}/d$ grading in which the empty set has grading zero.

(2) It follows from the Index Parity property \eqref{eqn:indexparity} that for every $\Gamma$ there is a canonical absolute $\Z/2$ grading on $ECH_*(Y,\xi,\Gamma)$ by the parity of the number of positive hyperbolic Reeb orbits.
\end{remarks}

\subsection{Example: the ECH of an ellipsoid}
\label{sec:ellipsoid}

To illustrate the above definitions, we now compute $ECH_*(Y,\lambda,0,J)$, where $Y$ is the three-dimensional ellipsoid $Y=\partial E(a,b)$ with $a/b$ irrational, and $\lambda$ is the contact form given by the restriction of the Liouville form \eqref{eqn:lambdastd}. We already saw in Example~\ref{ex:ell1} that the chain complex generators have the form $\gamma_1^{m_1}\gamma_2^{m_2}$ with $m_1,m_2\ge 0$. Since the Reeb orbits $\gamma_1$ and $\gamma_2$ are elliptic, it follows from the Index Parity property \eqref{eqn:indexparity} that the grading difference between any two generators is even, so the differential vanishes identically for any $J$.

\paragraph{The grading.}

To finish the computation of the homology, we just need to compute the grading of each generator.
We know from \S\ref{sec:grading} that the chain complex has a canonical $\Z$-grading, where the empty set (corresponding to $m_1=m_2=0$) has grading zero. The grading of $\alpha=\gamma_1^{m_1}\gamma_2^{m_2}$ can then be written as
\begin{equation}
\label{eqn:ellipsoidgrading}
|\alpha| = 
I(\alpha,\emptyset) = c_\tau(\alpha) + Q_\tau(\alpha) +
CZ_\tau^I(\alpha).
\end{equation}
Here $c_\tau(\alpha)$ is shorthand for $c_\tau(Z)$, and $Q_\tau(\alpha)$ is shorthand for $Q_\tau(Z)$, where $Z$ is the unique element of $H_2(Y,\alpha,\emptyset)$; and $CZ_\tau^I(\alpha)$ is shorthand for $CZ_\tau^I(\alpha,\emptyset)$.

To calculate the terms on the right hand side of \eqref{eqn:ellipsoidgrading}, we first need to choose a trivialization $\tau$ of $\xi$ over $\gamma_1$ and $\gamma_2$. Under the identification $T\R^4=\C\oplus\C$, the restriction of $\xi$ to $\gamma_1$ agrees with the second $\C$ summand, and the restriction of $\xi$ to $\gamma_2$ agrees with the first $\C$ summand. We use these identifications to define the trivialization $\tau$ that we will use.

The calculations in Example~\ref{ex:ell1} imply that with respect to this trivialization $\tau$, the rotation angle (see \S\ref{sec:ind}) of $\gamma_1$ is $a/b$, and the rotation angle of $\gamma_2$ is $b/a$. So by the formula \eqref{eqn:CZell} for the Conley-Zehnder index, we have
\[
CZ_\tau^I(\alpha) = \sum_{k=1}^{m_1}\left(2\lfloor ka/b\rfloor + 1\right) + \sum_{k=1}^{m_2}\left(2\lfloor kb/a\rfloor + 1\right).
\]
The remaining terms in \eqref{eqn:ellipsoidgrading} are given as follows:

\begin{exercise}
\label{ex:ellcq}
$c_\tau(\alpha) = m_1+m_2$, and $Q_\tau(\alpha) = 2m_1m_2$.
\end{exercise}

Putting the above together, we get that
\begin{equation}
\label{eqn:Iell}
I(\alpha) = 2\bigg((m_1+1)(m_2+1)-1 + \sum_{k=1}^{m_1}\lfloor ka/b\rfloor + \sum_{k=1}^{m_2}\lfloor kb/a\rfloor \bigg).
\end{equation}
In particular, this is a nonnegative even integer.

How many generators are there of each grading? By Taubes's isomorphism \eqref{eqn:echswf}, together with the calculation of the Seiberg-Witten Floer homology of $S^3$ in \cite{km}, we should get
\begin{equation}
\label{eqn:ECHell}
ECH_*(\partial E(a,b),\lambda,0,J) = \left\{ \begin{array}{cl} \Z/2, & *=0,2,4,\ldots,\\ 0, & \mbox{otherwise}.\end{array}\right.
\end{equation}

\begin{exercise}
\label{ex:ECHell}
Deduce \eqref{eqn:ECHell} from \eqref{eqn:Iell}. That is, show that \eqref{eqn:Iell} defines a bijection from the set of pairs of nonnegative integers $(m_1,m_2)$ to the set of nonnegative even integers. (See hint in \S\ref{sec:answers}.)
\end{exercise}

\subsection{The $U$ map}
\label{sec:Udetails}

We now explain some more details of the $U$ map which was introduced in \S\ref{sec:addstr}, following \cite[\S2.5]{wh}.

Suppose $Y$ is connected, and choose a point $z\in Y$ which is not on any Reeb orbit. Let $\alpha$ and $\beta$ be generators of the chain complex $ECC_*(Y,\lambda,\Gamma,J)$, and let $\mc{C}\in\mc{M}_2(\alpha,\beta)$ be a holomorphic current with $(0,z)\in\mc{C}$. By Proposition~\ref{prop:I03}, we have $\mc{C}=\mc{C}_0\sqcup C_2$ where $\mc{I}(\mc{C}_0)=0$, and $C_2$ is embedded and $\ind(C_2)=2$. Since $\mc{C}_0$ is a union of trivial cylinders and $z$ is not on any Reeb orbit, it follows that $(0,z)\in C_2$. Let $N_{(0,z)}C_2$ denote the normal bundle to $C_2$ at $(0,z)$. There is then a natural map
\begin{equation}
\label{eqn:Uregular}
T_{\mc{C}}\mc{M}_2(\alpha,\beta)\to N_{(0,z)}C_2.
\end{equation}
Transversality arguments as in Proposition~\ref{prop:ind} can be used to show that if $J$ is generic then the map \eqref{eqn:Uregular} is an isomorphism for all holomorphic currents $\mc{C}$ as above. In particular, this implies that the set of holomorphic currents $\mc{C}$ as above is discrete.
For $J$ with this property, we define a chain map
\[
U_z:ECC_*(Y,\lambda,\Gamma,J) \longrightarrow ECC_{*-2}(Y,\lambda,\Gamma,J)
\]
by
\[
U_z\alpha = \sum_\beta\#\{\mc{C}\in\mc{M}_2(\alpha,\beta)\mid (0,z)\in \mc{C}\}\beta,
\]
where $\#$ denotes the mod $2$ count as usual.

A compactness argument similar to the proof that $\partial$ is defined in \S\ref{sec:differentialdefined} shows that $U_z$ is defined. Likewise, the proof that $\partial^2=0$ introduced in \S\ref{sec:200pages} can be modified to show that $\partial U_z = U_z\partial$.

To show that the map \eqref{eqn:UJ} on ECH induced by $U_z$ does not depend on $z$,
suppose $z'\in Y$ is another point which is not on any Reeb orbit. Since there are only countably many Reeb orbits, we can choose an embedded path $\eta$ from $z$ to $z'$ which does not intersect any Reeb orbit. Define a map
\[
K_\eta: ECC_*(Y,\lambda,\Gamma,J) \longrightarrow ECC_{*-1}(Y,\lambda,\Gamma,J)
\]
by
\[
K_\eta\alpha = \sum_\beta\#\{(\mc{C},y)\in\mc{M}_1(\alpha,\beta)\times Y\mid (0,y)\in \mc{C}\}\beta.
\]
Similarly to the proof that $\partial$ is well-defined, $K_\eta$ is well-defined if $J$ is generic. Similarly to the proof that $\partial^2=0$, one proves the chain homotopy equation
\begin{equation}
\label{eqn:Keta}
\partial K_\eta + K_\eta\partial = U_z - U_{z'}.
\end{equation}

\begin{remark}
If $z=z'$, then it follows from \eqref{eqn:Keta} that $K_\eta$ induces a map on ECH of degree $-1$. In fact this map depends only on the homology class of the loop $\eta$, and thus defines a homomorphism from $H_1(Y)$ to the set of degree $-1$ maps on $ECH_*(Y,\xi,\Gamma)$. See \cite[\S12.1]{t3} for more about this map and an example where it is nontrivial, and \cite{e5} for the proof that it agrees with an analogous map on Seiberg-Witten Floer cohomology.
\end{remark}

\subsection{Partition conditions}
\label{sec:partitions}

The definitions of the ECH differential and the U map do not directly specify the topological type of the holomorphic currents to be counted. However it turns out that most of this information is determined indirectly.
We now explain how the covering multiplicities of the Reeb orbits at the ends of the nontrivial component of such a holomorphic current are uniquely determined if one knows the trivial cylinder components. (We will further see in \S\ref{sec:J0} that the genus of the nontrivial part of the holomorphic current is then determined by its relative homology class.)

Let $\alpha=\{(\alpha_i,m_i)\}$ and $\beta=\{(\beta_j,n_j)\}$ be orbit sets, and let $C\in\mc{M}(\alpha,\beta)$ be somewhere injective. For each $i$, the curve $C$ has ends at covers of $\alpha_i$ whose total covering multiplicity is $m_i$. The multiplicities of these covers are a partition of the positive integer $m_i$ which we denote by $p_i^+(C)$.  For example, if $C$ has two positive ends at $\alpha_i$, and one positive end at the triple cover of $\alpha_i$, then $m_i=5$ and $p_i^+(C)=(3,1,1)$. Likewise, the covering multiplicities of the negative ends of $C$ at covers of $\beta_j$ determine a partition of $n_j$, which we denote by $p_j^-(C)$.

For each embedded Reeb orbit $\gamma$ and each positive integer $m$, we will shortly define two partitions of $m$, the ``positive partition'' $p_\gamma^+(m)$ and the ``negative partition\footnote{In \cite{pfh2,ir}, $p_\gamma^+(m)$ is called the ``outgoing partition'' and denoted by $p_\gamma^{\op{out}}(m)$, while $p_\gamma^-(m)$ is called the ``incoming partition'' and denoted by $p_\gamma^{\op{in}}(m)$. It is never too late to change your terminology to make it clearer.}'' $p_\gamma^-(m)$. We then have:

\begin{partitionconditions} Suppose equality holds in the Writhe Bound \eqref{eqn:writhebound} for $C$. (This holds for example if $C$ is the nontrivial component of a holomorphic current that contributes to the ECH differential or the $U$ map.) Then $p_i^+(C)=p_{\alpha_i}^+(m_i)$ and $p_j^-(C)=p_{\beta_j}^-(n_j)$.
\end{partitionconditions}

The partitions $p_\gamma^\pm(m)$ are defined as follows. If $\gamma$ is positive hyperbolic, then
\[
p_\gamma^+(m) = p_\gamma^-(m) = (1,\ldots,1).
\]
Thus, if equality holds in the writhe bound for $C$, then $C$ can never have an end at a multiple cover of a positive hyperbolic Reeb orbit. If $\gamma$ is negative hyperbolic, then
\[
p_\gamma^+(m) = p_\gamma^-(m) = \left\{\begin{array}{cl} (2,\ldots,2), & \mbox{$m$ even,}\\ (2,\ldots,2,1), & \mbox{$m$ odd.}
\end{array}\right.
\]

Suppose now that $\gamma$ is elliptic with rotation angle $\theta$ with respect to some trivialization $\tau$ of $\xi|_\gamma$, see \S\ref{sec:ind}. Then $p_\gamma^\pm(m)=p_\theta^\pm(m)$, where the partitions $p_\theta^\pm(m)$ are defined as follows.

To define $p_\theta^+(m)$, let $\Lambda_\theta^+(m)$ be the maximal concave polygonal path in the plane (i.e.\ graph of a concave function) with vertices at lattice points which starts at the origin, ends at $(m,\lfloor m\theta\rfloor)$, and lies below the line $y=\theta x$. That is, $\Lambda_\theta^+(m)$ is the non-vertical part of the boundary of the convex hull of the set of lattice points $(x,y)$ with $0\le x\le m$ and $y\le\theta x$.
 Then $p_\theta^+(m)$ consists of the horizontal displacements of the segments of $\Lambda_\theta^+(m)$ connecting consecutive lattice points. 

The partition $p_\theta^-(m)$ is defined analogously from the path $\Lambda_\theta^-(m)$, which is the minimal convex polygonal path with vertices at lattice points which starts at the origin, ends at $(m,\lceil m\theta\rceil)$, and lies above the line $y=\theta x$. An equivalent definition is $p_\theta^-(m)=p_{-\theta}^+(m)$.

The partition $p_\theta^\pm(m)$ depends only on the class of $\theta$ in $\R/\Z$, and so $p_\gamma^\pm(m)$ does not depend on the choice of trivialization $\tau$.

The simplest example, which we will need for the computations in \S\ref{sec:examples}, is that if $\theta\in(0,1/m)$, then
\begin{equation}
\label{eqn:partitionexample}
\begin{split}
p_\theta^+(m) & = (1,\ldots,1),\\
p_\theta^-(m) &= (m).
\end{split}
\end{equation}
The partitions are more complicated for other $\theta$, see Figure~\ref{fig:partitions}.

\begin{figure}
\[
\begin{array}{|c||c|c|c|c|c|c|c|}
\hline
 & 2 & 3 & 4 & 5 & 6 & 7 & 8\\ \hline \hline
7/8,1 &  &  &  &  &  &  & 8 \\
\cline{8-8}
6/7,7/8 &  &  &  &  & 6 & \rb{7} & 7,1 \\
\cline{7-8}
5/6,6/7 &  &  & 4 & \rb{5} &  & 6,1 & 6,2 \\
\cline{6-8}
4/5,5/6 &  & 3 &  &  & 5,1 & 5,2 & 5,3 \\
\cline{5-8}
3/4,4/5 &  &  &  & 4,1 & 4,2 & 4,3 & 4,4 \\
\cline{4-8}
5/7,3/4 & 2 &  &  &  &  & 7 & 7,1 \\
\cline{7-8}
2/3,5/7 &  &  & \rb{3,1} & \rb{3,2} & \rb{3,3} & 3,3,1 & 3,3,2 \\
\cline{3-8}
5/8,2/3 &  &  &  &  &  &  & 8 \\
\cline{8-8}
3/5,5/8 &  &  &  & \rb{5} & \rb{5,1} & \rb{5,2} & 5,2,1 \\
\cline{5-8}
4/7,3/5 &  & \rb{2,1} & \rb{2,2} &  &  & 7 & 7,1 \\
\cline{7-8}
1/2,4/7 &  &  &  & \rb{2,2,1} & \rb{2,2,2} & 2,2,2,1 & 2,2,2,2 \\
\cline{2-8}
3/7,1/2 &  &  &  &  & & 7 & 7,1 \\
\cline{7-8}
2/5,3/7 &  &  &  & \rb{5} & \rb{5,1} & 5,1,1 & 5,3 \\
\cline{5-8}
3/8,2/5 &  & \rb{3} & \rb{3,1} &  &  &  & 8 \\
\cline{8-8}
1/3,3/8 &  &  &  & \rb{3,1,1} & \rb{3,3} & \rb{3,3,1} & 3,3,1,1 \\
\cline{3-8}
2/7,1/3 &  &  &  &  &  & 7 & 7,1 \\
\cline{7-8}
1/4,2/7 & 1,1 &  & \rb{4} & \rb{4,1} & \rb{4,1,1} & 4,1,1,1 & 4,4 \\
\cline{4-8}
1/5,1/4 &  &  &  & 5 & 5,1 & 5,1,1 & 5,1,1,1 \\
\cline{5-8}
1/6,1/5 &  & 1,1,1 &  &  & 6 & 6,1 & 6,1,1 \\
\cline{6-8}
1/7,1/6 &  &  & 1,1,1,1 &  &
& 7 & 7,1 \\
\cline{7-8}
1/8,1/7 &  &  &  & \rb{1,\ldots,1} & 1,\ldots,1 &  & 8 \\
\cline{8-8}
0,1/8 &  &  &  &  &  &
\rb{1,\ldots,1} & 1,\ldots,1 \\ \hline
\end{array}
\]
\caption{The positive partitions $p_\theta^+(m)$
for $2\le m\le 8$ and all $\theta$.  The left column shows the
interval in which $\theta\mod 1$ lies, and the top row indicates $m$. (Borrowed from \cite{pfh2})}
\label{fig:partitions}
\end{figure}

If $m>1$, then $p_\theta^+(m)$ and $p_\theta^-(m)$ are disjoint. (This makes the gluing theory to prove $\partial^2=0$ nontrivial, see \S\ref{sec:200pages}.) This is a consequence of the following exercise, which may help in understanding the partitions.

\begin{exercise}
\label{ex:partitions}
(See answer in \S\ref{sec:answers}.)
Write $p_\theta^+(m)=(q_1,\ldots,q_k)$ and $p_\theta^-(m)=(r_1,\ldots,r_l)$.
\begin{description}
\item{(a)} Show that if $(a,b)$ is an edge vector of the path $\Lambda_\theta^+(m)$, then $b=\floor{a\theta}$.
\item{(b)}  Show that $\sum_{i\in I}\floor{q_i\theta}=\left\lfloor\sum_{i\in I}q_i\theta\right\rfloor$ for each subset $I\subset\{1,\ldots,k\}$.
\item{(c)}  Show that there do not exist proper subsets $I\subset\{1,\ldots,k\}$ and $J\subset\{1,\ldots,l\}$ such that $\sum_{i\in I}q_i = \sum_{j\in J}r_j$.
\end{description}
\end{exercise}

Here is a related combinatorial exercise, some of which is needed for the proofs that $\partial$ is well-defined and $\partial^2=0$ in \S\ref{sec:differentialdefined} and \S\ref{sec:200pages}.

\begin{exercise}
\label{ex:partialorder}
(See answer in \S\ref{sec:answers}.)
Fix an irrational number $\theta$ and a positive integer $m$. Suppose $\gamma$ is an embedded elliptic Reeb orbit with rotation angle $\theta$.
\begin{description}
\item{(a)} Show that if $u:C\to\R\times\gamma$ is a degree $m$ branched cover, regarded as a holomorphic curve in $\R\times Y$, then the Fredholm index\footnote{The Fredholm index of a possibly multiply covered curve $u:C\to\R\times Y$ is defined as in \eqref{eqn:ind3}, with $c_\tau(C)$ replaced by $c_1(u^*\xi,\tau)$.} $\ind(u)\ge 0$.
\item{(b)}
If $(a_1,\ldots,a_k)$ and $(b_1,\ldots,b_l)$ are partitions of $m$, define $(a_1,\ldots,a_k)\ge (b_1,\ldots,b_l)$ if there is a branched cover $u$ of $\R\times\gamma$ with positive ends at $\gamma^{a_i}$, negative ends at $\gamma^{b_j}$, and $\ind(u)=0$. Show that $\ge$ is a partial order on the set of partitions of $m$.
\item{(c)}
Show that $p_\theta^-(m)\ge p_\theta^+(m)$.
\item{(d)}
Show that there does not exist any partition $q$ with $q>p_\theta^-(m)$ or $p_\theta^+(m)>q$.
\end{description}
\end{exercise}

\begin{remark}
If $\mc{C}\in\mc{M}(\alpha,\beta)$ contributes to the differential or the $U$ map, and if $\mc{C}$ contains trivial cylinders, then additional partition conditions must hold; see \cite[Prop.\ 7.1]{pfh2} and \cite[Lem.\ 7.28]{obg1} for these conditions.
\end{remark}


\section{More examples of ECH}
\label{sec:examples}

The calculation of the ECH of an ellipsoid in \S\ref{sec:ellipsoid} was fairly simple because we just had to determine the grading of each generator. We now outline some more complicated calculations which require counting holomorphic curves. These are useful for further understanding the machinery, and relevant to the symplectic embedding obstructions described in \S\ref{sec:capacities}.

\subsection{The $U$ map on the ECH of an ellipsoid}
\label{sec:Uell}

We first return to the ellipsoid example from \S\ref{sec:ellipsoid}. Recall from \eqref{eqn:ECHell} that $ECH_*(\partial E(a,b),\lambda,0)$ has one generator of grading $2k$ for each $k=0,1,\ldots$; denote this generator by $\zeta_k$. To calculate the ECH capacities of $E(a,b)$ in \S\ref{sec:definecapacities}, we needed:

\begin{proposition}
\label{prop:Uell}
For any $J$, the $U$ map on $ECH_*(\partial E(a,b),\lambda,0,J)$ is given by
\begin{equation}
\label{eqn:Uzeta}
U\zeta_k=\zeta_{k-1}, \quad k>0.
\end{equation}
\end{proposition}

As mentioned in Example~\ref{ex:zeta}, this follows from the isomorphism with Seiberg-Witten theory.
However it is instructive to try to prove Proposition~\ref{prop:Uell} directly in ECH, without using Seiberg-Witten theory.

First of all, we can see directly in this case that the $U$ map does not depend on the almost complex structure $J$. The idea is that if we generically deform $J$, then similarly to the compactness part of the proof that $\partial^2=0$, see Lemma~\ref{lem:compactness2}, the chain map $U_z$ can change only if at some time there is a broken holomorphic curve containing a level with $I=1$. But there are no $I=1$ curves by the Index Parity property \eqref{eqn:indexparity} since all Reeb orbits are elliptic. 

We now sketch a direct proof of Proposition~\ref{prop:Uell} in the special case when $a=1-\epsilon$ and $b=1+\epsilon$ where $\epsilon>0$ is sufficiently small with respect to $k$.  (One can probably prove the general case similarly with more work.)

If $\epsilon$ is sufficiently small with respect to $k$, then $\zeta_k$ is the $k^{th}$ generator in the sequence
\[
1,\gamma_1,\gamma_2,\gamma_1^2,\gamma_1\gamma_2,\gamma_2^2,\gamma_1^3,\gamma_1^2\gamma_2,\gamma_1\gamma_2^2,\gamma_2^3,\ldots
\]
(indexed starting at $k=0$). So to prove Proposition~\ref{prop:Uell} in our special case, it is enough to show the following:

\begin{lemma}
\label{lem:Uell}
If $a=1-\epsilon$ and $b=1+\epsilon$, then
the $U$ map on $ECH_*(\partial E(a,b),\lambda,0,J)$ is given by:
\begin{description}
\item{(a)}
$U(\gamma_1^i\gamma_2^j)=\gamma_1^{i+1}\gamma_2^{j-1}$ if $j>0$ and $\epsilon>0$ is sufficiently small with respect to $i+j$.
\item{(b)}
$U(\gamma_1^i)=\gamma_2^{i-1}$ if $i>0$ and $\epsilon>0$ is sufficiently small with respect to $i$.
\end{description}
\end{lemma}

\begin{proof} The proof has three steps.

{\em Step 1.\/} We first determine the types of holomorphic curves we need to count.

Let $\mc{C}$ be a holomorphic current that contributes to $U_z(\gamma_1^i\gamma_2^j)$ where $i+j>0$. Write $\mc{C}=\mc{C}_0\sqcup C_2$ as in Proposition~\ref{prop:I03}. 
It follows from the partition conditions \eqref{eqn:partitionexample} that
$C_2$ has at most one positive end at a cover of $\gamma_1$, all positive ends of $C_2$ at covers of $\gamma_2$ have multiplicity $1$, all negative ends of $C_2$ at covers of $\gamma_1$ have multiplicity $1$, and $C_2$ has at most one negative end at a cover of $\gamma_2$.

\begin{exercise}
\label{ex:Uell1}
Deduce from this and the equation $\ind(C_2)=2$ that if $j=0$, then $C_2$ is a cylinder if $i>1$, and a plane if $i=1$, assuming that $\epsilon>0$ is sufficiently small with respect to $i$. (See answer in \S\ref{sec:answers}.)
\end{exercise}

\begin{exercise}
\label{ex:Uell2}
Similarly show that if $j>0$, then $C_2$ is a cylinder with a positive end at $\gamma_2$ and a negative end at $\gamma_1$, assuming that $\epsilon>0$ is sufficiently small with respect to $i+j$. (See answer in \S\ref{sec:answers}.)
\end{exercise}

{\em Step 2.\/} We now observe that the transversality conditions needed to define $U_z$, see \S\ref{sec:Udetails}, hold automatically  for any symplectization-admissible $J$. This follows from two general facts. First, if $C$ is an immersed irreducible $J$-holomorphic curve such that
\begin{equation}
\label{eqn:at}
2g(C) -2 + h_+(C) < \ind(C),
\end{equation}
then $C$ is automatically regular. Here $g(C)$ denotes the genus of the domain of $C$, and  $h_+(C)$ denotes the number of ends of $C$ at positive hyperbolic orbits, including even covers of negative hyperbolic orbits. This and much more general automatic transversality results are proved in \cite{wendl:at}. Second, if
\begin{equation}
\label{eqn:ai}
2g(C)-2+\ind(C)+h_+(C)=0,
\end{equation}
then every nonzero element of the kernel of the deformation operator of $C$ is nonvanishing\footnote{The left side of \eqref{eqn:ai} is called the ``normal Chern number'' by Wendl \cite{wendl:c}. Any holomorphic curve $u$ in $\R\times Y$ has normal Chern number $\ge 0$, with equality only if the projection of $u$ to $Y$ is an immersion. In favorable cases one can further show that the projection of $u$ to $Y$ is an embedding. One such favorable case is described in \cite[Prop.\ 3.4]{wh}, which is used to characterize contact three-manifolds in which all Reeb orbits are elliptic.}. If $C=C_2$ where $C_2$ is one of the holomorphic curves described in Step 1, then $C_2$ has genus zero, Fredholm index 2, and all ends at elliptic orbits, so both conditions \eqref{eqn:at} and \eqref{eqn:ai} hold, and we conclude that $C_2$ is regular and the map \eqref{eqn:Uregular} has no kernel, which is exactly the transversality needed to define $U_z$. 

{\em Step 3.} We now count the holomorphic curves $C_2$ described in Step 1. To do so, consider the case $a=b=1$. Here the contact form is not nondegenerate, as every point on $Y=S^3$ is on a Reeb orbit. Indeed, the set of embedded Reeb orbits can be identified with $\C P^1$, so that the map $S^3\to \C P^1$ sending a point to the Reeb orbit on which it lies is the Hopf fibration. This is an example of a ``Morse-Bott'' contact form.

It is explained by Bourgeois \cite{bourgeois} how one can understand holomorphic curves for a nondegenerate perturbation of a Morse-Bott contact form in terms of holomorphic curves for the Morse-Bott contact form itself.  In the present case, this means that we can understand holomorphic curves for the ellipsoid with $a=1-\epsilon$, $b=1+\epsilon$, in terms of holomorphic curves for the sphere with $a=b=1$. Specifically, let $p_i\in \C P^1$ denote the point corresponding to the Reeb orbit $\gamma_i$ for $i=1,2$. Choose a Morse function $f:\C P^1\to \R$ with an index $2$ critical point at $\gamma_2$ and and index $0$ critical point at $\gamma_1$. Then \cite{bourgeois} tells us the following.

First, a holomorphic cylinder for the perturbed contact form with a positive end at $\gamma_2$ and a negative end at $\gamma_1$ (modulo $\R$ translation) corresponds to a negative gradient flow line of $f$ from $p_2$ to $p_1$. If we choose a base point $\overline{z}\in \C P^1\setminus \{p_1,p_2\}$, then there is exactly one such flow line passing through $\overline{z}$. One can deduce from this that if we choose a base point $z\in Y$ which is not on $\gamma_1$ or $\gamma_2$, then there is exactly one holomorphic cylinder with a positive end at $\gamma_2$ and a negative end at $\gamma_1$ passing through $(0,z)$. This proves part (a) of Lemma~\ref{lem:Uell}.

Second, to prove part (b) of Lemma~\ref{lem:Uell}, we need to count holomorphic cylinders (or planes when $i=1$) $C$ for the Morse-Bott contact form with a positive end at $\gamma_1^i$, and a negative end at $\gamma_2^{i-1}$ when $i>0$, which pass through a base point. To count these, let $\mc{L}$ denote the tautological line bundle over $\C P^1$.
Let $J$ denote the canonical complex structure on $\mc{L}$, and let $Z\subset\mc{L}$ denote the zero section.

\begin{exercise}
\label{ex:meromorphic}
One can identify $\mc{L}\setminus Z\simeq \R\times S^3$ so that $J$ corresponds to a symplectization-admissible almost complex structure. A meromorphic section $\psi$ of $\mc{L}$ determines a holomorphic curve in $\R\times S^3$ with positive ends corresponding to the zeroes of $\psi$, and negative ends corresponding to the poles of $\psi$. Conversely, a holomorphic curve in $\R\times S^3$ which intersects each fiber of $\mc{L}\setminus Z\to \C P^1$, except for the fibers over the Reeb orbits at the positive and negative ends, transversely in a single point, comes from a meromorphic section of $\mc{L}$.
\end{exercise}

If $C$ is a holomorphic curve as in the paragraph preceding the above exercise, then by the definition of linking number in $S^3$, the curve $C$ has algebraic intersection number $1$ with each fiber of $\mc{L}\setminus Z$ over $\C P^1\setminus\{p_1,p_2\}$. By intersection positivity, $C$ intersects each such fiber transversely in a single point. It follows then from Exercise~\ref{ex:meromorphic} that to compute $U\gamma_1^i$, we need to count meromorphic sections of $\mc{L}$ with a zero of order $i$ at $p_1$, a pole of order $i-1$ at $p_2$, and no other zeroes or poles, which pass through a base point in $\mc{L}\setminus Z$. 
There is exactly one such meromorphic section, and this completes the proof of Lemma~\ref{lem:Uell}.
\end{proof}

\subsection{The ECH of $T^3$}
\label{sec:t3}

Our next example of ECH is more complicated, but will ultimately be useful in computing many examples of ECH capacities.
We consider
\[
Y=T^3=(\R/2\pi\Z)\times(\R/ \Z)^2.
\]
Let $\theta$ denote the $\R/2\pi\Z$ coordinate and let $x,y$ denote the two $\R/\Z$ coordinates. We start with the contact form
\begin{equation}
\label{eqn:lambda1}
\lambda_1 = \cos\theta\, dx + \sin\theta\, dy.
\end{equation}
Let $\xi_1=\Ker(\lambda_1)$; we now describe how to compute $ECH_*(T^3,\xi_1,0)$, following \cite{t3}.

\paragraph{Perturbing the contact form.}
The Reeb vector field associated to $\lambda_1$ is
\[
R_1 = \cos\theta\frac{\partial}{\partial x} + \sin\theta\frac{\partial}{\partial y}.
\]
If $\tan\theta\in\Q\cup\{\infty\}$, so that the vector $(\cos\theta,\sin\theta)$ is a positive real multiple of a vector $(a,b)$ where $a,b$ are relatively prime integers, then every point on $\{\theta\}\times(\R/\Z)^2$ is on an embedded Reeb orbit $\gamma$ in the homology class $(0,a,b)\in H_1(T^3)$.  The symplectic action of the Reeb orbit $\gamma$ is
\[
\mc{A}(\gamma) = \sqrt{a^2+b^2}.
\]
In particular, there is a circle $S_{a,b}$ of such Reeb orbits. Thus the contact form $\lambda_1$ is not nondegenerate; again it is Morse-Bott.

To compute the ECH of $\xi_1$, we will perturb $\lambda_1$ to a nondegenerate contact form. Given $a,b$, one can perturb the contact form $\lambda_1$ near $S_{a,b}$ so that, modulo longer Reeb orbits, the circle of Reeb orbits $S_{a,b}$ becomes just two embedded Reeb orbits, one elliptic with rotation angle slightly positive, and one positive hyperbolic. We denote these by $e_{a,b}$ and $h_{a,b}$. The orbits $e_{a,b}$ and $h_{a,b}$ are still in the homology class $(0,a,b)$, and have symplectic action close to $\sqrt{a^2+b^2}$, with the action of $e_{a,b}$ slightly greater than that of $h_{a,b}$. For any given $L>0$, one can perform such a perturbation for all of the finitely many pairs of relatively prime integers $(a,b)$ with $\sqrt{a^2+b^2}<L$, to obtain a contact form $\lambda$ for which the embedded Reeb orbits with symplectic action less than $L$ are the elliptic orbits $e_{a,b}$ and the hyperbolic orbits $h_{a,b}$ where $(a,b)$ ranges over all pairs of relatively prime integers with $\sqrt{a^2+b^2}<L$.

It is probably not possible to do this for $L=\infty$, i.e.\ to find a contact form such that the embedded Reeb orbits of all actions are the orbits $e_{a,b}$ and $h_{a,b}$ where $(a,b)$ ranges over all pairs of relatively prime integers. Rather, one can show that to calculate the ECH of $\xi_1$, we can perturb as above for a given $L$, compute the filtered ECH in symplectic action less than $L$, and take the direct limit as $L\to\infty$. In the calculations below, we only consider generators of symplectic action less than $L$, and we omit $L$ from the notation.

\paragraph{The generators.}

A generator of the chain complex $ECC_*(Y,\lambda,0,J)$ now consists of a finite set of Reeb orbits $e_{a,b}$ and $h_{a,b}$ with positive integer multiplicities, where each $h_{a,b}$ has multiplicity $1$, and the sum with multiplicities of all the vectors $(a,b)$ is $(0,0)$. To describe this more simply, if $(a,b)$ are relatively prime integers and if $m$ is a positive integer, let $e_{ma,mb}$ denote the elliptic orbit $e_{a,b}$ with multiplicity $m$; and let $h_{ma,mb}$ denote the hyperbolic orbit $h_{a,b}$, together with the elliptic orbit $e_{a,b}$ with multiplicity $m-1$ when $m>1$. A chain complex generator then consists of a finite set of symbols $e_{a,b}$ and $h_{a,b}$, where each $(a,b)$ is a pair of (not necessarily relatively prime) integers which are not both zero, no pair $(a,b)$ appears more than once, and the sum of the vectors $(a,b)$ that appear is zero. If we arrange the vectors $(a,b)$ head to tail in order of increasing slope, we obtain a convex polygon in the plane. Thus, a generator of the chain complex $ECC_*(Y,\lambda,0,J)$ can be represented as convex polygon $\Lambda$ in the plane, modulo translation, with vertices at lattice points, with each edge labeled either `$e$' or `$h$'. The polygon can be a 2-gon (for a generator such as $e_{a,b}e_{-a,-b}$) or a point (for the empty set of Reeb orbits). The symplectic action of the generator is approximately the Euclidean length of the polygon $\Lambda$.

\paragraph{The grading.}

The two-plane field $\xi_1$ is trivial; indeed $\partial_\theta$ defines a global trivialization $\tau$. Thus $c_1(\xi_1)=0$, and the chain complex $ECC_*(T^3,\lambda,0,J)$ has a canonical $\Z$-grading, in which the empty set has grading zero. 

\begin{lemma}
The canonical $\Z$-grading of a generator $\Lambda$ is given by
\begin{equation}
\label{eqn:t3grading}
|\Lambda| = 2(\mc{L}(\Lambda)-1) - h(\Lambda),
\end{equation}
where $\mc{L}(\Lambda)$ denotes the number of lattice points enclosed by $\Lambda$ (including lattice points on the edges), and 
 $h(\Lambda)$ denotes the number of edges of $\Lambda$ that are labeled `$h$'.
\end{lemma}

\begin{proof}
As in \eqref{eqn:ellipsoidgrading}, we can write the grading of a generator $\Lambda$ as
\[
|\Lambda| = c_\tau(\Lambda) + Q_\tau(\Lambda) +
CZ_\tau^I(\Lambda).
\]
Since $\tau$ is a global trivialization, $c_\tau(\Lambda)=0$. We also have $CZ_\tau(e_{a,b})=1$ and $CZ_\tau(h_{a,b})=0$; consequently,
\[
CZ_\tau^I(\Lambda) = m(\Lambda) - h(\Lambda),
\]
where $m(\Lambda)$ denotes the total divisibility of all edges of $\Lambda$. Finally, it is a somewhat challenging exercise (which can be solved by the argument in \cite[Lem.\ 3.7]{pfh3}) to show that
\[
Q_\tau(\Lambda) = 2\op{Area}(\Lambda)
\]
where $\op{Area}(\Lambda)$ denotes the area enclosed by $\Lambda$.
Now Pick's formula for the area of a lattice polygon asserts that
\[
2\op{Area}(\Lambda) = 2\mc{L}(\Lambda) - m(\Lambda) - 2.
\]
The grading formula \eqref{eqn:t3grading} follows from the above four equations.
\end{proof}

\paragraph{Combinatorial formula for the differential.}

Define a combinatorial differential
\[
\delta: ECC_*(T^3,\lambda,0,J) \longrightarrow ECC_{*-1}(T^3,\lambda,0,J)
\]
as follows. If $\Lambda$ is a generator, then $\delta\Lambda$
 is the sum over all labeled polygons $\Lambda'$ that are obtained from $\Lambda$ by ``rounding a corner'' and ``locally losing one `$h$'{}''. Here ``rounding a corner'' means replacing the polygon $\Lambda$ by the boundary of the convex hull of the set of enclosed lattice points with one corner removed. ``Locally losing one `$h$' '' means that of the two edges adjacent to the corner that is rounded, at least one must be labeled `$h$'; if only one is labeled `$h$', then all edges created or shortened by the rounding are labeled `$e$'; otherwise exactly one of the edges created or shortened by the rounding is labeled `$h$'. All edges not created or shortened by the rounding keep their previous labels.  It follows from \eqref{eqn:t3grading} that the combinatorial differential $\delta$ decreases the grading by $1$, since $\mc{L}(\Lambda')=\mc{L}(\Lambda)-1$ and $h(\Lambda')=h(\Lambda)-1$. A less trivial combinatorial fact, proved in \cite[Cor.\ 3.13]{t3}, is that $\delta^2=0$.

\begin{proposition}
\label{prop:t3}
\cite[\S11.3]{t3}
For every $L>0$, the perturbed contact form $\lambda$ and almost complex structure $J$ can be chosen so that up to symplectic action $L$, the ECH differential $\partial$ agrees with the combinatorial differential $\delta$.
\end{proposition}

We will describe some of the proof of Proposition~\ref{prop:t3} at the end of this subsection.

The homology of the combinatorial differential $\delta$ is computed in \cite{t3} (there with $\Z$ coefficients), and the conclusion (with $\Z/2$ coefficients) is that
\begin{equation}
\label{eqn:echt3}
ECH_*(T^3,\xi_1,0) \simeq \left\{\begin{array}{cl} (\Z/2)^3, & *\ge 0,\\
0, & *<0.
\end{array}
\right.
\end{equation}

\begin{exercise}
\label{ex:ech0t3}
Prove that the homology of the combinatorial differential $\delta$ in degree $0$ is isomorphic to $(\Z/2)^3$.
\end{exercise}

\paragraph{The $U$ map.}

To compute ECH capacities, we do not need to know the homology \eqref{eqn:echt3}, but rather the following combinatorial formula for the $U$ map. Pick $\theta\in\R/2\pi\Z$ with $\tan\theta$ irrational. Define a combinatorial map
\[
U_\theta: ECC_*(T^3,\lambda,0,J) \longrightarrow ECC_{*-2}(T^3,\lambda,0,J)
\]
as follows. If $\Lambda$ is a generator, then it has a distinguished corner $c_\theta$ such that the oriented line $T$ through $c_\theta$ with direction vector $(\cos\theta,\sin\theta)$ intersects $\Lambda$ only at $c_\theta$, with the rest of $\Lambda$ lying to the left of $T$. Then $U_\theta$ is the sum over all generators $\Lambda'$ obtained from $\Lambda$ by rounding the distinguished corner $c_\theta$ and ``conserving the $h$ labels''.  To explain what this last condition means, note that $\Lambda'$ also has a distinguished corner $c_\theta'$. If the edge of $\Lambda$ preceding $c_\theta$ is labeled `$h$', then exactly one of the new or shortened edges of $\Lambda'$ preceding $c_\theta'$ is labeled `$h$'; otherwise all new or shortened edges of $\Lambda'$ preceding $c_\theta'$ are labeled `$e$'. Likewise for the edge of $\Lambda$ following $c_\theta$ and the new or shortened edges of $\Lambda'$ following $c_\theta'$. All other edge labels are unchanged.

To connect this with the $U$ map on ECH, let $z=(\theta,x,y)\in T^3$ 
where $x,y\in\R/\Z$ are arbitrary.

\begin{proposition}
\label{prop:t3u}
\cite[\S12.1.4]{t3}
For any $L>0$, one can choose $\lambda$ and $J$ as in Proposition~\ref{prop:t3} so that up to symplectic action $L$, we have $U_z=U_\theta$, modulo terms that decrease the number of `$h$' labels.
\end{proposition}

In particular, if all edges of $\Lambda$ are labeled `$e$', then $U_z\Lambda$ is the generator $\Lambda'$ obtained from $\Lambda$ by rounding the distinguished corner $c_\theta$ and keeping all edges labeled `$e$'. (If $\Lambda$ is a point then $U_z\Lambda=0$.)

\paragraph{ECH spectrum.} We now use the above facts to compute the ECH spectrum of $(T^3,\lambda_1)$ in terms of a discrete isoperimetric problem.

\begin{proposition}
\label{prop:t3spectrum}
The ECH spectrum of $(T^3,\lambda_1)$ is given by
\begin{equation}
c_k(T^3,\lambda_1) = \min\left\{\ell(\Lambda)\;|\; \mc{L}(\Lambda) = k+1\right\},
\end{equation}
where the minimum is over closed convex polygonal paths $\Lambda$ with vertices at lattice points, $\ell$ denotes the Euclidean length, and $\mc{L}(\Lambda)$ denotes the number of lattice points enclosed by $\Lambda$, including lattice points on the edges.
\end{proposition}

\begin{proof}
Fix a nonnegative integer $k$. Let $\Lambda_k$ be a length-minimizing closed convex polygon with vertices at lattice points subject to the constraint $\mc{L}(\Lambda_k)= k+1$. We need to show that $c_k(T^3,\lambda_1)=\ell(\Lambda_k)$.

Fix $z\in T^3$ for use in definining the chain map $U_z$. Choose $L>\ell(\Lambda_k)$, and let $\lambda$ and $J$ be a perturbed contact form and almost complex structure supplied by Propositions~\ref{prop:t3} and \ref{prop:t3u}. Label all edges of $\Lambda_k$ by `$e$' in order to regard $\Lambda_k$ as a generator of the chain complex $ECC(T^3,\lambda,0,J)$. Then $\Lambda_k$ is a cycle by Proposition~\ref{prop:t3}, and $U_z^k\Lambda_k=\emptyset$ by Proposition~\ref{prop:t3u}. Thus $c_k(T^3,\lambda)$ is less than or equal to the the symplectic action of $\Lambda_k$, which is approximately $\ell(\Lambda_k)$. It follows from the limiting definition of the ECH spectrum for degenerate contact forms in \S\ref{sec:definecapacities} that $c_k(T^3,\lambda_1)\le \ell(\Lambda_k)$.

To complete the proof, we now show that $c_k(T^3,\lambda_1)\ge \ell(\Lambda_k)$. It is enough to show that if $\Lambda$ is any other generator with $\langle U_z^k\Lambda,\emptyset\rangle \neq 0$, then $\ell(\Lambda)\ge \ell(\Lambda_k)$.  Since $|\Lambda|=2k$, it follows from the grading formula \eqref{eqn:t3grading} that
\[
\mc{L}(\Lambda) = k+1 + \frac{h(\Lambda)}{2}.
\]
We then have
\[
\ell(\Lambda) \ge \ell(\Lambda_{k+h(\Lambda)/2}) \ge \ell(\Lambda_k)
\]
where the first inequality holds by definition, and the second inequality holds because rounding corners of polygons decreases length\footnote{It is a combinatorial exercise to prove that rounding corners of polygons decreases length, see \cite[Lem.\ 2.14]{t3}.}.
\end{proof}

\paragraph{Computing the differential.}

We now indicate a bit of what is involved in the proof of Proposition~\ref{prop:t3}; similar arguments prove Proposition~\ref{prop:t3u}. For the application to ECH capacities, one may skip ahead to \S\ref{sec:toric}.

The easier half of the proof of Proposition~\ref{prop:t3} is to show that $\lambda$ and $J$ can be chosen so that
\begin{equation}
\label{eqn:easyhalf}
\langle\partial\Lambda,\Lambda'\rangle\neq 0\;  \Longrightarrow\; \langle\delta\Lambda,\Lambda'\rangle\neq 0.
\end{equation}
The following lemma is a first step towards proving \eqref{eqn:easyhalf}.

\begin{lemma}
\label{lem:easyhalf}
Let $\mc{C}\in\mc{M}(\Lambda,\Lambda')$ be a holomorphic current that contributes to the differential $\partial$, and write $\mc{C}=\mc{C}_0\sqcup C_1$ as in Proposition~\ref{prop:I03}. Then $C_1$ has genus zero, and one of the following three alternatives holds:
\begin{description}
\item{(i)}
$C_1$ is a cylinder with positive end at an embedded elliptic orbit $e_{a,b}$ and negative end at $h_{a,b}$.
\item{(ii)}
$C_1$ has two positive ends, and the number of positive ends at hyperbolic orbits is one more than the number of negative ends at hyperbolic orbits.
\item{(iii)}
$C_1$ has three positive ends, all at hyperbolic orbits; and all negative ends of $C_1$ are at elliptic orbits.
\end{description}
\end{lemma}

\begin{proof}
Let us first see what the Fredholm index formula \eqref{eqn:ind3} tells us about $C_1$. Let $g$ denote the genus of $C_1$, let $e_+$ denote the number of positive ends of $C_1$ at elliptic orbits, let $h_+$ denote the number of positive ends of $C_1$ at hyperbolic orbits, and let $e_-$ and $h_-$ denote the number of negative ends of $C_1$ at elliptic and hyperbolic orbits respectively. Then
\[
\chi(C_1) = 2-2g-e_+-h_+-e_--h_-
\]
and
\[
CZ_\tau^{\op{ind}}(C) = e_+ - e_-,
\]
so by the Fredholm index formula \eqref{eqn:ind3} we have
\[
\ind(C_1) = 2g - 2 + 2e_+ + h_+ + h_-.
\]
Since $ind(C_1)=1$, we obtain
\begin{equation}
\label{eqn:3dl}
2g + 2e_+ + h_+ + h_- = 3.
\end{equation}

Since the differential $\partial$ decreases symplectic action, $C_1$ has at least one positive end.

\begin{exercise}
\label{ex:t3action}
Further use the fact that the differential $\partial$ decreases symplectic action to show that $g=0$.
(See answer in \S\ref{sec:answers}.)
\end{exercise}

If $C_1$ has exactly one positive end, then similarly to the solution to Exercise~\ref{ex:t3action}, this positive end is at an elliptic orbit. By the partition conditions \eqref{eqn:partitionexample}, this positive end is at an embedded elliptic orbit $e_{a,b}$. Then, similarly to the solution to Exercise~\ref{ex:t3action}, $C_1$ has exactly one negative end, which is at $h_{a,b}$, so alternative (i) holds.

If $C_1$ has more than one positive end, then it follows from equation \eqref{eqn:3dl} that alternative (ii) or (iii) holds.
\end{proof}

Now $\langle\delta\Lambda,\Lambda'\rangle\neq 0$ is only possible in case (ii). So to prove \eqref{eqn:easyhalf} we would like to rule out alternatives (i) and (iii).
In fact alternative (i) cannot be ruled out; there are two holomorphic cylinders from $e_{a,b}$ to $h_{a,b}$ for each pair of relatively prime integers $(a,b)$. These arise in the Morse-Bott picture from the two flow gradient flow lines of the Morse function on the circle of Reeb orbits $S_{a,b}$ that we used to perturb the Morse-Bott contact form $\lambda_1$, similarly to the proof of Proposition~\ref{prop:Uell}(a). However these cylinders cancel\footnote{There is also a ``twisted'' version of ECH in which these cylinders do not cancel in the differential, see \cite[\S12.1.1]{t3}.} in the ECH differential $\partial$. Alternative (iii) may occur depending on how exactly one perturbs the Morse-Bott contact form $\lambda_1$. However it is shown in \cite[\S11.3, Step 5]{t3} that the perturbation $\lambda$ and almost complex structure $J$ can be chosen so that alternative (iii) does not happen.

The main remaining step in the proof of \eqref{eqn:easyhalf} is to show that $\lambda$ and $J$ above can be chosen so that if $\langle\partial\Lambda,\Lambda'\rangle\neq 0$, then the polygons $\Lambda$ and $\Lambda'$ can be translated so that $\Lambda'$ is nested inside $\Lambda$. The proof uses intersection positivity, see \cite[\S10.3]{t3}.

To complete the proof of Proposition~\ref{prop:t3}, we need to prove the converse of \eqref{eqn:easyhalf}, namely that $\lambda$ and $J$ above can be chosen so that if $\langle\delta\Lambda,\Lambda'\rangle\neq 0$ then $\langle\partial\Lambda,\Lambda'\rangle\neq 0$. One can calculate $\langle\partial\Lambda,\Lambda'\rangle$ by counting appropriate holomorphic curves for the Morse-Bott contact form $\lambda_1$. Work of Taubes \cite{taubestropical} and Parker \cite{parker} determines the latter curves in terms of tropical geometry. Unfortunately it would take us too far afield to explain this story here.

\subsection{ECH capacities of convex toric domains}
\label{sec:toric}

We now use the results of \S\ref{sec:t3} to compute the ECH capacities of a large family of examples. Let $\Omega$ be a compact domain in $[0,\infty)^2$ with piecewise smooth boundary. Define a ``toric domain'' or ``Reinhardt domain''
\[
X_\Omega = \left\{(z_1,z_2)\in\C^2\;\big|\; \left(\pi|z_1|^2,\pi|z_2|^2\right)\in\Omega\right\}.
\]
For example, if $\Omega$ is the triangle with vertices $(0,0)$, $(a,0)$, and $(0,b)$, then $X_\Omega$ is the ellipsoid $E(a,b)$. If $\Omega$ is the rectangle with vertices $(0,0)$, $(a,0)$, $(0,b)$, and $(a,b)$, then $X_\Omega$ is the polydisk $P(a,b)$.

Assume now that $\Omega$ is convex and does not touch the axes. We can then state a formula for the ECH capacities of $X_\Omega$, similar to Proposition~\ref{prop:t3spectrum}.
Let $\Omega'\subset\R^2$ be a translate of $\Omega$ that contains the origin in its interior. Let $\|\cdot\|$ denote the (not necessarily symmetric) norm on $\R^2$ that has $\Omega'$ as its unit ball. Let $\|\cdot\|^*$ denote the dual norm on $(\R^2)^*$, which we identify with $\R^2$ via the Euclidean inner product $\langle\cdot,\cdot\rangle$. That is, if $v\in\R^2$, then
\[
\|v\|^* = \max\left\{\langle v,w\rangle\mid w\in\partial\Omega'\right\}.
\]
If $\Lambda$ is a polygonal path in $\R^2$, let $\ell_\Omega(\Lambda)$ denote the length of the path $\Lambda$ as measured using the dual norm $\|\cdot\|^*$, i.e.\ the sum of the $\|\cdot\|^*$-norms of the edge vectors of $\Lambda$.

\begin{exercise}
\label{ex:dualloop}
If $\Lambda$ is a loop, then $\ell_\Omega(\Lambda)$ does not depend on the choice of translate $\Omega'$ of $\Omega$. (See answer in \S\ref{sec:answers}.)
\end{exercise}

\begin{theorem}
\label{thm:toric}
\cite[Thm.\ 1.11]{qech}\footnote{The definition of $X_\Omega$ in \cite{qech} is different, but symplectomorphic to the one given here.}
If $\Omega$ is convex and does not intersect the axes, then
\begin{equation}
\label{eqn:toric}
c_k(X_\Omega) = \min\left\{\ell_\Omega(\Lambda)\;|\; \mc{L}(\Lambda)= k+1\right\},
\end{equation}
where the minimum is over closed convex polygonal paths $\Lambda$ with vertices at lattice points, and $\mc{L}(\Lambda)$ denotes the number of lattice points enclosed by $\Lambda$, including lattice points on the edges.
\end{theorem}

\begin{remark}
One can weaken and possibly drop the assumption that $\Omega$ does not intersect the axes. For example, the formula \eqref{eqn:toric} is still correct when $\Omega$ is a triangle or rectangle with two sides on the axes, so that $X_\Omega$ is an ellipsoid or polydisk. This is a consequence of the following exercise, which should help with understanding the combinatorial formula \eqref{eqn:toric}.
\end{remark}

\begin{exercise}
\label{ex:wulff}
\begin{description}
\item{(a)}
Suppose that $\Omega$ is a convex polygon. Show that the minimum on the right hand side of \eqref{eqn:toric} is the same if it is taken over closed convex polygonal paths $\Lambda$ with arbitrary vertices whose edges are parallel to the edges of $\Omega$.
\item{(b)} Use part (a), together with the formulas \eqref{eqn:ellipsoidcapacities} and \eqref{eqn:polydiskcapacities} for the ECH capacities of ellipsoids and polydisks, to verify that equation \eqref{eqn:toric} is correct when $X_\Omega$ is an ellipsoid or a polydisk.
\end{description}
\end{exercise}

\begin{proof}[Proof of Theorem~\ref{thm:toric}.] We first need to understand a bit about the symplectic geometry of the domains $X_\Omega$.
Define coordinates $\mu_1,\mu_2\in(0,\infty)$ and $\theta_1,\theta_2\in\R/2\pi\Z$ on $(\C^*)^2$ by writing $z_k=\sqrt{\mu_k/\pi}e^{i\theta_k}$ for $k=1,2$. In these coordinates, the standard symplectic form on $\C^2$ restricts to
\begin{equation}
\label{eqn:omegatoric}
\omega=\frac{1}{2\pi}\sum_{k=1}^2d\mu_k\, d\theta_k.
\end{equation}
A useful corollary of this is that
\begin{equation}
\label{eqn:toricvolume}
\op{vol}(X_\Omega) = \op{area}(\Omega).
\end{equation}

\begin{exercise}
Use \eqref{eqn:omegatoric} to show that if $\Omega_1$ and $\Omega_2$ do not intersect the axes, and if $\Omega_2$ can be obtained from $\Omega_1$ by the action of $SL_2\Z$ and translation, then $X_{\Omega_1}$ is symplectomorphic to $X_{\Omega_2}$.
\end{exercise}

Now suppose that $\Omega$ has smooth boundary, does not intersect the axes, and is star-shaped with respect to some origin $(\eta_1,\eta_2)\in\op{int}(\Omega)$. This last condition means that each ray starting at $(\eta_1,\eta_2)$ intersects $\partial\Omega$ transversely. We claim then that $\partial X_\Omega$ is contact type, so that $\Omega$ is a Liouville domain. Indeed,
\[
\rho= \sum_{k=1}^2(\mu_k-\eta_k)\frac{\partial}{\partial \mu_k}
\]
is a Liouville vector field transverse to $\partial X_\Omega$, see \S\ref{sec:overview}.

To describe the resulting contact form $\lambda=\imath_\rho\omega$ on $\partial X_\Omega$, suppose further that $\Omega$ is strictly convex. Then $\partial X_\Omega$ is diffeomorphic to $T^3$ with coordinates $\theta_1,\theta_2,\phi$, where $\theta_1,\theta_2$ were defined above, and $(\cos\phi,\sin\phi)$ is the unit tangent vector to $\partial\Omega$, oriented counterclockwise. The contact form is now
\begin{equation}
\label{eqn:toriclambda}
\lambda=\frac{1}{2\pi}\sum_{k=1}^2(\mu_k-\eta_k)d\theta_k,
\end{equation}
and the Reeb vector field is
\begin{equation}
\label{eqn:toricreeb}
R = \frac{2\pi}{\|(\sin\phi,-\cos\phi)\|^*}
\left(\sin\phi\frac{\partial}{\partial\theta_1} - \cos\phi\frac{\partial}{\partial\theta_2}\right).
\end{equation}
Here $\|\cdot\|^*$ denotes the dual norm as above, defined using the translate of $\Omega$ by $-\eta$. This means that $\lambda$ has a circle of Reeb orbits for each $\phi$ for which $(\sin\phi,-\cos\phi)$ is a positive multiple of a vector $(a,b)$ where $a,b$ are relatively prime integers, and the symplectic action of each such Reeb orbit is the dual norm $\|(a,b)\|^*$.

For example, if $\Omega$ is a disk of radius 1 centered at $\eta$, then the contact form \eqref{eqn:toriclambda} agrees with the standard contact form \eqref{eqn:lambda1} on $T^3$ (via the coordinate change $\theta_1=2\pi x, \theta_2=2\pi y, \phi=\theta+\pi/2$), and the norm $\|\cdot\|^*$ is the Euclidean norm. So in this case, Theorem~\ref{thm:toric} follows from Proposition~\ref{prop:t3spectrum}. In the general case, by the arguments in \cite[Prop.\ 10.15]{t3}, the calculations in \S\ref{sec:t3} work just as well for the contact form \eqref{eqn:toriclambda}, except that symplectic action is computed using the dual norm $\|\cdot\|^*$ instead of the Euclidean norm. This proves Theorem~\ref{thm:toric} whenever the boundary of $\Omega$ is smooth and strictly convex. The general case of Theorem~\ref{thm:toric} follows by a limiting argument.
\end{proof}

The key property of the contact form \eqref{eqn:toriclambda} that makes the above calculation work is that the direction of the Reeb vector field \eqref{eqn:toricreeb} rotates monotonically with $\phi$.
It is an interesting open problem to compute the ECH capacities of $X_\Omega$ when $\Omega$ is star-shaped with respect to some origin but not convex. In this case the direction of the Reeb vector field no longer increases monotonically as one moves along $\partial \Omega$, so the calculations in \S\ref{sec:t3} do not apply, as there can be more than one circle of Reeb orbits in the same homology class.

\paragraph{Polydisks.}

We now prove the formula \eqref{eqn:polydiskcapacities} for the ECH capacities of a polydisk $P(a,b)$. Let $\Omega$ be a rectangle with sides of length $a$ and $b$ parallel to the axes which does not intersect the axes. It follows from Theorem~\ref{thm:toric} and Exercise~\ref{ex:wulff}(b) that
\[
c_k(X_\Omega) = \min\left\{am+bn\;\big|\; m,n\in\N, (m+1)(n+1)\ge k+1\right\}.
\]
So to prove equation \eqref{eqn:polydiskcapacities} for the ECH capacities of a polydisk, it is enough to show that
\begin{equation}
\label{eqn:polydiskomega}
c_k(P(a,b))=c_k(X_\Omega).
\end{equation}

Observe that $X_\Omega$ is symplectomorphic to the product of two annuli of area $a$ and $b$. Also, an annulus can be symplectically embedded into a disk of slightly larger area, and a disk can be symplectically embedded into an annulus of slightly larger area. Consequently, for any $\epsilon>0$, there are symplectic embeddings
\[
P((1-\epsilon)a,(1-\epsilon)b)\subset X_\Omega \subset P((1+\epsilon)a,(1+\epsilon)b).
\]
It follows from this and the Monotonicity and Conformality properties of ECH capacities that \eqref{eqn:polydiskomega} holds. Indeed, {\em any\/} symplectic capacity satisfying the Monotonicity and Conformality properties must agree on $P(a,b)$ and $X_\Omega$.

\section{Foundations of ECH}

We now give an introduction to some of the foundational matters which were skipped over in \S\ref{sec:defech}. The subsections below introduce foundational issues in the logical order in which they arise in developing the theory, but for the most part can be read in any order.

Below, fix a closed oriented three-manifold $Y$, a nondegenerate contact form $\lambda$ on $Y$, and a generic symplectization-admissible almost complex structure $J$ on $\R\times Y$. 

\subsection{Proof of the writhe bound and the partition conditions}
\label{sec:writhebound}

We now outline the proof of the writhe bound \eqref{eqn:writhebound} and the partition conditions in \S\ref{sec:partitions}. One can prove this one Reeb orbit at a time. That is, let $C$ be a somewhere injective $J$-holomorphic curve, let $\gamma$ be an embedded Reeb orbit, and suppose that $C$ has positive ends at covers of $\gamma$ with multiplicities $a_1,\ldots,a_k$ satisfying $\sum_{i=1}^ka_i=m$.
Let $N$ be a tubular neighborhood of $\gamma$ and let $\zeta=C\cap(\{s\}\times N)$ where $s>>0$. Let $\tau$ be a trivialization of $\xi|_\gamma$.
We then need to prove the following lemma (together with an analogus lemma for the negative ends which will follow by symmetry):

\begin{lemma}
\label{lem:wb}
If $s>>0$, then $\zeta$ is a braid whose writhe satisfies
\[
w_\tau(\zeta) \le \sum_{i=1}^m CZ_\tau(\gamma^i) - \sum_{i=1}^k CZ_\tau(\gamma^{a_i}),
\]
with equality only if $(a_1,\ldots,a_k) = p_\gamma^+(m)$.
\end{lemma}

To sketch the proof of Lemma~\ref{lem:wb}, we assume for simplicity that $C$ contains no trivial cylinders, although this assumption is easily dropped. We now need to recall some facts about the asymptotics of holomorphic curves. To set this up, identify $N\simeq (\R/\Z)\times D^2$ so that $\gamma$ corresponds to $(\R/\Z)\times\{0\}$, and the derivative of the identification along $\gamma$ sends $\xi|_\gamma$ to $\{0\}\oplus \C$ in agreement with the trivialization $\tau$. It turns out that a nontrivial positive end of $C$ at the $d$-fold cover of $\gamma$ can be written using this identification as the image of a map
\[
\begin{split}
u: [s_0,\infty)\times (\R/d\Z) &\longrightarrow \R\times (\R/\Z) \times D^2,\\
(s,t) &\longmapsto (s,\pi(t),\eta(s,t))
\end{split}
\]
where $s_0>>0$ and $\pi:\R/d\Z\to\R/\Z$ denotes the projection.

We now want to describe the asymptotics of the function $\eta(s,t)$.
Under our tubular neighborhood identification, the almost complex structure $J$ on $\xi|_\gamma$ defines a family of $2\times 2$ matrices $J_t$ with $J_t^2=-1$ parametrized by $t\in\R/\Z$. Also, the linearized Reeb flow along $\gamma$ determines a connection $\nabla_t=\partial_t+S_t$ on $\xi|_\gamma$, where $J_tS_t$ is a symmetric matrix for each $t\in\R/\Z$.
If $d$ is a positive integer, define the ``asymptotic operator'' $L_d$ on functions $\R/d\Z\to\C$ by
\[
L_d=J_{\pi(t)}(\partial_t+S_{\pi(t)}).
\]
Note that the operator $L_d$ is formally self-adjoint, so its eigenvalues are real.

\begin{lemma}
\label{lem:hwz1}
\cite{hwz1}
Given an end $\eta$ of a holomorphic curve as above,
there exist $c,\kappa>0$, and a nonzero eigenfunction $\varphi$ of $L_d$ with eigenvalue $\lambda>0$, such that
\[
\left|\eta(s,t) - e^{-\lambda s}\varphi(t)\right| < ce^{(-\lambda-\kappa)s}
\]
for all $(s,t)\in[s_0,\infty)\times (\R/d\Z)$.
\end{lemma}

To make use of this lemma, we need to know something about the eigenfunctions of $L_d$ with positive eigenvalues.

\begin{lemma}
\label{lem:hwz2}
Let $\varphi:\R/d\Z\to\C$ be a nonzero eigenfunction of $L_d$ with eigenvalue $\lambda$. Then:
\begin{description}
\item{(a)}
$\varphi(t)\neq 0$ for all $t\in\R/d\Z$, so $\varphi$ has a well-defined winding number around $0$, which we denote by $\op{wind}_\tau(\varphi)$.
\item{(b)}
If $\lambda>0$ then $\op{wind}_\tau(\varphi)\le \floor{CZ_\tau(\gamma^d)/2}$.
\end{description}
\end{lemma}

\begin{proof}
The eigenfunction $\varphi$ satisfies the ordinary differential equation
\[
\partial_t\varphi = -\left(S_{\pi(t)}+\lambda J_{\pi(t)}\right)\varphi.
\]
Assertion (a) then follows from the uniqueness of solutions to ODE's. Assertion (b) is proved in \cite[\S3]{hwz2}.
\end{proof}

\begin{example}
Suppose $\gamma$ is elliptic with monodromy angle $\theta$ with respect to $\tau$. 
We can then choose the trivialization $\tau$ so that
\[
\nabla_t=\partial_t-2\pi i\theta.
\]
Suppose $J$ is chosen so that $J_t$ is multiplication by $i$ for each $t$. Then
\[
L_d = i\partial_t + 2\pi\theta.
\]
Eigenfunctions of $L_d$ are complex multiples of the functions
\[
\varphi_n(t) = e^{2\pi i n t / d}
\]
for $n\in\Z$, with eigenvalues
\begin{equation}
\label{eqn:lambdan}
\lambda_n = -2\pi n/d + 2\pi\theta
\end{equation}
and winding number
\begin{equation}
\label{eqn:windn}
\op{wind}_\tau(\varphi_n)=n.
\end{equation}
It follows from \eqref{eqn:lambdan} and \eqref{eqn:windn} that $\lambda_n>0$ if and only if $\op{wind}_\tau(\varphi_n)<d\theta$.
This is consistent with Lemma~\ref{lem:hwz2}(b) since by \eqref{eqn:CZell} we have
\[
\floor{CZ_\tau\left(\gamma^d\right)/2}=\floor{d\theta}.
\]
\end{example}

Now return to the slice $\zeta=C\cap(\{s\}\times N)$ where $s>>0$. 
The positive ends of $C$ at covers of $\gamma$ determine loops $\zeta_1,\ldots,\zeta_k$ in $N$ transverse to the fibers of $N\to \gamma$. We conclude from Lemmas~\ref{lem:hwz1} and \ref{lem:hwz2} that $\zeta_i$ has a well-defined winding number around $\gamma$ with respect to $\tau$, which we denote by $\op{wind}_\tau(\zeta_i)$, and this satisfies
\begin{equation}
\label{eqn:windingbound}
\op{wind}_\tau(\zeta_i)\le\floor{CZ_\tau(\gamma^{a_i})/2}.
\end{equation}
To simplify notation, write $\rho_i = \floor{CZ_\tau(\gamma^{a_i})/2}$.

In principle the loops $\zeta_i$ might intersect themselves or each other, but we will see below that if $s$ is sufficiently large then they do not, so that their union is a braid $\zeta$. Its writhe is then given by
\begin{equation}
\label{eqn:totalwrithe}
w_\tau(\zeta) = \sum_{i=1}^k w_\tau(\zeta_i) + \sum_{i\neq j}\ell_\tau(\zeta_i,\zeta_j).
\end{equation}
Here $\ell_\tau(\zeta_i,\zeta_j)$ is the ``linking number'' of $\zeta_i$ and $\zeta_j$; this is defined like the writhe, except now we count crossings of $\zeta_i$ with $\zeta_j$ and divide by $2$. The terms on the right hand side of \eqref{eqn:totalwrithe} are bounded as follows:

\begin{lemma}
\label{lem:siefring}
If $s>>0$, then:
\begin{description}
\item{(a)}
Each component $\zeta_i$ is embedded and has writhe bounded by
\begin{equation}
\label{eqn:singlewrithe}
w_\tau(\zeta_i) \le \rho_i(a_i-1).
\end{equation}
\item{(b)}
If $i\neq j$, then $\zeta_i$ and $\zeta_j$ are disjoint, and
\[
\ell_\tau(\zeta_i,\zeta_j) \le \max(\rho_ia_j,\rho_ja_i).
\]
\end{description}
\end{lemma}

\begin{proof}
An analogous result was proved in \cite[\S6]{pfh2} in an analytically simpler situation. In the present case, parts of the argument require a result of Siefring \cite{siefring1} which generalizes Lemma~\ref{lem:hwz1} to give ``higher order'' asymptotics of holomorphic curves. We now outline how this works.

(a) If the integers $\op{wind}_\tau(\zeta_i)$ and $a_i$ are relatively prime, then an elementary argument in \cite[Lem.\ 6.7]{pfh2} related to Lemma~\ref{lem:hwz2}(a) shows that $\zeta_i$ is a torus braid, so that 
\begin{equation}
\label{eqn:iwb}
w_\tau(\zeta_i) = \op{wind}_\tau(\zeta_i)(a_i-1).
\end{equation}
The inequality \eqref{eqn:singlewrithe} now follows from this and the winding bound \eqref{eqn:windingbound}.
When $\op{wind}_\tau(\zeta_i)$ and $a_i$ have a common factor, one can prove that $\zeta_i$ is embedded and satisfies \eqref{eqn:iwb} using the analysis of Siefring.

(b) Let $\lambda_i$ and $\lambda_j$ denote the eigenvalues of the operators $L_{a_i}$ and $L_{a_j}$ associated to the ends $\zeta_i$ and $\zeta_j$ via Lemma~\ref{lem:hwz1}. If $\lambda_i<\lambda_j$, then it follows from Lemma~\ref{lem:hwz1} that the braid $\zeta_j$ is ``inside'' the braid $\zeta_i$ (assuming as always that we take $s$ sufficiently large), from which it follows that $\zeta_i$ and $\zeta_j$ do not intersect and
\begin{equation}
\label{eqn:ilb}
\ell_\tau(\zeta_i,\zeta_j) = \op{wind}_\tau(\zeta_i)a_j\le\rho_ia_j.
\end{equation}
The proof that $\zeta_i$ and $\zeta_j$ do not intersect and satisfy \eqref{eqn:ilb} when $\lambda_i=\lambda_j$ is more delicate and uses the analysis of Siefring.
\end{proof}

\begin{remark}
When $\rho_i$ and $a_i$ have a common factor one can strengthen the inequality \eqref{eqn:singlewrithe}; optimal bounds are given in \cite[\S3.1]{siefring2}. We do not need this strengthening when $\gamma$ is elliptic, but it is needed for the proof of the partition conditions when $\gamma$ is hyperbolic, see \cite[Lem.\ 4.16]{ir}.
\end{remark}

\begin{proof}[Proof of Lemma~\ref{lem:wb}.]
We will restrict attention to the most interesting case where $\gamma$ is elliptic with monodromy angle $\theta$. (See \cite[\S4]{ir} for the proof when $\gamma$ is hyperbolic.) By equation \eqref{eqn:CZell} we have $\rho_i=\floor{a_i\theta}$. So by equation \eqref{eqn:totalwrithe} and Lemma~\ref{lem:siefring}, it is enough to show that
\[
\sum_{i=1}^k\floor{a_i\theta}(a_i-1) + \sum_{i\neq j}\max(\floor{a_i\theta}a_j,\floor{a_j\theta}a_i) \le \sum_{i=1}^m(2\floor{i\theta}+1) - \sum_{i=1}^k(2\floor{a_i\theta}+1),
\]
with equality only if $(a_1,\ldots,a_k)=p_\theta^+(m)$. We can write the above inequality a bit more simply as
\begin{equation}
\label{eqn:workhorse}
\sum_{i, j=1}^n\max(\floor{a_i\theta}a_j,\floor{a_j\theta}a_i) \le 2\sum_{i=1}^m\floor{i\theta} - \sum_{i=1}^k\floor{a_i\theta}+m-k.
\end{equation}

To prove \eqref{eqn:workhorse}, following \cite{ir}, order the numbers $a_1,\ldots,a_k$ so that
\[
\frac{\floor{a_1\theta}}{a_1}\ge \frac{\floor{a_2\theta}}{a_2}\ge\cdots \ge\frac{\floor{a_k\theta}}{a_k}.
\]
Let $\Lambda$ be the path in the plane starting at $(0,0)$ with edge vectors $(a_1,\floor{a_1\theta}),\ldots,(a_k,\floor{a_k\theta})$ in that order. Let $P$ denote the region bounded by the path $\Lambda$, the horizontal line from $(0,0)$ to $(m,0)$, and the vertical line from $(m,0)$ to $\left(m,\sum_{i=1}^k\floor{a_i\theta}\right)$. Let $L$ denote the number of lattice points contained in $P$ (including the boundary), let $A$ denote the area of $P$, and let $B$ denote the number of lattice points on the boundary of $P$.

By dividing $P$ into rectangles and triangles, we find that the left hand side of \eqref{eqn:workhorse} is twice the area of $P$, i.e.\
\begin{equation}
\label{eqn:wh1}
2A = \sum_{i, j=1}^n\max(\floor{a_i\theta}a_j,\floor{a_j\theta}a_i).
\end{equation}
Counting by vertical strips, we find that the number of lattice points enclosed by $P$ is bounded by
\begin{equation}
\label{eqn:wh2}
L \le m+1+\sum_{i=1}^m\floor{i\theta},
\end{equation}
with equality if and only if the image of the path $\Lambda$ agrees with the image of the path $\Lambda_\theta^+(m)$ defined in \S\ref{sec:partitions}. In addition, the number of boundary lattice points satisfies
\begin{equation}
\label{eqn:wh3}
B \ge m + k + \sum_{i=1}^k\floor{a_i\theta},
\end{equation}
with equality if and only if none of the edge vectors of the path $\Lambda$ is divisible in $\Z^2$. Now Pick's formula for the area of a lattice polygon asserts that
\begin{equation}
\label{eqn:wh4}
2A = 2L - B - 2.
\end{equation}
Combining \eqref{eqn:wh1}, \eqref{eqn:wh2}, \eqref{eqn:wh3}, and \eqref{eqn:wh4}, we conclude that the inequality \eqref{eqn:workhorse} holds, with equality if and only if $\Lambda=\Lambda_\theta^+(m)$.
\end{proof}

\begin{remark}
Counts of lattice points in polygons have now arisen in two, seemingly independent, ways in our story: first in the above proof of the writhe bound and the partition conditions, and second in the calculation of the ECH of $T^3$ and the ECH capacities of toric domains in \S\ref{sec:t3} and \S\ref{sec:toric}. We do not know if there is a deep explanation for this.
\end{remark}

\subsection{Topological complexity of holomorphic curves}
\label{sec:J0}

The definitions of the ECH differential and the $U$ map do not directly specify the genus of the (nontrivial component of the) holomorphic currents to be counted. However this is determined indirectly by the relative homology class of the holomorphic current if one knows the trivial cylinder components, as we now explain. This fact is useful for understanding the holomorphic currents (see e.g.\ \cite[\S4.5]{wh} and \cite[Appendix]{lw} for applications), and also in the compactness argument in \S\ref{sec:differentialdefined} below.

Let $\alpha=\{(\alpha_i,m_i)\}$ and $\beta=\{(\beta_j,n_j)\}$ be orbit sets in the homology class $\Gamma$, and let $Z\in H_2(Y,\alpha,\beta)$.
Define\footnote{It is perhaps not optimal to denote this number by $J_0$ since $J$ usually denotes an almost complex structure. However the idea is that $J_0$ is similar to $I$ and so should be denoted by a nearby letter.}
\begin{equation}
\label{eqn:J0}
J_0(\alpha,\beta,Z) = -c_\tau(Z) + Q_\tau(Z) + CZ_\tau^{J_0}(\alpha,\beta),
\end{equation}
where
\begin{equation}
\label{eqn:CZJ}
CZ_\tau^{J_0}(\alpha,\beta) = \sum_i\sum_{k=1}^{m_i-1}CZ_\tau(\alpha_i^k) - \sum_j\sum_{k=1}^{n_j-1}CZ_\tau(\beta_j^k).
\end{equation}
The definition of $J_0$ is very similar to the definition of the ECH index in \eqref{eqn:I3} and \eqref{eqn:CZI}; however the sign of the relative first Chern class term is switched, and the Conley-Zehnder term is slightly different. More importantly, while $I$ bounds the {\em Fredholm index\/} of holomorphic curves, $J_0$ bounds the ``{\em topological complexity\/}'' of holomorphic curves.

To give a precise statement in the case that we will need to consider, let $\mc{C}\in\mc{M}(\alpha,\beta)$ be a holomorphic current. Suppose that $\mc{C}=\mc{C}_0\sqcup C$ where $\mc{C}_0$ is a union of trivial cylinders with multiplicities, and $C$ is somewhere injective. Let $n_i^+$ denote ``the number of positive ends of $\mc{C}$ at covers of $\alpha_i^+$'', namely the number of positive ends of $C$ at $\alpha_i^+$, plus $1$ if $\mc{C}_0$ includes the trivial cylinder $\R\times \alpha_i^+$ with some multiplicity. Likewise, let $n_j^-$ denote ``the number of negative ends of $\mc{C}$ at covers of $\beta_j^-$'', namely the number of negative ends of $C$ at $\beta_j^-$, plus $1$ if $\mc{C}_0$ includes the trivial cylinder $\R\times\beta_j^-$ with some multiplicity. Write $J_0(\mc{C})=J_0(\alpha,\beta,[\mc{C}])$.

\begin{proposition}
\label{prop:complexity}
Let $\alpha=\{(\alpha_i,m_i)\}$ and $\beta=\{(\beta_j,n_j)\}$ be generators of the ECH chain complex, and let $\mc{C}=\mc{C}_0\sqcup C\in\mc{M}(\alpha,\beta)$ be a holomorphic current as above. Then
\begin{equation}
\label{eqn:complexitybound}
-\chi(C) + \sum_i(n_i^+-1) + \sum_j(n_j^--1) \le J_0(\mc{C}).
\end{equation}
If $\mc{C}$ is counted by the ECH differential or the $U$ map, then equality holds in \eqref{eqn:complexitybound}.
\end{proposition}

 For example, it follows from \eqref{eqn:complexitybound} that
 $J_0(\mc{C})\ge -1$, with equality only if $C$ is a plane with positive end at a cover of some Reeb orbit $\gamma$, and $\mc{C}_0$ does not contain any trivial cylinders over $\gamma$.
Proposition~\ref{prop:complexity} is proved in \cite[Lem.\ 3.5]{wh}, using more general results from \cite[\S6]{ir}. 

\begin{exercise}
\label{ex:complexity}
(See answer in \S\ref{sec:answers}.)
Use the relative adjunction formula \eqref{eqn:adj3} and the partition conditions in \S\ref{sec:partitions} to prove the following special case of Proposition~\ref{prop:complexity}: If $C\in\mc{M}(\alpha,\beta)$ is an embedded holomorphic curve which is counted by the ECH differential or the $U$ map, then
\[
-\chi(C) + \sum_i(n_i^+-1) + \sum_j(n_j^--1) = J_0(C).
\]
\end{exercise}

\subsection{Proof that $\partial$ is well defined}
\label{sec:differentialdefined}

Assume now that the almost complex structure $J$ on $\R\times Y$ is generic. To complete the proof in \S\ref{sec:differential} that the ECH differential $\partial$ is well-defined, we need to prove the following:

\begin{lemma}
\label{lem:mabfinite}
If $\alpha$ and $\beta$ are orbit sets, then $\mc{M}_1(\alpha,\beta)/\R$ is finite.
\end{lemma}

To prove this we want to assume that $\mc{M}_1(\alpha,\beta)$ is infinite and apply a compactness argument to obtain a contradiction. A relevant version of Gromov compactness was proved by Bourgeois-Eliashberg-Hofer-Wysocki-Zehnder \cite{behwz}. To describe this result, say that a holomorphic curve $u$ is ``nontrivial'' if it is not a union of trivial cylinders; branched covers of trivial cylinders with a positive number of branched points are considered nontrivial. If $u_+$ and $u_-$ are nontrivial holomorphic curves, define ``gluing data'' between $u_+$ and $u_-$ to consist of the following:
\begin{itemize}
\item
A bijection between the negative ends of $u^+$ and the positive ends of $u^-$ such that ends paired up under the bijection are at the same (possibly multiply covered) Reeb orbit $\gamma$.
\item
When $\gamma$ above is multiply covered with degree $m$, then the negative end of $u^+$ and the positive end of $u^-$ each determine a degree $m$ cover of the underlying embedded Reeb orbit, and the gluing data includes an isomorphism of these covering spaces (there are $m$ possible choices for this).
\end{itemize}
Finally, define a {\em broken holomorphic curve\/} to be a finite sequence $(u^0,\ldots,u^k)$, where each $u_i$ is a nontrivial holomorphic curve in $\R\times Y$ modulo $\R$ translation, called a ``level'', together with gluing data as above between $u^{i-1}$ and $u^i$ for each $i=1,\ldots,k$. It is shown in \cite{behwz} that if $\{u_\nu\}_{\nu\ge 0}$ is a sequence of holomorphic curves with fixed genus between the same sets of Reeb orbits, then there is a subsequence which converges in an appropriate sense to a broken holomorphic curve.

Unfortunately we cannot directly apply the above result to a sequence of holomorphic currents in $\mc{M}_1(\alpha,\beta)/\R$, because we do not have an a priori bound on the genus of their nontrivial components. One can obtain a bound on the genus of a holomorphic curve from Proposition~\ref{prop:complexity}, but this bound depends on the relative homology class of the holomorphic curve. To get control over the relative homology classes of holomorphic currents in $\mc{M}_1(\alpha,\beta)/\R$, we will first use a second version of Gromov compactness which we now state.

If $\alpha$ and $\beta$ are orbit sets, define a {\em broken holomorphic current\/} from $\alpha$ to $\beta$ to be a finite sequence of nontrivial $J$-holomorphic currents $(\mc{C}^0,\ldots,\mc{C}^k)$ in $\R\times Y$, modulo $\R$ translation, for some $k\ge 0$ such that there are orbit sets $\alpha=\gamma^0,\gamma^1,\ldots,\gamma^{k+1}=\beta$ with $\mc{C}^{i}\in\mc{M}(\gamma^i,\gamma^{i+1})/\R$ for $i=0,\ldots,k$. Here ``nontrivial'' means not a union of trivial cylinders with multiplicities. Let $\overline{\mc{M}(\alpha,\beta)/\R}$ denote the set of broken holomorphic currents from $\alpha$ to $\beta$.

We say that a sequence of holomorphic currrents $\{\mc{C}_\nu\}_{\nu\ge 0}$ in $\mc{M}(\alpha,\beta)/\R$ {\em converges\/} to the broken holomorphic current $(\mc{C}^0,\ldots,\mc{C}^k)$ if for each $i=0,\ldots,k$ there are representatives $\mc{C}_{\nu}^i\in\mc{M}(\alpha,\beta)$ of the equivalence classes $\mc{C}_\nu\in\mc{M}(\alpha,\beta)/\R$ such that the sequence $\{\mc{C}_\nu^i\}_{\vu\ge0}$ converges as a current and as a point set on compact sets to $\mc{C}^i$, see \S\ref{sec:4dspecial}.

\begin{lemma}
\label{lem:9.8}
\begin{description}
\item{(a)}
Any sequence $\{\mc{C}_\nu\}_{\nu\ge 0}$ of holomorphic currents in $\mc{M}(\alpha,\beta)/\R$ has a subsequence which converges to a broken holomorphic current $(\mc{C}^0,\ldots,\mc{C}^k)\in\overline{\mc{M}(\alpha,\beta)/\R}$.
\item{(b)} If the sequence $\{\mc{C}_\nu\}_{\nu\ge 0}$ converges to $(\mc{C}^0,\ldots,\mc{C}^k)$, then
\[
\sum_{i=0}^k[\mc{C}^i]=[\mc{C}_\nu]\in H_2(Y,\alpha,\beta)
\]
for all $\nu$ sufficiently large.
\end{description}
\end{lemma}

\begin{proof}
(a) The proof has three steps.

{\em Step 1.\/}
For each $\nu$, suppose that $\mc{C}_\nu^*\in\mc{M}(\alpha,\beta)$ is a representative of the equivalence class $\mc{C}_\nu\in\mc{M}(\alpha,\beta)/\R$. We claim that $\{\mc{C}_\nu^*\}_{\nu\ge 0}$ has a subsequence which converges as a current and a point set on compact sets to some holomorphic current $\hat{\mc{C}}$ in $\R\times Y$.

To prove the claim, let $a<b$. We apply Gromov compactness via currents, see \S\ref{sec:4dspecial}, to the sequence of intersections $\mc{C}_\nu^*\cap([a,b]\times Y)$. To see why this is applicable, note that $[a,b]\times Y$ is equipped with the symplectic form $\omega=d(e^s\lambda)$ where $s$ denotes the $\R$ coordinate, and $J$ is $\omega$-compatible. Assume that $\mc{C}_\nu^*$ is transverse to $\{a\}\times Y$ and $\{b\}\times Y$, which we can arrange by perturbing $a$ and $b$. Then by Stokes's theorem,
\[
\int_{\mc{C}_\nu^*\cap((-\infty,a]\times Y)}e^ad\lambda + \int_{\mc{C}_\nu^*\cap([a,b]\times Y)}\omega + \int_{\mc{C}_\nu^*\cap([b,\infty)\times Y)}e^bd\lambda = e^b\mc{A}(\alpha) - e^a\mc{A}(\beta).
\]
The conditions on $J$ imply that $d\lambda$ is everywhere nonnegative on $\mc{C}_\nu^*$. Thus we obtain the a priori bound
\[
\int_{\mc{C}_\nu^*\cap([a,b]\times Y)}\omega \le e^b\mc{A}(\alpha).
\]
Gromov compactness via currents now implies that we can pass to a subsequence so that the sequence $\{\mc{C}_\nu^*\cap([a,b]\times Y)\}_{\nu\ge 0}$ converges as a current and as a point set to some holomorphic current in $[a,b]\times Y$. By diagonalizing, we can pass to a subsequence so that the sequence
$\{\mc{C}_\nu^*\}_{\nu\ge 0}$ converges as a current and as a point set on compact sets to some holomorphic current $\hat{\mc{C}}$ in $\R\times Y$.

Steps 2 and 3 are a fairly standard argument which we will just outline.  See e.g.\ \cite[Lem.\ 9.8]{pfh2} for details in a similar situation.

{\em Step 2.\/} By applying Step 1 to translates of $\hat{\mc{C}}$, one shows that $\hat{\mc{C}}\in\mc{M}(\gamma^+,\gamma^-)$, where $\gamma^+$ and $\gamma^-$ are orbit sets with $\mc{A}(\alpha)\ge \mc{A}(\gamma^+) \ge \mc{A}(\gamma^-)\ge\mc{A}(\beta)$.

{\em Step 3.\/} One can now choose representatives $\mc{C}_\nu^*\in\mc{M}(\alpha,\beta)$ of the equivalence classes $\mc{C}_\nu$ so that the intersection of each $\mc{C}_\nu^*$ with $\{0\}\times Y$ contains a point with distance at least $\epsilon$ from all Reeb orbits of action less than or equal to $\mc{A}(\alpha)$.
One then applies Steps 1 and 2 to this sequence $\mc{C}_\nu^*$. The limiting current $\hat{\mc{C}}$ must be nontrivial. If $\gamma^+=\alpha$ and $\gamma^-=\beta$, then we are done. Otherwise one applies the same argument to different chioices of $\mc{C}_\nu^*$ to find the other holomorphic currents $\mc{C}^i$ in the limiting broken holomorphic current.

(b) If this fails, then one uses arguments from the proof of part (a) to pass to a further subsequence which converges to a broken holomorphic current including $\mc{C}^0,\ldots,\mc{C}^k$ together with at least one additional level. But this is impossible by symplectic action considerations.
\end{proof}

We can now complete the proof that the differential $\partial$ is well-defined.

\begin{proof}[Proof of Lemma~\ref{lem:mabfinite}.]
Suppose to get a contradiction that there is an infinite sequence $\{\mc{C}_\nu\}_{\nu\ge 0}$ of distinct elements of $\mc{M}_1(\alpha,\beta)/\R$. 

For each $\nu$, by Proposition~\ref{prop:I03} we can write $\mc{C}_\nu=\mc{C}_{\nu,0}\sqcup C_{\nu,1}$, where $\mc{C}_{\nu,0}$ is a union of trivial cylinders with multiplicities, and $C_{\nu,1}$ is somewhere injective with $I(C_{\nu,1})=\ind(C_{\nu,1})=1$.
Since there are only finitely many possibilities for the trival part $\mc{C}_{\nu,0}$, we can pass to a subsequence so that $\mc{C}_{\nu,0}$ is the same for all $\nu$. There are then orbit sets $\alpha'$ and $\beta'$ which do not depend on $\nu$ such that $C_{\nu,1}\in\mc{M}_1(\alpha',\beta')$ for each $\nu$.

By Lemma~\ref{lem:9.8}, we can pass to a further subsequence such that the holomorphic curves $C_{\nu,1}$ all represent the same relative homology class $Z\in H_2(Y,\alpha',\beta')$.

By Proposition~\ref{prop:complexity}, there is a $\nu$-independent upper bound on the genus of $C_{\nu,1}$ in terms of $J_0(\alpha,\beta,Z)$. Thus we can pass to a further subsequence so that the holomorphic curves $C_{\nu,1}$ all have the same genus.

{\em Now\/} we can apply the compactness result of \cite{behwz} to pass to a further subsequence so that the sequence of holomorphic curves $\{C_{\nu,1}\}_{\nu\ge 0}$ converges in the sense of \cite{behwz} to a broken holomorphic curve $(u^0,\ldots,u^k)$.

By the Additivity property of the ECH index, see \S\ref{sec:ECHindex}, we have $\sum_{i=0}^kI(u^i)=1$.  By Proposition~\ref{prop:I03}, one of the curves $u_i$ has $I=1$, and the rest of the curves $u^i$ have $I=0$ and are unions of branched covers of trivial cylinders.

We will now be a bit sketchy for the rest of the proof. By a similar additivity property of the Fredholm index which follows from \eqref{eqn:ind3}, we also have $\sum_{i=0}^k\ind(u^i)=1$.
It then follows from Exercise~\ref{ex:partialorder} that in fact there is no level $u^i$ with $I(u^i)=0$. Hence the limiting broken holomorphic curve is a single holomorphic curve $u^0$, which is somewhere injective and has $ind(u^0)=1$.  Since $J$ is generic, $u^0$ is an isolated point in the moduli space of holomorphic curves modulo translation. But this contradicts the fact that $u^0$ is the limit of the sequence of distinct curves $\{C_{\nu,1}\}_{\nu\ge 0}$.
\end{proof}

\subsection{Proof that $\partial^2=0$}
\label{sec:200pages}

The proof that $\partial^2=0$ is much more subtle than the proof that $\partial$ is defined, for reasons which we now explain.

Fix a generic $J$. Let $\alpha_+$ and $\alpha_-$ be generators of the chain complex $ECC_*(Y,\lambda,\Gamma,J)$. We would like to show that the coefficient $\langle\partial^2\alpha_+,\alpha_-\rangle =0$. To do so, consider the moduli space of $I=2$ holomorphic currents $\mc{M}_2(\alpha_+,\alpha_-)/\R$.

\begin{lemma}
\label{lem:compactness2}
Any sequence $\{\mc{C}_\nu\}_{\nu\ge 0}$ of holomorphic currents in $\mc{M}_2(\alpha_+,\alpha_-)/\R$ has a subsequence which converges either to an element of $\mc{M}_2(\alpha_+,\alpha_-)/\R$, or to a broken holomorphic currrent $(\mc{C}^+,\mc{C}^-)\in\overline{\mc{M}_2(\alpha_+,\alpha_-)/\R}$ with $I(\mc{C}^+)=I(\mc{C}^-)=1$.
\end{lemma}

\begin{proof}
By Lemma~\ref{lem:9.8}, there is a subsequence which converges to a broken holomorphic current $(\mc{C}^0,\ldots,\mc{C}^k)$, where by definition each $\mc{C}^i$ is nontrivial. By the Additivity property of the ECH index, $\sum_{i=0}^kI(\mc{C}^i)=2$. By Proposition~\ref{prop:I03}, $I(\mc{C}^i)\ge 1$ for each $i$. The lemma follows from these two facts.
\end{proof}

The usual strategy now would be to add one point to each end of $\mc{M}_2(\alpha_+,\alpha_-)/\R$ to form a compact one-manifold with boundary, whose boundary points correspond to ends converging to broken holomorphic currents as above. In the present situation this is not quite correct; in fact we do not even know a priori that the moduli space $\mc{M}_2(\alpha_+,\alpha_-)/\R$ has only finitely many components\footnote{The compactness result of \cite{behwz} does not imply that $\mc{M}_2(\alpha_+,\alpha_-)/\R$ has only finitely many components, because of the failure of transversality of branched covers of trivial cylinders that can arise as levels in limits of sequences of $\ind=2$ holomorphic curves.}. Instead, one can {\em truncate\/} the moduli space $\mc{M}_2(\alpha_+,\alpha_-)$, i.e.\ remove holomorphic currents which are ``close to breaking'' in an appropriate sense, to obtain a compact one-manifold with boundary $\mc{M}_2'(\alpha_+,\alpha_-)/\R$. The boundary is equipped with a natural map
\begin{equation}
\label{eqn:truncate}
\partial\left(\frac{\mc{M}_2'(\alpha_+,\alpha_-)}{\R}\right) \longrightarrow \bigsqcup_{\alpha_0}\frac{\mc{M}_1(\alpha_+,\alpha_0)}{\R} \times \frac{\mc{M}_1(\alpha_0,\alpha_-)}{\R}
\end{equation}
which sends a boundary point to the broken holomorphic current that it is ``close to breaking into''.  The details of this truncation procedure are explained in \cite[\S1.3]{obg1}.

To complete the proof that $\langle\partial^2\alpha_+,\alpha_-\rangle=0$, we want to show that $\langle\partial^2\alpha_+,\alpha_-\rangle$ counts boundary points of $\mc{M}_2'(\alpha_+,\alpha_-)/\R$. For this purpose let $\alpha_0$ be an orbit set and let $(\mc{C}^+,\mc{C}^-)\in (\mc{M}_1(\alpha_+,\alpha_0)/\R)\times (\mc{M}_1(\alpha_0,\alpha_-)/\R)$. We then want to show the following:
\begin{description}
\item{(1)} If $\alpha_0$ is a generator of the chain complex $ECC_*(Y,\lambda,\Gamma,J)$, then $(\mc{C}^+,\mc{C}^-)$ has $1$ (mod $2$) inverse image under the map \eqref{eqn:truncate}.
\item{(2)} If $\alpha_0$ is not a generator of the chain complex $ECC_*(Y,\lambda,\Gamma,J)$, i.e.\ if $\alpha_0$ includes a hyperbolic Reeb orbit with multiplicity greater than one, then $(\mc{C}^+,\mc{C}^-)$ has $0$ (mod $2$) inverse images under the map \eqref{eqn:truncate}.
\end{description}

The standard picture from symplectic field theory is that if $(u^+,u^-)$ is a broken holomorphic curve such that $u^+$ and $u^-$ are regular and have $\ind(u^+)=\ind(u^-)=1$, then for each choice of gluing data between $u^+$ and $u^-$, see \S\ref{sec:differentialdefined}, one can ``glue'' $u^+$ and $u^-$ to obtain a unique end of the moduli space of index $2$ holomorphic curves.

To describe the proof of (1) and (2) above,
let us restrict attention to the case where $\alpha_0$ consists of a single pair $(\gamma,m)$ where $\gamma$ is an embedded Reeb orbit and $m\ge 1$.
Write $\mc{C}^\pm=\mc{C}^\pm_0\sqcup C^\pm_1$ where $\mc{C}^\pm_0$ is a union of trivial cylinders with multiplicities and $C^\pm_1$ is somewhere injective with $\ind(C^\pm_1)=I(C^\pm_1)=1$. To further simplify the discussion, let us also assume that there are no trivial cylinders involved, i.e.\ $\mc{C}^\pm_0=\emptyset$.

\paragraph{Gluing in the hyperbolic case.}

Suppose first that $\gamma$ is positive hyperbolic. In this case, the partition conditions from \S\ref{sec:partitions} tell us that $C^+_1$ has $m$ negative ends at $\gamma$, and $C^-_1$ has $m$ positive ends at $\gamma$. It follows that there are $m!$ choices of gluing data between $C^+_1$ and $C^-_1$, see \S\ref{sec:differentialdefined}. Hence SFT gluing implies that $C^+_1$ and $C^-_1$ can be glued to obtain $m!$ different ends of the moduli space of index $2$ curves. The number of gluings $m!$ is odd (namely $1$) when $m=1$ and even when $m>1$, which is exactly what we want in order to prove (1) and (2) above.

Suppose next that $\gamma$ is negative hyperbolic. Let $k=\floor{m/2}$. Then by the partition conditions in \S\ref{sec:partitions}, the curve $C^+_1$ (resp.\ $C^-_1$) has $k$ negative (resp.\ positive) ends at the double cover of $\gamma$, together with one negative (resp.\ positive) end at $\gamma$ when $m$ is odd. It follows that there are $2^kk!$ choices of gluing data between $C^+_1$ and $C^-_1$. Again, this is odd (namely $1$) when $m=1$ and even when $m>1$, as desired.

Although we are using $\Z/2$ coefficients here, we remark that in the proof that $\partial^2=0$ with $\Z$ coefficients, work of Bourgeois-Mohnke \cite{bm} implies that in the above cases when $m>1$, half of the gluings have one sign and half of the gluings have the other sign, so that the signed count of gluings is still zero.

\paragraph{Gluing in the elliptic case.}

Suppose now that $\gamma$ is elliptic. If $m=1$ then there is one gluing as usual. But if $m>1$, then it follows from Exercise~\ref{ex:partitions}(c) that $p_\gamma^+(m)$ and $p_\gamma^-(m)$ are disjoint, so the covering multiplicities of the negative ends of $C^+_1$ at covers of $\gamma$ are disjoint from the covering multiplicities of the positive ends of $C^-_1$ at covers of $\gamma$. Hence, there does not exist {\em any\/} gluing data between $C^+_1$ and $C^-_1$.
So how can we glue them?

It helps to think backwards from the process of breaking.
If a sequence of holomorphic currents in $\mc{M}_2(\alpha_+,\alpha_-)/\R$ coverges to the broken holomorphic current $(C^+_1,C^-_1)$, then as in the proof of Lemma~\ref{lem:mabfinite}, we can pass to a subsequence which converges in the sense of \cite{behwz} to a broken holomorphic curve $(u^0,\ldots,u^k)$, with $\sum_{i=0}^k\ind(u^i)=\sum_{i=0}^kI(u^i)=2$. Since $\sum_{i=0}^k\ind(u^i)=2$, Exercise~\ref{ex:partialorder} implies that $u^0=C^+_1$, $u^k=C^-_1$, and each $u^i$ with $0<i<k$ is a union of branched covers of trivial cylinders.

To reverse this process, let $u^+$ and $u^-$ be any irreducible somewhere injective holomorphic curves with $\ind=1$, but not necessarily with $I=1$. Suppose that $u^+$ has negative ends at covers of the embedded elliptic orbit $\gamma$ of multiplicities $a_1,\ldots,a_k$ with $\sum_{i=1}^ka_i=m$ and no other negative ends, and $u^-$ has positive ends at covers of $\gamma$ of multiplicities $b_1,\ldots,b_l$ with $\sum_{j=1}^lb_j=m$ and no other positive ends. We can try to glue $u^+$ and $u^-$ to an $\ind=2$ curve as follows. First, try to find an $\ind=0$ branched cover $u^0$ of $\R\times\gamma$ of degree $m$ with positive ends at covers of $\gamma$ with multiplicities $a_1,\ldots,a_k$ and negative ends at covers of $\gamma$ with multiplicities $b_1,\ldots,b_l$; see Exercise~\ref{ex:partialorder} for a discussion of when such a branched cover exists. Second, try to glue $u^+$, $u^0$, and $u^-$ to a holomorphic curve. There is an obstruction to gluing here because $u^0$ is not regular. However one can also vary $u^0$. The obstructions to gluing for various $u^0$ comprise a section of an ``obstruction bundle'' over the moduli space of all branched covers $u^0$. The (signed) number of ways to glue is then the (signed) number of zeroes of this section of the obstruction bundle. See \cite[\S1]{obg2} for an introduction to this analysis.

This signed count of gluings is denoted by $\#G(u^+,u^-)$ and computed in \cite[Thm.\ 1.13]{obg1}.
The result is that $\#G(u^+,u^-)=\pm c_\gamma(a_1,\ldots,a_k|b_1,\ldots,b_l)$, where $c_\gamma(a_1,\ldots,a_k|b_1,\ldots,b_l)$ is a nonnegative integer which depends only on (the monodromy angle of) $\gamma$ and the multiplicities $a_i$ and $b_j$. It turns out that $c_\gamma(a_1,\ldots,a_k|b_1,\ldots,b_l)=1$ if (and only if) $(a_1,\ldots,a_k)=p_\gamma^-(m)$ and $(b_1,\ldots,b_l)=p_\gamma^+(m)$, see \cite[Ex.\ 1.29]{obg1}. Thus, up to signs, the number of gluings is $1$ in the case needed to show that $\partial^2=0$ (and in no other case).


\subsection{Cobordism maps}
\label{sec:cobmap}

We now discuss what is involved in the construction of cobordism maps on ECH, as introduced in \S\ref{sec:addstr}.

\paragraph{Holomorphic curves in exact symplectic cobordisms}

We begin with the nicest kind of cobordism. Let $(Y_+,\lambda_+)$ and $(Y_-,\lambda_-)$ be nondegenerate contact three-manifolds, and let $(X,\omega)$ be an exact symplectic cobordism from $(Y_+,\lambda_+)$ to $(Y_-,\lambda_-)$. In this situation, one can define for each $L\in\R$ a cobordism map
\begin{equation}
\label{eqn:cobexact}
\Phi^L(X,\omega): ECH^L(Y_+,\lambda_+) \longrightarrow ECH^L(Y_-,\lambda_-)
\end{equation}
satisfying various axioms \cite[Thm.\ 1.9]{cc2}. Here
\[
ECH^L(Y,\lambda) = \bigoplus_{\Gamma \in H_1(Y)}ECH^L(Y,\lambda,\Gamma).
\]

The first step in the construction of the map \eqref{eqn:cobexact} is to ``complete'' the cobordism as follows. Let $\lambda$ be a primitive of $\omega$ on $X$ with $\lambda|_{Y_\pm}=\lambda_\pm$. If $\epsilon>0$ is sufficiently small, then there is a neighborhood $N_+$ of $Y_+$ in $X$, identified with $(-\epsilon,0]\times Y_+$, such that $\lambda = e^s\lambda_+$ where $s$ denotes the $(-\epsilon,0]$ coordinate. The neighborhood identification is the one for which $\partial/\partial s$ corresponds to the
unique vector field $\rho$ on $X$ with $\imath_\rho\omega=\lambda$. Likewise there is a neighborhood $N_-$ of $Y_-$ in $X$, identified with $[0,\epsilon)\times Y_-$, on which $\lambda = e^s\lambda_-$.
We now define the ``symplectization completion''
\[
\overline{X} = ((-\infty,0]\times Y_-)\cup_{Y_-} X \cup_{Y_+} ([0,\infty)\times Y_+),
\]
glued using the above neighborhood identifications.

Call an almost complex structure $J$ on $\overline{X}$ ``cobordism-admissible'' if it agrees with a symplectization-admissible almost complex structure $J_+$ for $\lambda_+$ on $[0,\infty)\times Y_+$, if it agrees with a symplectization-admissible almost complex structure $J_-$ for $\lambda_-$ on $(-\infty,0]\times Y_-$, and if it is $\omega$-compatible on $X$.

Given a cobordism-admissible almost complex structure $J$, one can consider $J$-holomorphic curves in $\overline{X}$ with positive ends at Reeb orbits in $Y_+$ and negative ends at Reeb orbits in $Y_-$, by a straightforward modification of the definition in the symplectization case in \S\ref{sec:ind}. If $J$ is generic, and if $C$ is a somewhere injective holomorphic curve as above, then the moduli space of holomorphic curves near $C$ is a manifold of dimension $\ind(C)$, where $\ind(C)$ is defined as in \eqref{eqn:ind3}, except that in the relative first Chern class term, the complex line bundle $\xi$ is replaced by $\det(TX)$.

Likewise, if $\alpha_\pm$ are orbit sets for $\lambda_\pm$, then there is a corresponding moduli space $\mc{M}(\alpha_+,\alpha_-)$ of $J$-holomorphic currents in $\overline{X}$.  One can define the ECH index $I$ of a holomorphic current in $\overline{X}$ as in \eqref{eqn:I3}, again replacing $\xi$ by $\det(TX)$ in the first Chern class term.
The index inequality \eqref{eqn:ii} then holds for somewhere injective holomorphic curves $C$ in $\overline{X}$,
by the same proof as in the symplectization case, see \cite[\S4]{ir}.
As in \S\ref{sec:differential}, let $\mc{M}_k(\alpha_+,\alpha_-)$ denote the set of holomorphic currents $\mc{C}\in\mc{M}(\alpha_+,\alpha_-)$ with ECH index $I(\mc{C})=k$.

We have the following important generalization of \eqref{eqn:daf}: If there exists $\mc{C}\in\mc{M}(\alpha_+,\alpha_-)$, then
\begin{equation}
\label{eqn:daf4}
\mc{A}(\alpha_+)\ge \mc{A}(\alpha_-).
\end{equation}
The reason is that by Stokes's theorem,
\begin{equation}
\label{eqn:Stokesexact}
{\mathcal A}(\alpha_+) - {\mathcal A}(\alpha_-) = \int_{\mc{C}\cap ([0,\infty)\times Y_+)}d\lambda_+ + \int_{\mc{C}\cap X}\omega + \int_{\mc{C} \cap ((-\infty,0]\times Y_-)}d\lambda_-,
\end{equation}
and the conditions on $J$ imply that each integrand is pointwise nonnegative on $\mc{C}$.

\paragraph{The trouble with multiple covers.}

One would now like to define a chain map
\[
\phi: ECC(Y_+,\lambda_+,J_+) \longrightarrow ECC(Y_-,\lambda_-,J_-)
\]
by declaring that if $\alpha_\pm$ are ECH generators for $\lambda_\pm$, then $\langle\phi\alpha_+,\alpha_-\rangle$ is the mod 2 count of $I=0$ holomorphic currents in $\mc{M}_0(\alpha_+,\alpha_-)$. The inequality \eqref{eqn:daf4} implies that only finitely many $\alpha_-$ could arise in $\phi\alpha_+$, and moreover we would get a chain map on the filtered complexes
\[
\phi^L:ECC^L(Y_+,\lambda_+,J_+) \longrightarrow ECC^L(Y_-,\lambda_-,J_-)
\]
for each $L>0$.

Unfortunately this does not work. The problem is that $\mc{M}_0(\alpha_+,\alpha_-)$ need not be finite, even if $J$ is generic. The compactness argument from \S\ref{sec:differentialdefined} does not carry over here, because the key Proposition~\ref{prop:I03} can fail in cobordisms. In particular, multiply covered holomorphic currents may have negative ECH index. We do know from \cite[Thm.\ 5.1]{ir} that the ECH index of a $d$-fold cover of a somewhere injective irreducible curve $C$ satisfies
\begin{equation}
\label{eqn:IdC}
I(dC) \ge dI(C) + \left(\frac{d^2-d}{2}\right) \left(2g(C)-2+\ind(C)+h(C)\right),
\end{equation}
where $g(C)$ denotes the genus of $C$, and $h(C)$ denotes the number of ends of $C$ at (positive or negative) hyperbolic orbits\footnote{Note that the magic number $2g(C)-2+\ind(C)+h(C)$ in \eqref{eqn:IdC} is similar to the normal Chern number in \eqref{eqn:ai}.}. If $J$ is generic then the index inequality implies that $I(C)\ge 0$; but $I(dC)<0$ is still possible when $2g(C)-2+\ind(C)+h(C)<0$.

To correctly define the coefficient $\langle\phi\alpha_+,\alpha_-\rangle$, one needs to take into account the entire ``compactification'' of $\mc{M}_0(\alpha_+,\alpha_-)$, namely the set of all broken holomorphic currents from $\alpha_+$ to $\alpha_-$ with total ECH index $0$. This moduli space may have many components of various dimensions, and each may make some contribution to the coefficient $\langle\phi\alpha_+,\alpha_-\rangle$. In fact, there is a simple example in which the coefficient $\langle\phi\alpha_+,\alpha_-\rangle$ must be nonzero, but there does not exist {\em any\/} $I=0$ holomorphic current from $\alpha_+$ to $\alpha_-$; rather, the contribution to $\langle\phi\alpha_+,\alpha_-\rangle$ comes from a broken holomorphic current with two levels, one of which is an $I=-1$ double cover. The example is the cobordism where $X=[0,1]\times Y$ which one obtains in trying to prove that ECH is unchanged under a period-doubling bifurcation. Even more interestingly, the orbit set in between the two levels is not a generator of the ECH chain complex, because it includes a doubly covered negative hyperbolic orbit.

Because of the above complications, it is a highly nontrivial, and currently unsolved problem, to define a chain map directly from the compactified moduli space of $I=0$ holomorphic currents.

\paragraph{Seiberg-Witten theory to the rescue.}

The definition of the cobordism map \eqref{eqn:cobexact} in \cite{cc2} instead counts solutions to the Seiberg-Witten equations, perturbed as in the proof of the isomorphism \eqref{eqn:echswf}.
The cobordism maps satisfy a ``Holomorphic Curves axiom'' which says among other things that for any cobordism-admissible $J$, the cobordism map is induced by a (noncanonical) chain map $\phi$ such that the coefficient $\langle\phi\alpha_+,\alpha_-\rangle\neq 0$ only if there exists a broken $J$-holomorphic current from $\alpha_+$ to $\alpha_-$. In particular, the coefficient $\langle\phi\alpha_+,\alpha_-\rangle\neq 0$ only if \eqref{eqn:daf4} holds, which is why the cobordism map preserves the symplectic action filtration.

\paragraph{The weakly exact case.}

If $(X,\omega)$ is only a weakly exact symplectic cobordism from $(Y_+,\lambda_+)$ to $(Y_-,\lambda_-)$, see \S\ref{sec:addstr}, then using Seiberg-Witten theory as above, one still gets a cobordism map
\[
\Phi^L(X,\omega): ECH^L(Y_+,\lambda_+,0) \longrightarrow ECH^L(Y_-,\lambda_-,0)
\]
which satisfies the Holomorphic Curves axiom. The reason why this map preserves the symplectic action filtrations is that a modification of the calculation in \eqref{eqn:Stokesexact} shows that in the weakly exact case, if there exists a holomorphic current $\mc{C}\in\mc{M}(\alpha_+,\alpha_-)$, and if moreover $[\alpha_\pm]=0\in H_1(Y_\pm)$, then the inequality \eqref{eqn:daf4} still holds, see \cite[Thm.\ 2.3]{qech}. It is this inequality which ultimately leads to all of the symplectic embedding obstructions coming from ECH capacities.

\section{Comparison of ECH with SFT}

To conclude, we now outline how ECH compares to the symplectic field theory (SFT) of Eliashberg-Givental-Hofer \cite{egh}. Although both theories are defined using the same ingredients, namely Reeb orbits and holomorphic curves, their features are quite different.

\paragraph{Dimensions.}

ECH is only defined for three-dimensional contact manifolds (and in some cases stable Hamiltonian structures) and certain four-dimensional symplectic cobordisms between them.
SFT is defined in all dimensions. It is an interesting question whether there exists an analogue of ECH in higher dimensions, and what that would mean.

\paragraph{Multiply covered Reeb orbits.}

In an ECH generator, we only care about the total multiplicity of each Reeb orbit. One can think of an ECH generator as a ``Reeb current''. In an SFT generator, one keeps track of individual covering multiplicities of Reeb orbits.
For example, if $\gamma_1$ is an elliptic Reeb orbit, and if $\gamma_k$ denotes the $k$-fold multiple cover of $\gamma_1$, then $\gamma_1^2$ and $\gamma_2$ are distinct SFT generators which correspond to the same ECH generator $\{(\gamma_1,2)\}$. Likewise, the SFT generators $\gamma_1^3$, $\gamma_2\gamma_1$ and $\gamma_3$ all correspond to the ECH generator $\{(\gamma_1,3)\}$.

\paragraph{Holomorphic curves.}

The full version of SFT counts all Fredholm index 1 holomorphic curves (after suitable perturbation to make the moduli spaces transverse). Other versions of SFT just count genus 0 Fredholm index 1 curves (rational SFT), or genus 0 Fredholm index 1 curves with one positive end (the contact homology algebra).

ECH counts holomorphic currents with ECH index 1, without explicitly specifying their genus (although the genus is more or less determined indirectly by the theory as explained in \S\ref{sec:J0}). These also have Fredholm index 1, although the way we are selecting a subset of the Fredholm index 1 curves to count (by setting the ECH index equal to 1) is very different from the way this is done in SFT (by setting the genus to 0, etc.).

\paragraph{Grading.}

SFT is relatively graded by the Fredholm index. 
ECH is relatively graded by the ECH index, and has an absolute grading by homotopy classes of oriented 2-plane fields.

\paragraph{Topological invariance.}

ECH depends only on the three-manifold, if one uses the absolute grading, as explained in Remark~\ref{rem:ag}.
SFT depends heavily on the contact structure; for example, the basic versions are trivial for overtwisted contact structures. On the other hand, ECH does contain the contact invariant (the homology class of the empty set of Reeb orbits) which can distinguish some contact structures, as explained in \S\ref{sec:addstr}. The ECH contact invariant is analogous to the unit in the contact homology algebra.

\paragraph{Disallowed Reeb orbits.}

In ECH, hyperbolic orbits cannot have multiplicity greater than 1.
In SFT, ``bad'' Reeb orbits are thrown out; in the three-dimensional case, a bad Reeb orbit is an even cover of a negative hyperbolic orbit. The reasons for discarding bad orbits in SFT are similar to the reasons for disallowing multiply covered hyperbolic orbits in ECH, see \S\ref{sec:gf} and \S\ref{sec:200pages}.

\paragraph{Keeping track of topological complexity.}

In SFT, there is a formal variable $\hbar$ which keeps track of the topological complexity of holomorphic curves; whenever one counts a curve with genus $g$ and $p$ positive ends, one multiplies by $\hbar^{p+g-1}$.
In ECH, topological complexity is measured by the number $J_0$ defined in \S\ref{sec:J0}. There is also a variant of $J_0$, denoted by $J_+$, which is closer to the exponent of $\hbar$, see \cite[\S6]{ir} and \cite[Appendix]{lw}.

\paragraph{U maps.}

ECH has a U map counting holomorphic curves passing through a base point, and also operations determined by elements of $H_1$ of the three-manifold, counting holomorphic curves intersecting a 1-cycle, see \S\ref{sec:Udetails}.
There are analogous structures on SFT, which can be more interesting for higher dimensional contact manifolds with lots of homology.

\paragraph{Algebra structure.}

SFT has some algebra structure (for example the contact homology algebra is an algebra). ECH does not. There is a natural way to ``multiply'' two ECH generators, by adding the multiplicities of all Reeb orbits in the two generators, but the differential and grading are not well behaved with respect to this ``multiplication''.

\paragraph{Legendrian knots.}

SFT defines invariants of Legendrian knots by counting holomorphic curves with boundary in $\R$ cross the Legendrian knot. No analogous construction in ECH is known, although one can define invariants of Legendrian knots using sutured ECH, see \cite[\S7.3]{sutured}.

\paragraph{Technical difficulties with multiply covered holomorphic curves.}

Both SFT and ECH have serious technical difficulties arising from multiply covered holomorphic curves of negative Fredholm index or ECH index.  In SFT, it is expected that the polyfold theory of Hofer-Wysocki-Zehnder \cite{polyfolds} will resolve these difficulties.  In ECH, we could manage these difficulties to prove that $\partial^2=0$ using holomorphic curves as outlined in \S\ref{sec:200pages}. Defining cobordism maps on ECH is harder, and it is not clear whether polyfolds will help, but fortunately one can define ECH cobordism maps using Seiberg-Witten theory, as described in \S\ref{sec:cobmap}.

\paragraph{Field theory structure.}

SFT can recover Gromov-Witten invariants of closed symplectic manifolds by cutting them into pieces along contact-type hypersurfaces.
ECH can similarly recover Taubes's Gromov invariant of closed symplectic four-manifolds \cite{field}.

\paragraph{Symplectic capacities.}

ECH can be used to define symplectic capacities. Other kinds of contact homology or SFT can also be used to define symplectic capacities, and this is an interesting topic for further research. For example, one can define an analogue of ECH capacities using linearized contact homology, and these turn out to agree with the Ekeland-Hofer capacities, at least for four-dimensional ellipsoids and polydisks, see Remark~\ref{rem:sharp}.

\appendix

\section{Answers and hints to selected exercises}
\label{sec:answers}

\begin{small}

\paragraph{\ref{ex:ellvol}.}
 We need to show that
\begin{equation}
\label{eqn:ellvol}
\lim_{k\to\infty}\frac{N(a,b)_k^2}{k} = 2ab.
\end{equation}
Given nonnegative integers $m$ and $n$, let $T(m,n)$ denote the triangle in the plane bounded by the $x$ and $y$ axes and the line $L$ through $(m,n)$ with slope $-b/a$. Then $N(a,b)_k=am+bn$ where $T(m,n)$ encloses $k+1$ lattice points (including the edges). When $k$ is large, the number of lattice points enclosed by $T(m,n)$ is the area of the triangle, plus an $O(k^{1/2})$ error. The line $L$ intersects the axes at the points $(a^{-1}N(a,b)_k,0)$ and $(0,b^{-1}N(a,b)_k)$, so its area is
\[
\op{area}(T(m,n))=\frac{N(a,b)_k^2}{2ab}.
\]
Thus
\[
k = \frac{N(a,b)_k^2}{2ab} + O(k^{1/2}).
\]
This implies \eqref{eqn:ellvol}.

\paragraph{\ref{ex:wesc}.}
It is enough to show that
\begin{equation}
\label{eqn:wescvol}
2\op{vol}(X,\omega) = \op{vol}(Y_+,\lambda_+) - \op{vol}(Y_-,\lambda_-),
\end{equation}
where $\op{vol}(Y,\lambda) = \int_Y\lambda \wedge d \lambda$. 
To prove \eqref{eqn:wescvol}, let $\lambda$ be a primitive of $\omega$ on $X$. Then by Stokes's theorem,
\[
2\op{vol}(X,\omega) = \int_{Y_+}\lambda\wedge d\lambda_+ - \int_{Y_-}\lambda\wedge d\lambda_-.
\]
Since $d\lambda=d\lambda_{\pm}$ on $Y_\pm$, by Stokes's theorem again we have
\[
\int_{Y_\pm}\lambda\wedge d\lambda_\pm = \op{vol}(Y_\pm,\lambda_\pm).
\]


\paragraph{\ref{ex:KerDC}.}
We use an infinitesimal analogue of the proof of Lemma~\ref{lem:S1inv}.
Let $\psi\in\Ker(D_C)$. Let $\epsilon>0$ be small and let $C'$ be the image of the map $C\to S^1\times Y_\phi$ sending $z\mapsto \exp_z(\epsilon\psi(z))$. Then
\[
\int_{C'}\omega = \epsilon^2\int_C\omega(\partial_s\psi,\nabla_t\psi)ds\,dt + O(\epsilon^3).
\]
Since $C'$ is homologous to $C$, we have $\int_{C'}\omega=0$, so
\begin{equation}
\label{eqn:KerDC1}
\int_C\omega(\partial_s\psi,\nabla_t\psi)ds\,dt = 0.
\end{equation}
On the other hand, since $\psi\in\Ker(D_C)$, we have $\nabla_t\psi = J\partial_s\psi$, so the integrand above is
\begin{equation}
\label{eqn:KerDC2}
\omega(\partial_s\psi,\nabla_t\psi) = \|\partial_s\psi\|^2,
\end{equation}
where $\|\cdot\|$ denotes the metric on $T^{\op{vert}}Y|_\gamma$ determined by $\omega$ and $J$. It follows from \eqref{eqn:KerDC1} and \eqref{eqn:KerDC2} that $\partial_s\psi\equiv 0$.

\paragraph{\ref{ex:ind3}.} Given a Reeb orbit $\gamma$, the set of homotopy classes of trivializations of $\xi|_\tau$ is an affine space over $\Z$. For an appropriate sign convention, shifting the trivalization over $\gamma_i^\pm$ by $1$ shifts $c_1$ by $\mp1$ and shifts $CZ_\tau(\gamma_i^\pm)$ by $2$.

\paragraph{\ref{ex:ind3}.} For an appropriate sign convention, shifting the trivialization $\tau$ over $\alpha_i$ by $1$ shifts $c_1$ by $-m_i$, shifts $Q_\tau$ by $m_i^2$, and shifts $w_\tau$ by $-m_i(m_i-1)$.

\paragraph{\ref{ex:ECHell}.} Let $T$ be the triangle in the plane which is bounded by the coordinate axes together with the line through $(m_1,m_2)$ with slope $-a/b$, cf.\ the answer to Ex.~\ref{ex:ellvol}. Then $\frac{1}{2}I(\alpha)$ can be interpreted as the number of lattice points in the triangle $T$ (including the boundary) minus $1$.

\paragraph{\ref{ex:partitions}.} (a) Since the path $\Lambda_\theta^+(m)$ starts at the origin and stays below the line $y=\theta x$, the initial edge has slope less than $\theta$. Since the path is the graph of a concave function, every subsequent edge also has slope less than $\theta$. Thus $b\le \floor{a\theta}$. If $b<\floor{a\theta}$ then there is a lattice point which is above the path $\Lambda_\theta^+(m)$ but below the line $y=\theta x$, contradicting the definition of $\Lambda_\theta^+(m)$.

(b) Since the total vertical displacement of the path $\Lambda_\theta^+(m)$ is $\floor{m\theta}$, it follows from part (a) that
\[
\sum_{i=1}^k\floor{q_i\theta}=\floor{\sum_{i=1}^kq_i\theta}.
\]
Since $\floor{x}+\floor{y}\le \floor{x+y}$ for any real numbers $x,y$, we have
\[
\begin{split}
\sum_{i\in I}\floor{q_i\theta} &\le \floor{\sum_{i\in I}q_i\theta},\\
\sum_{i\in\{1,\ldots,k\}\setminus I}\floor{q_i\theta} & \le \floor{\sum_{i\in\{1,\ldots,k\}\setminus I}q_i\theta}.
\end{split}
\]
Adding the above two inequalities and comparing with the previous equation, we see that both inequalities must be equalities.

(c) Suppose that such proper subsets $I,J$ exist. Let $m_1=\sum_{i\in I}q_i=\sum_{j\in J}r_j$ and let $m_2=m-m_1$. By part (b) applied to the subsets $I$, $\{1,\ldots,k\}\setminus I$, and $\{1,\ldots,k\}$, we have
\[
\floor{m_1\theta} + \floor{m_2\theta} = \floor{m\theta}.
\]
By the analogue of part (b) for $p_\theta^-(m)$, we have
\[
\ceil{m_1\theta} + \ceil{m_2\theta} = \ceil{m\theta}.
\]
Subtracting the above two equations gives $2=1$.

\paragraph{\ref{ex:partialorder}.} (a) Without loss of generality $C$ is connected. Let $a_1,\ldots,a_k$ denote the covering multiplicities of the positive ends of $u$, and let $b_1,\ldots,b_l$ denote the covering multiplicities of the negative ends of $u$. Let $g$ denote the genus of $C$. By the Fredholm index formula \eqref{eqn:ind3} and the Conley-Zehnder index formula \eqref{eqn:CZell}, we have
\[
\begin{split}
\ind(u) &= 2g-2+k+l+ \sum_{i=1}^k\left(2\floor{a_i\theta}+1\right) -\sum_{j=1}^l\left(2\floor{b_j\theta}+1\right)\\
&= 2\left(g-1+\sum_{i=1}^k\ceil{a_i\theta} - \sum_{j=1}^l\floor{b_j\theta}\right)\\
&\ge 2\left(g-1+\ceil{m\theta}-\floor{m\theta}\right).
\end{split}
\]
Since $\ceil{m\theta}-\floor{m\theta}=1$, it follows that $\ind(u)\ge 0$.

(b) We need to check: (i) if $p\ge q$ and $q\ge r$ then $p\ge r$, and (ii) if $p\ge q$ and $q\ge p$ then $p=q$.

Suppose $u_1$ is a branched cover with positive ends corresponding to $p$ and negative ends corresponding to $q$, and $u_2$ is a branched cover with positive ends corresponding to $q$ and negative ends corresponding to $r$. Gluing these together gives a branched cover $u_1\#u_2$ (defined up to sliding the branched points around) with positive ends corresponding to $p$ and negative ends corresponding to $r$. It follows immediately from the index formula \eqref{eqn:ind3} that $\ind(u_1\#u_2)=\ind(u_1)+\ind(u_2)$. So if $\ind(u_1)=\ind(u_2)=0$, then $\ind(u_1\#u_2)=0$ also, and this proves (i). Now suppose further that $p=r$. Then $q=r$, because otherwise $u_1\#u_2$ has at least two branch points, so its domain has $\chi\le -2$, so $\ind(u_1\#u_2)\ge 2$, a contradiction. This proves (ii).

(c) Let $u$ be a connected genus $0$ branched cover with positive ends corresponding to $p_\theta^-(m)$ and negative ends corresponding to $p_\theta^+(m)$.
Write $p_\theta^-(m)=(a_1,\ldots,a_k)$ and $p_\theta^+(m)=(b_1,\ldots,b_l)$. By the calculation in part (a), we have
\[
\ind(u) = 2\left(\sum_{i=1}^k\ceil{a_i\theta} - \sum_{j=1}^l\floor{b_j\theta} - 1\right).
\]
By Exercise~\ref{ex:partitions}(b) we have $\sum_{i=1}^k\ceil{a_i\theta}=\ceil{m\theta}$, and by symmetry $\sum_{j=1}^l\floor{b_j\theta}=\floor{m\theta}$. Hence $\ind(u)=0$.

(d) Suppose there exists a partition $q$ with $p_\theta^+(m)>q$.
Write $p_\theta^+(m)=(a_1,\ldots,a_k)$ and $q=(b_1,\ldots,b_l)$.
By Exercise~\ref{ex:partitions}(b) we have $\sum_{i=1}^k\floor{a_i\theta}=\floor{m\theta}$.  By the calculation in part (a) above we have $\sum_{i=1}^k\ceil{a_i\theta}=\ceil{m\theta}$. These two equations imply that $k=1$. Thus the path $\Lambda_\theta^+(m)$ is just the line segment from $(0,0)$ to $(m,\floor{m\theta})$.

 Now the calculation in part (a) above also implies that $\sum_{j=1}^l\floor{b_j\theta}=\floor{m\theta}$. 
But this is impossible. To see why, order the numbers $b_j$ so that $\floor{b_j\theta}/b_j\ge \floor{b_{j+1}\theta}/b_{j+1}$. Let $\Lambda'$ be the path in the plane that starts at $(0,0)$ and whose edge vectors are the segments $(b_j,\floor{b_j\theta})$ in order of increasing $j$. Since $(b_1,\ldots,b_l)\neq (m)$ and since there are no lattice points above the path $\Lambda_\theta^+(m)$ and below the line $y=\theta x$, it follows that the path $\Lambda'$ is below the path $\Lambda_\theta^+(m)$, with the two paths intersecting only at $(0,0)$. Hence the right endpoint of $\Lambda'$ is below the right endpoint of $\Lambda_\theta^+(m)$, which means that $\sum_j\floor{b_j\theta} < \floor{m\theta}$.

By symmetry, there also does not exist a partition $q$ with $q > p_\theta^-(m)$.

\paragraph{\ref{ex:Uell1}.} 
By Exercise~\ref{ex:ellcq} we have $c_\tau(C_2)=1$. Since $\ind(C_2)=2$, it follows from \eqref{eqn:ind3} that
\[
\chi(C_2) = CZ_\tau^{\op{ind}}(C_2).
\]
If $\epsilon$ is sufficiently small with respect to $i$, then $CZ_\tau(\gamma_1^i)=2i-1$ when $i>0$, and $CZ_\tau(\gamma_2^{i-1})=2i-1$ when $i>1$. It follows that $CZ_\tau^{\op{ind}}(C_2)=0$ when $i>1$, and $CZ_\tau^{\op{ind}}(C_2)=1$ when $i=1$.

\paragraph{\ref{ex:Uell2}.} 
Without loss of generality, $\mc{C}_0=\emptyset$.
We then compute that
\[
CZ_\tau^{\op{ind}}(C_2) = \left\{\begin{array}{cl} i+j-1, & i>0,j>1\\
i+1, & i>0, j=1,\\
j, & i=0, j>1,\\
2, & i=0, j=1.
\end{array}
\right.
\]
On the other hand, letting $g$ denote the genus of $C_2$, we have
\[
\chi(C_2) = \left\{\begin{array}{cl} -2g-i-j-1, & i>0, j>1,\\
-2g-i-1, & i>0, j=1,\\
-2g-j, & i=0, j>1,\\
-2g, & i=0, j=1.
\end{array}
\right.
\]
Since $c_\tau(C_2)=0$ (by Exercise~\ref{ex:ellcq}) and $\ind(C_2)=2$, it follows from \eqref{eqn:ind3} that
\[
\chi(C_2) = CZ_\tau^{\op{ind}}(C_2)-2.
\]
Combining the above three equations, we find that if $i>0$ or $j>1$ then $g<0$, which is a contradiction. Thus $i=0$ and $j=1$, and combining the above three equations again we find that $g=0$.

\paragraph{\ref{ex:t3action}.}
Otherwise $g=1$. Then equation \eqref{eqn:3dl} (together with the fact that $C_1$ has at least one positive end) implies that $C_1$ has exactly one positive end at some hyperbolic orbit $h_{a,b}$, and all negative ends of $C_1$ are elliptic. Let $(a_1,b_1),\ldots,(a_k,b_k)$ denote the vectors corresponding to the negative ends. The action of $h_{a,b}$ is slightly less than $\sqrt{a^2+b^2}$, and the sum of the symplectic actions of the negative ends is slightly greater $\sum_{i=1}^k\sqrt{a_i^2+b_i^2}$.  Since the differential decreases symplectic action,
\[
\sum_{i=1}^k\sqrt{a_i^2+b_i^2} < \sqrt{a^2+b^2}.
\]
But this contradicts the triangle inequality, since $\sum_{i=1}^k(a_i,b_i)=(a,b)$, since $h_{a,b}$ is homologous in $T^3$ to $\sum_{i=1}^k e_{a_i,b_i}$.

\paragraph{\ref{ex:dualloop}.}
Let $\Lambda$ be any polygonal path with edge vectors
$v_1,\ldots,v_k$. Then
\[
\ell_\Omega(\Lambda) = \sum_{i=1}^k\langle v_i,w_i\rangle
\]
where $w_i\in\partial\Omega'$ is a point at which an outward normal
vector to $\Omega'$ is a positive multiple of $v_i$. (When $w_i$ is a
corner of $\partial\Omega'$, ``an outward normal vector'' means a
vector whose direction is between the directions of the limits of the
normal vectors on either side of $w_i$.) If we replace $\Omega'$ by its translate by some vector $\eta$, then the above formula is replaced by
\[
\ell_\Omega(\Lambda) = \sum_{i=1}^k\langle v_i,w_i+\eta\rangle.
\]
If $\Lambda$ is a loop, then the two formulas for $\ell_\Omega(\Lambda)$ agree since
$\sum_iv_i=0$.

\paragraph{\ref{ex:complexity}.}
By the relative adjunction formula \eqref{eqn:adj3}, and since equality holds in the writhe bound \eqref{eqn:writhebound}, we have
\[
-\chi(C) = -c_\tau(C) + Q_\tau(C) + CZ_\tau^I(C) - CZ_\tau^{\op{ind}}(C).
\]
So by the definition of $J_0$ in \eqref{eqn:J0} and \eqref{eqn:CZJ}, we need to show that
\[
\sum_i(n_i^+-1) + \sum_j(n_j^--1) = CZ_\tau^{J_0}(C) - CZ_\tau^I(C) + CZ_\tau^{\op{ind}}(C).
\]
This equation can be proved one Reeb orbit at a time. Namely, it is enough to show that for each $i$, if $C$ has positive ends at covers of $\alpha_i$ with multiplicities $q_1,\ldots,q_{n_i^+}$ where $\sum_{k=1}^{n_i^+}q_k=m_i$, then 
\begin{equation}
\label{eqn:onetime}
n_i^+-1 = -CZ_\tau(\alpha_i^{m_i}) + \sum_{k=1}^{n_i^+}CZ_\tau(\alpha^{q_k}),
\end{equation}
and an analogous equation for each Reeb orbit $\beta_j$.

To prove \eqref{eqn:onetime}, first note that if $\alpha_i$ is hyperbolic, then $n_i^+=m_i=1$ and the equation is trivial. Suppose now that $\alpha_i$ is elliptic with rotation angle $\theta$ with respect to $\tau$. Then \eqref{eqn:onetime} becomes
\[
0 = -2\floor{m_i\theta} + \sum_{k=1}^{n_i^+}2\floor{q_k\theta}.
\]
This last equation holds by the partition conditions and Exercise~\ref{ex:partitions}(b).

\end{small}

\end{document}